\documentclass[12pt,a4paper]{amsart}

\usepackage{amssymb}
\usepackage{fullpage}
\usepackage{amsfonts}
\usepackage{dsfont}
\usepackage{amsmath}
\usepackage{stmaryrd}
\usepackage{hyperref}
\newcommand{\abs}[1]{\ensuremath{\lvert#1\rvert}}
\newcommand{\norm}[1]{\ensuremath{\lVert#1\rVert}}
\newcommand{\floor}[1]{\lfloor #1 \rfloor}

\newcommand{\eps}{\epsilon}

\newcommand{\C}{\ensuremath{\mathbb{C}}}

\newcommand{\R}{\ensuremath{\mathbb{R}}}
\newcommand{\Z}{\ensuremath{\mathbb{Z}}}

\newtheorem{theorem}{Theorem}[section]
\newtheorem{lemma}[theorem]{Lemma}
\newtheorem{proposition}[theorem]{Proposition}
\newtheorem{corollary}[theorem]{Corollary}

\theoremstyle{definition}
\newtheorem{definition}[theorem]{Definition}
\newtheorem*{question}{Question}

\theoremstyle{remark}
\newtheorem{remark}[theorem]{Remark}

\DeclareMathOperator{\tr}{Tr}
\DeclareMathOperator{\rank}{rank}
\DeclareMathOperator{\im}{im}
\DeclareMathOperator{\spn}{span}
\DeclareMathOperator{\diag}{diag}

\title{Limit operators for circular ensembles}
\author{Kenneth Maples}
\address{Institut f\"ur Mathematik, Universit\"at Z\"urich, Winterthurerstrasse 190, CH-8057 Z\"urich, Switzerland.}
\email{kenneth.maples@math.uzh.ch}

\author{Joseph Najnudel}
\address{Institut de Math\'ematiques de Toulouse, Universit\'e Paul Sabatier, 118 route de Narbonne F-31062 Toulouse Cedex 9, France.}
\email{joseph.najnudel@math.univ-toulouse.fr}

\author{Ashkan Nikeghbali}
\address{Institut f\"ur Mathematik, Universit\"at Z\"urich, Winterthurerstrasse 190, CH-8057 Z\"urich, Switzerland.}
\email{ashkan.nikeghbali@math.uzh.ch}

\date{\today}

\begin{document}
    
\begin{abstract}
It is known that a unitary matrix can be decomposed into a product of reflections, one for each dimension, and that the Haar measure on the unitary group pushes forward to independent uniform measures on the reflections. We consider the sequence of unitary matrices given by successive products of random reflections.

In this coupling, we show that powers of the sequence of matrices converge in a suitable sense to a flow of operators which acts on a random vector space. The vector space has an explicit description as a subspace of the space of sequences of complex numbers. The eigenvalues of the matrices converge almost surely to the eigenvalues of the flow, which are distributed in law according to a sine-kernel point process. The eigenvectors of the matrices converge almost surely to vectors which are distributed in law as Gaussian random fields on a countable set.

This flow gives the first example of a random operator with a spectrum distributed according to a sine-kernel point process which is naturally constructed from finite dimensional random matrix theory.
\end{abstract}

\maketitle
\hrule\tableofcontents
\hrule

\clearpage
\section*{Notation} \label{sec:notation}

If $v \in \C^n$ is a vector, then we write $v[m]$ for the image of $v$ under the canonical projection map $\C^n \to \C^m$ onto the first $m$ standard basis vectors and we write $(v)_k$ for its k-th coordinate.

We write $O(n)$ for the orthogonal group of dimension $n$, i.e.~the group of invertible operators on $\R^n$ which preserve the standard real inner product. We write $U(n)$ for the unitary group of dimension $n$ which preserves the standard complex inner product. For any vector space $V$, real or complex, we write $GL(V)$ for the group of invertible transformations. We always write $1$ for the identity operator in every space.

We write $\mathbb{U} = U(1)$ for the unit circle in $\C$, i.e. those complex numbers with modulus~$1$.

Calligraphic characters denote $\sigma$-algebras, i.e.~$\mathcal{A}$, $\mathcal{B}$, $\mathcal{C}$, etc. If $\mathcal{A}$ and $\mathcal{B}$ are $\sigma$-algebras on a common set, then $\mathcal{A} \vee \mathcal{B}$ denotes the smallest $\sigma$-algebra containing both $\mathcal{A}$ and $\mathcal{B}$.

We also write $a \vee b$, for $a, b \geq 0$, to mean $\max(a,b)$ and $a \wedge b$ to mean $\min(a,b)$.

We employ asymptotic notation for inequalities where precise constants are not important. In particular, we write $X = O(Y)$ to mean that there exists a constant $C > 0$ such that $\abs{X} \leq C Y$. All constants are assumed to be absolute unless indicated otherwise by an appropriate subscript; e.g.~we write $f_n(x) = O_n(g_n(x))$ to mean that there are constants $C_n$, one for each value of $n$, such that $\abs{f_n(x)} \leq C_n g_n(x)$ for all $x$. We also use (modified) Vinogradov notation, where $X \lesssim Y$ means $X = O(Y)$, for convenience.

If $t$ is a real number, we write $\lfloor t \rfloor$ its integer part.

If $H$ is a Hilbert space with scalar product $\langle ., . \rangle$, and if $F\subset H$, then $F^\perp=\{x\in H;\;\; \langle x, y \rangle=0\;\forall y\in F \}$. If $H$ is a complex Hilbert space, then we will always use the scalar product which is linear in the first variable and conjugate linear in the second, i.e.~$\langle ax, by \rangle = a \overline{b} \langle x, y \rangle$.

\section{Introduction}

It has been observed that for many models of random matrices, the eigenvalues have a limiting short-scale behavior when the dimension goes to infinity which depends on the global symmetries of the model, but not on its detailed features. For example, the Gaussian Orthogonal Ensemble (GOE), for which the matrices are real symmetric with independent gaussian entries on and above the diagonal, corresponds to a limiting short-scale behavior for the eigenvalues that is also obtained for several other models of random real symmetric matrices. Similarly, the limiting spectral behavior of a large class of random hermitian and unitary ensembles, including the Gaussian Unitary Ensemble (GUE, with independent, complex gaussians above the diagonal), and the Circular Unitary Ensemble (CUE, corresponding to the Haar measure on the unitary group of a given dimension), involves a remarkable random point process, called the {\it determinantal sine-kernel process}. It is a point process for which the $k$-point correlation function is given by 
\[
  \rho_k (x_1, \ldots, x_k) = \det \left( \frac{\sin( \pi(x_p - x_q))}{\pi(x_p - x_q)} \right)_{1 \leq p, q \leq k}.
\]
From an observation of Montgomery, it has been conjectured that the limiting short-scale behavior of the imaginary parts of the zeros of the Riemann zeta function is also described by a determinantal sine-kernel process. This similar behavior supports the conjecture of Hilbert and P\'olya, who suggested that the non-trivial zeros of the Riemann zeta functions should be interpreted as the spectrum of an operator $\frac12 + iH$ with $H$ an unbounded Hermitian operator.
  
In order to understand the kind of randomness which would be involved in such an operator, it is natural to try to construct a random version $H_0$ of it, for which the spectrum has the conjectured limiting behavior of the zeros of the Riemann zeta function, i.e.~is a determinantal sine-kernel point process. Since the spectrum of $H_0$ should also correspond to the limiting behavior of many ensembles of hermitian and unitary matrices, it is natural to expect that these ensembles can, in one way and another, be related to $H_0$. 

Instead of looking directly for $H_0$, one can also directly seek the flow of linear operators $(U_0^{\alpha})_{\alpha \in \mathbb{R}} := (e^{i \alpha H_0})_{\alpha \in \R}$ generated by exponentiation. 
 This point of view is both consistent with what would be a possible interpretation in quantum mechanics ($U_0^{\alpha}$ playing the role of the operator of evolution at time $\alpha$, whereas $H_0$ corresponds to the Hamiltonian), and with the number theoretic point of view: if the Hilbert-P\'olya operator $H$ would exist, the Chebyshev function $\psi_0$, which is defined as 
\[
\psi_0(x) = \sum_{p^m < x} \log p + \frac12 \begin{cases} \log p, & x = p^m \text{ for some prime power } p^m \\ 0, & \text{otherwise} \end{cases}
\]
where the summation is through all powers of primes bounded by $x$, would formally satisfy
\[
\psi_0(e^x) = \int_{-\infty}^x \left( e^y - e^{y/2} \tr(e^{iHy}) \right) dy + O(1),
\]
which suggests an important role played by the conjectural flow $(e^{iHy})_{y \in \mathbb{R}}$ of unitary operators. To see this, recall the von Mangoldt formula (see e.g.~\cite{titchmarsh}):
\[
\psi_0(x)=x-\sum_{\rho}\frac{x^\rho}{\rho}+O(1),
\]
where the summation is over all zeros $\rho$ of the Riemann zeta function.  Now taking $x=e^y$ in the above,  writing $e^y-\sum_{\rho}\dfrac{e^{y\rho}}{\rho}$ as the integral of its derivative and then writing the last sum as the trace of the operator yields the formal identity above.
On the other hand, as we will see below, considering the flow 
$(e^{i \alpha H_0})_{\alpha \in \mathbb{R}}$, instead of taking directly $H_0$, is also consistent with the main construction of the present article, since, in a sense which will be made precise, this flow is an approximation, for large $n$, of the successive powers of 
a random matrix following the Haar measure on $U(n)$. 

It should be mentioned that the problem of the existence of the operator $H_0$ is also hinted at in the work by Katz and Sarnak \cite{katzsarnak}. For instance, the zeta function of algebraic curves are related to the unitary group or some other compact groups (e.g. the symplectic group) and there the question of coupling all different dimensions of the unitary group together in a consistent way in order to prove strong limit theorems (i.e. almost sure convergence) and having an infinite dimension space and operator sitting above is raised. 

The main goal of the present paper is the construction of a flow of random operators $(V^{\alpha})_{\alpha \in \mathbb{R}}$, whose spectrum, in a sense which can be made precise, is a determinantal sine-kernel process, and which is directly related to the Circular Unitary Ensemble. A similar, but simpler, construction has been made in \cite{NN13}, in which we consider permutation matrices instead of unitary matrices. 
The construction which is made in the present paper uses the recent construction of virtual isometries given in \cite{BNN12}  and  needs several further  steps.

First, the space on which $(V^{\alpha})_{\alpha \in \mathbb{R}}$ acts must be infinite-dimensional. In order to relate this space with the Circular Unitary Ensemble, we will construct a coupling between the matrix models of each dimension. That is, we define a sequence $(u_n)_{n \geq 1}$ of random unitary matrices, in such a way that for all $n \geq 1$, each matrix $u_n \in U(n)$ is Haar distributed on the unitary group. Of course, there are many ways to couple the random variables $u_n$ to each other: for example, by taking all the matrices to be independent. However, in order to have sufficient consistency to construct limiting objects, we need to be more subtle.

In fact, we will couple $(u_n)_{n \geq 1}$ in such a way that almost surely, this sequence of random matrices is a \emph{virtual isometry}. The notion of virtual isometry has been introduced in one of our former articles \cite{BNN12}, generalizing both the notion of {\it virtual permutation} studied by Kerov, Olshanski, Vershik \cite{Kerov},
and the previous notion of virtual unitary group introduced by Neretin \cite{Ner}.
By definition, a virtual isometry is a sequence $(u_n)_{n \geq 1}$ of unitary matrices $u_n$ of dimension $n$, such that for all $n$, $u_n$ is the matrix $u$ such that, for $u_{n+1}$ fixed, the rank of the difference
\[
  \begin{pmatrix} u & 0 \\ 0 & 1 \end{pmatrix} - u_{n+1}
\]
is minimal (which, as we will see, is always $0$ or $1$). From this definition one can deduce that for all $n \geq 1$, $u_n$ completely determines the sequence $u_1, u_2, \dots, u_{n-1}$, and from these matrices, one directly obtains a decomposition of $u_n$ as a product of complex reflections. Moreover, the Haar measures in different dimensions are compatible with respect to the notion of virtual isometries, i.e.~it is possible to construct a probability distribution on the space of  virtual isometries in such a way that for all $n \geq 1$, the marginal distribution of the $n$-dimensional component coincides with the Haar measure on $U(n)$. Therefore, the Circular Unitary Ensemble can be coupled in all dimensions by considering a random virtual isometry following a suitable probability distribution.

Note that in the notion of virtual isometry defined here, the vectors of the canonical basis of $\mathbb{C}^n$ play a particular role. One could attempt to generalize the notion of virtual isometries by considering sequences of unitary operators on $E_n$, $n \geq 1$, where $(E_n)_{n \geq 1}$ is a sequence of complex inner product spaces, $E_n$ being of dimension $n$. However, this reduces to the particular case $E_n = \mathbb{C}^n$ by a change of basis and so we have chosen to use the standard basis for simplicity.


In \cite{BNN12}, it is shown that if $(u_n)_{n \geq 1}$ follows this distribution, then for all $k$, the $k$th positive (resp.~negative) eigenangle of $u_n$, multiplied by $n/2 \pi$ (i.e. the inverse of the average spacing between eigenangles for any matrix in $U(n)$), converges almost surely to a random variable $y_k$ (resp.~$y_{1-k}$). The random set $(y_k)_{k \in \mathbb{Z}}$ is a determinantal sine-kernel process, and for each $k$, the convergence holds with a rate dominated by some negative power of $n$. In the present paper, we improve our estimate of this rate, and more importantly, we prove that almost sure convergence  not only holds for the eigenangles of $u_n$, but also for the components of the corresponding eigenvectors. More precisely, we show that, for all $k, \ell \geq 1$, the $\ell$th component of the eigenvector of $u_n$ associated to the $k$th positive (resp.~negative) eigenangle converges almost surely to a non-zero limit $t_{k, \ell}$ (resp.~$t_{1-k, \ell}$) when $n$ goes to infinity, if the norm of the eigenvector is taken equal to $\sqrt{n}$ and if the phases are suitably chosen. Moreover, the variables $(t_{k, \ell})_{k \in \mathbb{Z}, \ell \geq 1}$ are iid complex gaussians. Note that taking
a norm equal to $\sqrt{n}$ is natural in this setting: with this normalization, the expectation of the squared modulus of each coordinate of a given eigenvector of $u_n$ is equal to $1$, so we can expect a convergence to a non-trivial limit. If the norm of the eigenvectors is taken equal to $1$ instead of $\sqrt{n}$, then the coordinates  converge to zero when $n$ goes to infinity. 

Knowing the joint convergence of the renormalized eigenvalues and the corresponding eigenvectors, it is natural to expect that, in a sense which has to be made precise, the limiting behavior of $(u_n)_{n \geq 1}$ when $n$ goes to infinity can be described by  an operator whose eigenvectors are the sequences $(t_{k, \ell})_{\ell \geq 1}$, $k \in \mathbb{Z}$ and for which the corresponding eigenvalues are $(y_k)_{k \in \mathbb{Z}}$. 

The precise statement of our main result can in fact be described as follows:
\begin{theorem} \label{thm:main}
Almost surely there exists a random vector subspace $\mathcal{F}$ of $\mathbb{C}^{\infty}$ and a flow of linear maps $(V^{\alpha})_{\alpha \in \mathbb{R}}$ on $\mathcal{F}$, such that $V^{\alpha + \beta}= V^{\alpha} V^{\beta}$, and satisfying the following properties, for any sequence $w = (w_{\ell})_{\ell \geq 1}$ in $\mathcal{F}$:
\begin{enumerate}
\item For any fixed $\ell \geq 1$, the $\ell$th component of the $n$-dimensional vector of $u_n^{\lfloor \alpha n \rfloor} (w_1,\dots,w_n)$ tends to the $\ell$th component of $V^{\alpha} (w)$ when $n$ goes to infinity. 
\item If $w \neq 0$, the $L^2$ distance between $u_n^{\lfloor \alpha n \rfloor} (w_1,\dots,w_n)$ and 
$((V^{\alpha} w)_1,\dots, (V^{\alpha} w)_n)$ is negligible with respect to the norm of $(w_1,\dots,w_n)$.
\end{enumerate}
Moreover, the eigenvectors of the flow, i.e.~the sequences $w \in \mathcal{F}$ such that there exists $\lambda \in \mathbb{R}$ for which $V^{\alpha} w = e^{2 i \pi \lambda \alpha} w$ for all $\alpha \in \mathbb{R}$, are exactly the sequences proportional to $(t_{k, \ell})_{\ell \geq 1}$ for some $k \in \mathbb{Z}$. The corresponding value of $\lambda$ is equal to $y_k$. 
\end{theorem}
Let us emphasize that the last part of this theorem gives the complete family of eigenvectors of the flow $(V^{\alpha})_{\alpha \in \mathbb{R}}$. In particular, the eigenvalues form exactly 
the determinantal sine-kernel process $(y_k)_{k \in \Z}$, with no extra eigenvalue. 
Note that the notion of eigenvalue and eigenvector is not exactly the same as in the usual situation where a single operator is considered. 

Intuitively, the operator $V^{\alpha}$ can be viewed as a limit, for sequences in the space $\mathcal{F}$, of the iteration $u_n^{\lfloor \alpha n\rfloor}$, when $n$ goes to infinity. The flow $(V^{\alpha})_{\alpha \in \mathbb{R}}$ can be compared with the flow constructed in \cite{NN13} for permutation matrices. The space $\mathcal{F}$ is random and explicitly defined in terms of the virtual isometry $(u_n)_{n \geq 1}$. 

Some natural questions which can then be asked are the following: 
\begin{enumerate}
\item Is it possible to replace the space $\mathcal{F}$ by another space with a simpler description? 
\item Is it possible to define a version of the flow $(V^{\alpha})_{\alpha \in \mathbb{R}}$ of operators, which can be naturally related to ensembles of unitary or hermitian matrices, which are different from the CUE?
\end{enumerate}
An answer to this last question would give a more generic version of $(V^{\alpha})_{\alpha \in \mathbb{R}}$, which may enlighten in a new, more geometric way, the properties of universality enjoyed by the sine-kernel process in random matrix theory.

Such a generalization seems to be very difficult to construct. In particular, we do not know how one could couple the GUE in such a way that the renormalized eigenvalues converge almost surely. Note that the successive minors of an infinite GUE matrix cannot converge almost surely to a non trivial (i.e. non constant) limiting distribution. Indeed, if we note $(A_n)_{n\geq1}$ the successive minors, and $F_n: \mathcal M_n(\mathbb C)\to \mathbb R$, which depends only on the eigenvalues and such that $F_n(A_n)$ converges a.s. to $X$, then $F_n(A_n)$ and $F_{2n}(A_{2n})$ both converge a.s. to $X$. In particular $(F_{n}(A_{n}),F_{2n}(A_{2n}))$ converges in law to $(X,X)$. Since we work with GUE matrices, it also follows that $(F_{n}(B_{n}),F_{2n}(A_{2n}))$ also converges to $(X,X)$, where $B_n$ is the $n$ by $n$ bloc matrix obtained by taking the lower right bloc in $A_{2n}$. Thus it follows that $F_{n}(A_{n})-F_{2n}(A_{2n})$ converges in law (and also in probability) to $X-X=0$. Similarly $F_{n}(B_{n})-F_{2n}(A_{2n})$ also converges in probability to $0$. Summing up we obtain that $F_{n}(B_{n})-F_{n}(A_{n})$ converges in probability to $0$. On the other hand $F_{n}(B_{n})-F_{n}(A_{n})$ also converges to $X-Y$ where $X$ and $Y$ are independent with the same law. This implies that $|\mathbb E[e^{itX}]|^2=1$ and hence $X$ is equal to a constant a.s.

It would also be interesting to relate the operators we construct to the Brownian carousel construction, which was introduced by Valk\'o and Vir\'ag in \cite{VV} (or alternatively to the work by Killip and Stoiciu \cite{KilSto}), in order to generalize the sine-kernel process to the setting of $\beta$-ensembles for any $\beta \in (0,\infty)$ (instead of just $\beta = 2$).

The paper is devoted to providing the details of the construction of the flow of operators given in Theorem \ref{thm:main}. Our construction is both geometric and  probabilistic and this is reflected in  our arguments which involve for instance  coupling techniques or martingale convergence theorems. For the convenience of the reader, we will first briefly outline in the next section, without proof, the construction of the flow of operators $(V^\alpha)_{\alpha \in \mathbb{R}}$ in Theorem~\ref{thm:main}. 

\section{Overview of the construction}

In Section~\ref{sec:isometries} we define the set of \emph{complex reflections} of a finite dimensional complex vector space. These reflections are unitary maps which are linear combinations of the identity and a map of rank one. It turns out that a unitary matrix $u \in U(n)$ can be written as a product
\begin{equation}\label{NTM}
 u = r_n \cdots r_1
\end{equation}
where each $r_k$, $1 \leq k \leq n$ is a complex reflection. In this decomposition, each reflection $r_k$ fixes every standard basis vector $e_j$ with $k < j \leq n$, so that the partial products $r_k \cdots r_1$ are naturally unitary maps on $U(k)$. This decomposition has already been used in \cite{bhny} to provide a new approach to the study of the characteristic polynomial of random unitary matrices (in particular in relation with the moments conjecture of Keating and Snaith \cite{keating-snaith}) and it was also exploited in \cite{BNR} in the study of the circular Jacobi ensemble.

Decomposition (\ref{NTM}) was also used in \cite{BNN12} to define virtual isometries $(u_n)_{n \geq 1}$ as families of unitary matrices, $u_n \in U(n)$, so that $u_n = r_n u_{n-1}$ for every $n \geq 1$ for some family of reflections $(r_n)_{n \geq 1}$. The push-forward of Haar measure on $U(n)$ decomposes in a simple way: the distribution of each $r_k$ is independent, and each $r_k(e_k)$ takes values uniformly on the unit sphere in $\C^k$.  The Haar measure on $U(n)$ for $n$ finite can then be naturally extended to a uniform measure on the set of virtual isometries.

The next step is to find an explicit relation between the eigenvalues and eigenvectors of $u_{n+1}$ and $u_n$ for each $n \geq 1$. Let us write $\lambda_1^{(n)} = e^{i \theta_1^{(n)}}, \ldots, \lambda_n^{(n)} = e^{i \theta_n^{(n)}}$ for the eigenvalues of $u_n$, almost surely distinct and different from $1$, with the ordering $0 < \theta_1^{(n)} < \cdots < \theta_n^{(n)} < 2 \pi$. It will be convenient 
to extend the notation $\lambda_k^{(n)}$ and 
$\theta_k^{(n)}$ to all $k \in \Z$, in
such a way that $\theta_{k+n}^{(n)} 
= \theta_k^{(n)} + 2\pi$ and
$\lambda_{k+n}^{(n)} = \lambda_{k}^{(n)}$, i.e. the sequence  $(\lambda_{k}^{(n)})_{k \in 
\Z}$ is $n$-periodic. Note that 
with this convention, $(\theta_k^{(n)})_{k \in 
\Z}$ is the increasing sequence of eigenangles of $u_n$, taken in the whole real line. 
 For $1 \leq k \leq n$, let $f_k^{(n)} \in \C^n$ denote a unit length representative of the eigenspace for $\lambda_k^{(n)}$. Then if we expand $r_{n+1}(e_{n+1})$ in the basis of eigenvectors
\[
  r_{n+1}(e_{n+1}) = \sum_{j=1}^n \mu_j^{(n)} f_j^{(n)} + \nu_n e_{n+1}
\]
then the eigenvalues of $u_{n+1}$ are precisely the zeros of the rational equation
\[
  \sum_{j=1}^n \abs{\mu_j^{(n)}}^2 \frac{\lambda_j^{(n)}}{\lambda_j^{(n)} - z} + \frac{\abs{1 - \nu_n}^2}{1 - z} = 1 - \overline{\nu}_n
\]
and the eigenvectors of $u_{n+1}$ are given by the $n+1$ equations
\[
  C_k f_k^{(n+1)} = \sum_{j=1}^n \frac{\mu_j^{(n)}}{\lambda_j^{(n)} - \lambda_k^{(n+1)}} f_j^{(n)} + \frac{\nu_n - 1}{1 - \lambda_k^{(n+1)}} e_{n+1}
\]
for $1 \leq k \leq n+1$; here $C_k \in \R^+$ is a constant so that $f_k^{(n+1)}$ has unit length. As above, we will also extend, when needed,  the notation $f_k^{(n)}$, $\mu_j^{(n)}$, in such a way that the sequences $(f_k^{(n)})_{k \in \Z}$ and 
$(\mu_j^{(n)})_{j \in \Z}$ are $n$-periodic.

In Section~\ref{sec:filtration} we make the following key observation: the sequence of eigenvalues $\lambda_k^{(n)}$, with $1 \leq k \leq n$ and $n \geq 1$, is independent of the argument of the coefficients $\mu_j^{(n)}$, with $1 \leq j \leq n$ and $n \geq 1$. Therefore, we can consider the sequence of eigenvalues of the virtual isometry and prove that it converges almost surely, and then, \emph{conditioning on the eigenvalues of every matrix in the virtual isometry}, consider the sequence of eigenvectors and show that they also converge in a suitable sense.

The first part of this plan is carried out in Section~\ref{sec:eigenvalues}, where it is shown that the eigenangles converge to a limiting process $(y_k)_{k \in \Z}$ in the sense that
\[
  \frac{n}{2 \pi} \theta_k^{(n)} = y_k + O_\eps((1 + k^2) n^{-\frac13 + \eps})
\]
almost surely for every $\eps > 0$.

Next, in Section~\ref{sec:eigenvectors} we condition on the eigenangles of the entire sequence of matrices. We show that for each fixed $k$ there is a renormalization factor $D_k^{(n)}$ so that for each $\ell \geq 1$ the sequence $\langle D_k^{(n)} f_k^{(n)}, e_\ell \rangle$ is a martingale which converges in $L^2$ and almost surely to a limiting value $g_{k,\ell}$. For each $k \geq 1$, we call the sequence $(g_{k,\ell})_{\ell \geq 1} \in \C^\infty$ an eigenvector of the isometry.

We require stronger information on the convergence of the eigenvectors in order to make a suitable definition. In Section~\ref{sec:flowoperators} we consider just those sequences in $\C^\infty$ which are spanned by the eigenvectors of the isometry. Then we show that we can define a ``flow'' $(U^\alpha)_{\alpha \in \R}$ on this vector space by considering sequences $u_n^{\floor{\alpha n}} t[n]$, where $(t_\ell)_{\ell \geq 1}$ is a linear combination of eigenvectors. In fact, for each eigenvector $g_k$ and $\ell \geq 1$ we show that $\langle u_n^{\floor{\alpha n}} g_k[n], e_\ell \rangle$ converges almost surely for every $\alpha$ to $e^{2 \pi i \alpha y_k} g_{k,\ell}$. Furthermore, we show for all $\delta > 0$ that uniformly for $\alpha$ in compact sets,
\[
  \norm{ u_n^{\floor{\alpha n}} g_k[n] - e^{2 \pi i \alpha y_k} g_k[n] } = O_\delta(n^{\frac12 - \delta}) = o_\delta(\norm{g_k[n]})
\]
so that the flow converges strongly (in this sense) in the limit.

We are now ready to define our random vector space and the associated flow. We define our random space $\mathcal{F}$ to be those sequences $(w_\ell)_{\ell \geq 1}$, along with sequences $((V^\alpha w)_\ell)_{\ell \geq 1}$, so that $u_n^{\floor{\alpha n}} w$ converges weakly to $V^\alpha w$ and so that
\[
  \norm{u_n^{\floor{\alpha n}} w[n] - (V^\alpha w) [n]} = O(n^{\frac12 - \delta}).
\]
for some $\delta > 0$ and choice of the implied constant. This latter inequality ensures that we do not introduce additional ``exotic'' eigenvalues in the limit. We then define the flow $V^\alpha : w \mapsto V^\alpha w$ from the definition of $\mathcal{F}$.

\section{Reflections and the definition of virtual isometries} \label{sec:isometries}
In this section, we recall the definition and elementary properties of reflections and the way they are used to construct virtual isometries. The reader can refer to \cite{BNN12} for more details but we wanted to give an independent presentation here to make the paper self-contained and also to fix notation and conventions. The idea behind is to provide a way to generate the Haar measure with the help of ``simple'' and independent unitary transformations. The idea is not new and  for instance a general algorithm is given by \cite{DS87}.
\subsection{Reflections over $\R$}

We will begin by briefly recalling the definition of reflections over $\R$ since this is the one which is used if one wants to carry our construction to the orthogonal group. We also wish to state them here in order  to understand why real reflections or Householder transformations would not be suitable if the ground field is the field of complex numbers.

\begin{definition}
    Let $H$ denote a real vector space of dimension $n$. A \textbf{reflection} of $H$ (sometimes called a \textbf{Householder transformation} or an \textbf{elementary reflector}) is defined as an orthogonal transformation of $H$ which fixes each element of a hyperplane (i.e.~linear subspace of codimension $1$).
\end{definition}

Let $r$ be a reflection. Let us assume that $r \neq 1$ so that $\ker (1 - r) = K$ is a hyperplane and $\dim \im(1 - r) = 1$.

Since $r \in O(\R)$ we must have $\det r \in \{\pm 1\}$. Clearly $1$ is an eigenvalue with multiplicity $n - 1$. Let $\lambda \neq 1$ denote the other eigenvalue. As $\det r$ is the product of the eigenvalues we conclude $\lambda = -1$ and $\det r = -1$. Consequently there exists a vector $a \in H$ such that $r(a) = -a$ and note that $r$ is a map of order $2$ (namely $r^2 = 1$). We will write $r_a$ for the reflection which maps $a \mapsto -a$.

It is easy to find a formula for $r_a$: First note that $\R a \oplus (\R a)^\perp = H$. There exists $\phi \in H^*$ (where $H^*$ is the space of linear forms on $H$) such that $x - r_a(x) = \phi(x) a$ for every $x \in H$ and $\ker \phi = \ker (1 - r_a) = K$. Since $\ker(\langle \cdot, a \rangle) = K$, there exists a non-zero $\lambda \in \C^*$ such that $\phi = \lambda \langle \cdot, a \rangle$. Therefore $x - r_a(x) = \lambda \langle x, a \rangle a$. Now $r_a(a) = -a$ so we have $\lambda = \frac{2}{\langle a, a \rangle}$, hence
\begin{equation}\label{desperatehousewives}
  r_a(x) = x - 2 \frac{\langle x, a \rangle}{\langle a, a \rangle} a.
\end{equation}
This is the equation of a real orthogonal reflection that sends $a$ to $-a$ and hence a householder transformation is parametrized by a vector $a$.

Finally, we note that if $m, e \in H$ are two distinct vectors with $\norm{m} = \norm{e} = 1$, there exists a unique reflection which maps $m$ on $e$, namely  $r_{m - e}$. But it is easy to see that given two vectors of norm $1$ in a complex Hilbert space, then there does not necessarily exist a Householder transformation (\ref{desperatehousewives}) which maps one on the other. Hence we need to introduce another type of reflection.

\subsection{Reflections over $\C$}

Over the complex numbers the theory is different since the reflections we consider are not necessarily of finite order and they are parametrized by one vector and a phase (that is, a complex number of modulus $1$). 

We now assume we  are given a Hilbert space $H$ with $\dim H = n$. For $u, v \in H$, we use the standard inner product
\[
  \langle u, v \rangle := \sum_{k=1}^n u_k \overline{v}_k.
\]
If $F \subseteq H$, $F$ a subset of $H$, we write $F^\perp := \{ x \in H \mid \langle u, x \rangle = 0 \text{ for all } u \in F \}$.

If $F$ and $G$ are subspaces of $H$, we write $H = F \oplus^\perp G$ to indicate $H = F + G = \{x + y \mid x \in F, y \in G\}$ and $\langle x, y \rangle = 0$ for all $x \in F$ and $y \in G$. It is then easy to check that $F = G^\perp$ and $(G^\perp)^\perp = G$ and $\dim G + \dim G^\perp = \dim H = n$.

Let $U(H)$ denote the set of unitary operators, i.e.~linear bijections $u : H \to H$ which preserve the inner product: $\langle ux, uy \rangle = \langle x,y \rangle$ for every $x, y \in H$. Let $1$ denote the identity map. We first recall the elementary and well known result: 

\begin{lemma} \label{lem:fredholm}
    If $u \in U(H)$ then $\im(1 - u) = \ker(1 - u)^\perp$.
\end{lemma}


\begin{definition}
Let $u$ denote a linear transform of a complex vector space $H$ of finite dimension. We call $u$ a \textbf{complex reflection} if it is the identity or if it is unitary and $\rank (1 - u) = 1$. We will just write \textbf{reflection} if the base field is clear.
\end{definition}

\begin{remark}
    Usually in the literature, a linear transformation $g \in GL(H)$ is a reflection if the order of $g$ is finite and $\rank(1 - u) = 1$. For our applications we do not require that $g$ have finite order.
\end{remark}


As in the real case, we can compute a formula for a general complex reflection.

\begin{proposition}
Suppose that $r$ is a reflection of the space $H$ and that the vector $a \in H$ spans $\im(1 - r)$. Then there exists $\alpha \in \mathbb{U}$ such that for every $x \in H$,
\[
  r(x) = x - (1 - \alpha) \frac{\langle x, a \rangle}{\langle a, a \rangle} a.
\]
\end{proposition}

\begin{proof}
    We have by construction and Lemma~\ref{lem:fredholm} the equivalences $\im(1 - u) = \C a$ and $\ker (1 - u) = (\C a)^\perp$. Let $\phi \in H^*$ denote the linear form defined by $u(x) = x - \phi(x) a$. Clearly $\phi(x) = 0$ if and only if $(1 - u) x = 0$, so $\ker \phi = \ker (1 - u)$. But the linear form $\langle \cdot, a \rangle$ also vanishes on $\ker (1 - u) = (\C a)^\perp$, so we must have some $\lambda \in \C^*$ such that $\phi = \lambda \langle \cdot, a \rangle$. Hence $u(x) = x - \lambda \langle x, a \rangle a$.

    To determine $\lambda$, we note that $u(a) = \alpha a$ for some $\alpha \in \C$, $\abs{\alpha} = 1$, so we must have $\lambda = \frac{1 - \alpha}{\langle a, a \rangle}$ as required.
\end{proof}

\begin{definition}
For non-zero $a \in H$ and $\alpha \in \C$, $\abs{\alpha} = 1$, we define the reflection $r_{a, \alpha}(x)$ by
\[
  r_{a,\alpha}(x) = x - (1 - \alpha) \frac{\langle x, a \rangle}{\langle a, a \rangle} a.
\]
\end{definition}

Note that $r_{a, \alpha}$ has eigenvalue $1$ with multiplicity $n-1$ and eigenvalue $\alpha$ with multiplicity $1$. The following facts are easy to check and we omit the proofs.
\begin{proposition} \label{prop:propertiesofreflections}
    For any non-zero $a \in H$ and $\alpha, \beta \in \mathbb{U}$ we have the following.
    \begin{enumerate}
        \item $r_{a, \alpha} r_{a, \beta} = r_{a, \alpha \beta}$.
        \item For every $g \in U(H)$, $g r_{a, \alpha} g^{*} = r_{ga, \alpha}$.
        \item For every non-zero $\lambda \in \C$, $r_{\lambda a, \alpha} = r_{a, \alpha}$.
        \item $r_{a,\alpha}^{-1} = r_{a,\overline{\alpha}}$.
    \end{enumerate}
\end{proposition}

\begin{proposition}
    Let $a, b \in H$ be non-zero vectors and $\alpha, \beta \in \mathbb{U}$. Then the reflections $r_{a,\alpha}$ and $r_{b,\beta}$ commute if and only if $\C a = \C b$ or $\langle a, b \rangle = 0$.
\end{proposition}

\begin{proof}
This follows immediately by writing
\[
   r_{a,\alpha} r_{b,\beta}(x) = x - (1 - \alpha) \frac{\langle x, a \rangle}{\langle a, a \rangle} a - (1 - \beta) \frac{\langle x, b \rangle}{\langle b, b \rangle} + (1 - \alpha) (1 - \beta) \frac{\langle b, a \rangle \langle x, b \rangle}{\langle a,a \rangle \langle b,b \rangle} a
\]
which shows that the reflections commute if and only if
\[
  \langle b, a \rangle \langle x,b \rangle a = \langle a, b \rangle \langle x,a \rangle b. \qedhere
\]
\end{proof}

Now we note that given two distinct vectors $e, m \in H$ of unit length, there exists a unique complex reflection $r$ such that $r(e) = m$, which is $r_{m-e,\alpha}$ where $\alpha =- \frac{1 - \langle m, e \rangle}{1 - \overline{\langle m, e \rangle}}.$ Such $r$ is given by the equation
\[
  r(x) = x - \frac{\langle x, m - e \rangle}{1 - \overline{\langle m, e \rangle}}  (m - e).
\]

\subsection{Virtual isometries and projections of unitary matrices}

We now explain how to construct an infinite dimensional structure which contains in a natural way each finite dimensional unitary group. It is crucial to our construction that the unitary matrices between different dimensions be closely linked. This will allow us to give a meaningful definition to almost sure convergence of random matrices as the dimension grows to infinity.

 \begin{proposition} \label{prop:projectionofunitary}
     Let $H$ be a complex Hilbert space, $E$ a finite dimensional subspace of $H$ and $F$ a subspace of $E$. Then for any unitary operator $u$ acting on $H$ which fixes every vector in $E^\perp$, there exists a unique unitary operator $\pi_{E, F}(u)$ on $H$ which satisfies the following two conditions.
     \begin{enumerate}
         \item $\pi_{E,F}(u)$ fixes every vector in $F^\perp \supseteq E^\perp$.
         \item The image of $H$ under $u - \pi_{E,F}(u)$ is contained in the image of $F^\perp$ under $u - 1$.
     \end{enumerate}
     Moreover if $G$ is a subspace of $F$, then $\pi_{F, G} \circ \pi_{E, F}(u)$ is a well-defined unitary operator on $H$ and is equal to $\pi_{E, G}(u)$.
 \end{proposition}
 \begin{proof}
 First we prove uniqueness. Let $x \in F \cap (u - 1)(F^\perp)$. There exists $y \in F^\perp$ such that $x = u(y) - y$. Since $x \in F$ and $y \in F^\perp$ we have
 \[
   \norm{u(y)}^2 = \norm{y}^2 + \norm{x}^2
 \]
 by the Pythagorean theorem. Since $u$ is unitary $\norm{u(y)} = \norm{y}$ so we would require $\norm{x} = 0$, hence $F \cap (u - 1)(F^\perp) = \{0\}$.

 Now if $v_1$ and $v_2$ are two unitary operators satisfying the properties of $\pi_{E,F}(u)$, then:
 \begin{enumerate}
     \item $v_1$ and $v_2$ fix globally $F$ since they fix $F^\perp$.
     \item $v_1 - v_2$ vanishes on $F^\perp$ by construction.
     \item The range of $v_1 - v_2$ is included in $(u - 1)(F^\perp)$ since the range of each of $v_1 - u$ and $u - v_2$ are.
 \end{enumerate}
  We conclude from (2) and (3) that the image of $v_1 - v_2$ is contained in $F \cap (u-1)(F^\perp) = \{0\}$ and so the operators must agree.

  Next we prove the tower property of $\pi_{E,F}(u)$, assuming that the operator exists.

  Let $G \subset F \subset E \subset H$. We write $v = \pi_{E,F}(u)$ and $w = \pi_{F,G}(v)$. These operators are well-defined and satisfy
  \begin{enumerate}
     \item $v$ fixes every vector in $F^\perp$
     \item $(u-v)(H) \subseteq (u - 1)(F^\perp)$
     \item $w$ fixes every vector in $G^\perp$
     \item $(v - w)(H) \subseteq (v - 1)(G^\perp)$.
  \end{enumerate}
  This yields the elementary calculation
  \begin{align*}
    (u - w)(H) & \subseteq \spn\{(u - v)(H), (v - w)(H)\} \\
    & \subseteq \spn \{(u - 1)(F^\perp), (v - 1)(G^\perp)\} \\
    & \subseteq \spn \{(u - 1)(F^\perp), (u - 1)(G^\perp), (u - v)(G^\perp)\} \\
    & \subseteq \spn \{(u - 1)(G^\perp), (u - v)(H)\} \\
    & \subseteq (u - 1)(G^\perp).
  \end{align*}
  Since $w$ fixes each vector of $G^\perp$, it is equal to $\pi_{E,G}(u)$.

  Now we prove the existence of the operator by induction. It is sufficient to prove the existence of $\pi_{E,F}$ in the particular case where $E = \spn\{F,e\}$, where $e$ is a unit vector orthogonal to $F$. In this case, if $u$ is a unitary operator fixing each vector of $E^\perp$, then the operator $v = \pi_{E,F}(u)$ can be constructed explicitly as follows.

  If $u(e) = e$ then on taking $v = u$ which fixes $\spn \{E^\perp, e\}$, hence it fixes $F^\perp$ and $(u - v)(H) = (u - 1)(F^\perp)$.

  If $u(e) \neq e$ then for all $x \in H$ we define
  \[
    v(x) := u(x) - \frac{\langle u(x), e - u(e) \rangle}{1 - \langle u(e), e \rangle} (e - u(e)) = \widetilde{r} \circ u
  \]
 where $\widetilde{r}$, as indicated, is the unique reflection mapping $u(e) \mapsto e$. Hence $v$ is a unitary transformation.

 Now let $x \in E^\perp$ so that $u(x) = x$ and $e - u(e) \in E$ since $E$ is globally fixed by $u$. Here $\langle u(x), e - u(e) \rangle = \langle x, e - u(e) \rangle = 0$ so $v(x) = x$. Moreover by construction $v(e) = e$, so consequently $v$ fixes each vector in $F^\perp$.

 Finally, for all $x \in H$, we have $u(x) - v(x) = \gamma (e - u(e))$ for some $\gamma \in \C$, so $(u-v)(H) \subseteq (u - 1)(F^\perp)$, hence $v$ satisfies the conditions of $\pi_{E,F}(u)$.
 \end{proof}

 \begin{remark}
     It follows from Proposition~\ref{prop:propertiesofreflections} that if $r := (\widetilde{r})^{-1}$ is the reflection mapping $e$ onto $u(e)$ then
     \[
       u = r \pi_{E,F}(u).
     \]
     Similarly one can easily prove that $u = \pi_{E,F} \circ r'$ where $r'$ is the unique reflection such that $r'(u^{-1}(e)) = e$.
 \end{remark}

 The existence of the projection map described above suggests how to define ``virtual isometries'', the infinite dimensional objects alluded to above. Indeed, let $H = \ell^2(\C)$ and $(e_\ell)_{\ell \geq 1}$ be the canonical Hilbert basis of $H$. Then for all $n \geq 1$ the space of unitary operators fixing each element of $V_n := \spn(e_1, \ldots, e_n)^\perp$ can be canonically identified with the unitary group $U(n)$.

 By identification, for $n \geq m \geq 1$ the projection $\pi_{V_n, V_m}$ gives a map from $U(n)$ to $U(m)$, more simply noted $\pi_{n,m}$, so that, for $n \geq m \geq p \geq 1$,
 \[
   \pi_{n,p} = \pi_{m,p} \circ \pi_{n,m}.
 \]

 \begin{definition}
     A \textbf{virtual isometry} is a sequence $(u_n)_{n \geq 1}$ of unitary matrices such that for all $n \geq 1$, $u_n \in U(n)$ and $\pi_{n+1,n}(u_{n+1}) = u_n$. The space of virtual isometries will be denoted $U^\infty$.
 \end{definition}

 \begin{remark}
 $U^\infty$ does not appear to have a natural group structure.
 \end{remark}

 \begin{remark}
 In this definition we made a particular choice of vector spaces lying in $\ell^2(\C)$. Although it appears to be necessary to make a choice for the probabilistic arguments later, the ideas in this construction apply also to other Hilbert spaces with a sequence of basis vectors.
 \end{remark}

 \begin{proposition} \label{prop:vectorstoisometry}
     Let $(x_n)_{n \geq 1}$ be a sequence of vectors with $x_n \in \C^n$ lying on the unit sphere (i.e.~$\norm{x_n}_{\ell^2(\C^n)} = 1$). Then there exists a unique virtual isometry $(u_n)_{n \geq 1}$ such that $u_n(e_n) = x_n$ for every $n \geq 1$. In particular, $u_n$ is given by
     \[
       u_n = r_n r_{n-1} \cdots r_1
     \]
     where for each $1 \leq j \leq n$, $r_j = 1$ if $x_j = e_j$ and otherwise is the unique reflection mapping $e_j \mapsto x_j$.
 \end{proposition}

 \begin{proof}
     This follows directly from the proof of Proposition~\ref{prop:projectionofunitary} and the remarks which follow it.
 \end{proof}

 \begin{remark}
 In the particular case where every $x_n = e_{t(n)}$ for some $t(n) \in \{1, \ldots, n\}$, then $(u_n)_{n \geq 1}$ is the sequence of matrices associated to a virtual permutation \cite{NN13}. The sequence of permutations $(\sigma_n)_{n \geq 1}$ is constructed by the so-called Chinese restaurant process \cite{Pit06}: for all $n \geq 1$,
     \[
       \sigma_n = \tau_{n,t(n)} \tau_{n-1, t(n-1)} \cdots \tau_{1,1}
     \]
 where $\tau_{k,j} = 1$ if $j = k$ and otherwise is the transposition $(j,k)$.
 \end{remark}

 \begin{remark}
     The above proposition shows in particular the fact that any unitary matrix of $U(n)$ can be expanded into a product of reflections.
 \end{remark}

 Proposition~\ref{prop:vectorstoisometry} shows in particular that $U^\infty$ is non-empty. Moreover, our construction will allow us to construct measures on $U^\infty$. Indeed, it is natural to look for the analogue of Haar measure on $U^\infty$. First, we recall the following correspondence from \cite{BNN12} (this can also be obtained as a consequence of the results above and the subgroup algorithm of Diaconis and Shahshahani \cite{DS87}). 

 \begin{proposition}[\cite{BNN12}] \label{prop:haartovectors}
 Let $(x_n)_{n \geq 1}$ be a random sequence of vectors with $x_n \in \C^n$ and let $(u_n)_{n \geq 1}$ be the unique virtual isometry such that $u_n(e_n) = x_n$ for all $n \geq 1$. Then for each $n \geq 1$ the matrix $u_n$ is distributed according to Haar measure if and only if $x_1, \ldots, x_n$ are independent and for each $1 \leq j \leq n$, $x_j$ is distributed according to the uniform measure on the unit sphere in $\C^j$.
 \end{proposition}

 As a consequence, we deduce the compatibility between Haar measure on $U(n)$, $n \geq 1$, and the projections $\pi_{n,m}$ for $n \geq m \geq 1$.

 \begin{proposition} \label{prop:projecthaar}
 For all $n \geq m \geq 1$, the push-forward of the Haar measure on $U(n)$ under $\pi_{n,m}$ is the Haar measure on $U(m)$.
 \end{proposition}


 \begin{remark}
     One has to check that $\pi_{n,m}$ is measurable as is done in \cite{BNN12} \end{remark}

 The compatibility property of Proposition~\ref{prop:projecthaar} allows us to define the Haar measure on $U^\infty$.

 \begin{proposition} \label{prop:measuresforvirtualisometries}
     Let $\mathcal{U}$ denote the $\sigma$-algebra on $U^\infty$ generated by the sets
     \[
       W(k, B) := \{ (u_n)_{n \geq 1} \mid u_k \in B \}
     \]
     where $k \geq 1$ and $B \subseteq U(k)$ is a Borel set. Let $(\mu_n)_{n \geq 1}$ be a family of probability measures, $\mu_n$ defined on $(U(n), \mathcal B(U(n)))$, where $\mathcal B(U(n))$ is the Borel set on $U(n)$, and such that the push-forward of $\mu_{n+1}$ by $\pi_{n+1,n}$ is equal to $\mu_n$. Then there exists a unique probability measure $\mu$ on $(U^\infty, \mathcal{U})$ such that push-forward of $\mu$ by the $n$th coordinate map is equal to $\mu_n$ for all $n \geq 1$.
 \end{proposition}

 \begin{proof}
     This is the Kolmogorov extension theorem applied to the family $(\mu_n)_{n \geq 1}$. See \cite{BNN12}.
 \end{proof}

 Now combining Proposition~\ref{prop:haartovectors} and Proposition~\ref{prop:measuresforvirtualisometries} we obtain the following.

 \begin{proposition} [\cite{BNN12}]\label{prop:haaronvirtualisometries}
 There exists a unique probability measure $\mu^{(\infty)}$ on the space $(U^\infty, \mathcal{U})$ such that its push-forward by the coordinate maps is equal to the Haar measure on the corresponding unitary group. In particular, let $(x_n)_{n \geq 1}$ be a random sequence of vectors with each $x_n \in \C^n$ almost surely on the unit sphere, and let $(u_n)_{n \geq 1}$ be the unique virtual isometry such that $u_n(e_n) = x_n$. Then the distribution of $(u_n)_{n \geq 1}$ is equal to $\mu^{(\infty)}$ if and only if the $x_n$ are independent random variables and for all $n \geq 1$, $x_n$ is uniformly distributed on the unit sphere.
 \end{proposition}

\section{Spectral analysis of the virtual isometries} \label{sec:spectral}

It is classical that if $u_n$ is a random unitary matrix following the Haar measure on $U(n)$, then the distribution of the eigenangles of $u_n$, multiplied by $n/2 \pi$, converges in law to a determinantal sine-kernel process. In fact this result can be found in the literature under the form of the convergence of the correlation functions against a suitably chosen family of test functions. However we were not able to find a statement with its proof on the fact that the convergence takes place in the sense of weak convergence of point processes, using Laplace functionals. So for completeness, we give such a statement below and postpone  its proof until the appendix.

\begin{proposition}\label{Laplacefunctionals}
Let $E_n$ denote the set of eigenvalues taken in $(-\pi,\pi]$ and multiplied by $n/2\pi$ of a random unitary matrix of size $n$ following the Haar measure. Let us also define for $y\neq y^{'}$
$$K^{(\infty)}(y,y^{'})=\dfrac{\sin[\pi(y^{'}-y)]}{\pi (y^{'}-y)}$$and $$K^{(\infty)}(y,y)=1.$$ Then there exists a point process $E_{\infty}$ such that for  all $r \in \{1, \dots, n\}$, and for all 
Borel measurable and bounded functions $F$  with compact support from $\mathbb{R}^r$ to $\mathbb{R}$, we have, as $n\to\infty$,

$$\mathbb{E} \left( \sum_{x_1 \neq \dots
\neq x_r \in E_n} F(x_1, \dots, x_r) \right) \to \int_{\mathbb{R}^r} 
F(y_1, \dots, y_r) \rho^{(\infty)}_r(y_1, \dots,
y_r) dy_1 \dots dy_r,$$
where
$$\rho_r^{(\infty)}(y_1, \dots, y_r) 
= \operatorname{det} ( (K^{(\infty)} (y_j, y_k))_{1 \leq j, k \leq r} ).$$
Moreover the point process $E_n$ converges to $E_{\infty}$ in the following sense: 
for all Borel measurable bounded functions $f$ with compact support from $\mathbb{R}$ to
$\mathbb{R}$, 
$$\sum_{x \in E_n} f(x) 
\underset{n \rightarrow \infty}{\longrightarrow}
\sum_{x \in E_{\infty}} f(x),$$
where the convergence above holds in law. 

\end{proposition}

In \cite{BNN12} it was proven that the eigenangles of a virtual isometry, taken according to Haar measure and renormalizing the eigenangles by the dimension $n$, converge almost surely to a point process with this determinantal distribution. Precisely, we have the following.
\begin{proposition} \label{convergence-eigenangles}
Let $(u_n)_{n \geq 1}$ be a random virtual isometry following the Haar measure. 
The eigenvalues of $u_n$ are almost surely distinct and different from $1$, 
and then, as explained before, it is possible to order the eigenangles as an increasing sequence indexed by $\Z$: 
 $$\dots < \theta_{-2}^{(n)}  < \theta_{-1}^{(n)}
 < \theta_{0}^{(n)} < 0 < \theta_{1}^{(n)} < \theta_{2}^{(n)} <\dots.$$
Moreover, almost surely, for all
 $m \in \mathbb{Z}$, there exists $y_m$ such that 
$$\frac{n}{2 \pi} \theta_m^{(n)} = y_m + O(n^{-\epsilon}),$$
when $n$ goes to infinity, $\epsilon > 0$ being some universal constant, and 
the process $(y_m)_{m \in \mathbb{Z}}$ is a determinantal sine-kernel process. 
\end{proposition}
\begin{remark}
In this proposition, the periodic extension of 
the sequence $(\theta_k^{(n)})_{1 \leq k \leq n}$ is needed to define $y_m$ for non-positive
values of $m$. We also note that $y_m$ depends only on the behavior of $\theta_m^{(n)}$ for a fixed value of $m$. Informally, the limiting sine-kernel process $(y_m)_{m \in 
\mathbb{Z}}$ depends only on the behavior
of the eigenvalues of $u_n$ which are close to $1$.  

\end{remark}
To construct our limiting flow of operators we need to understand the behavior of the eigenvectors of $u_n$ as $n$ goes to infinity. Our method will give a formula for the spectrum and eigenvectors of $u_{n+1}$ in terms of the spectrum and eigenvectors of $u_n$.

We assume throughout that for each $n \geq 1$, the $n$ eigenvalues of $u_n$ are distinct; this holds almost surely for virtual isometries constructed according to the Haar measure in Proposition~\ref{prop:haaronvirtualisometries}.

We recall that the eigenvalues of $u_n$, $\lambda_1^{(n)}, \lambda_2^{(n)}, \ldots, \lambda_n^{(n)}$, are ordered in such a way that $\lambda_k^{(n)} = e^{i \theta_k^{(n)}}$, and 
\[
  0 < \theta_1^{(n)} < \cdots < \theta_n^{(n)} < 2 \pi.
\]
 Moreover, the eigenangles enjoy a property of periodicity: for all $k \in \Z$, $\theta_{k + n}^{(n)} = \theta_k^{(n)} + 2 \pi$.

As all the eigenvalues are distinct, each eigenvalue corresponds to a one-dimensional eigenspace. We can therefore write $f_1^{(n)}$, ..., $f_n^{(n)}$ for the family 
of unit length eigenvectors of $u_n$, which are well-defined up to a complex phase: 
the notation $f_k^{(n)}$ is then extended $n$-periodically to all $k \in \Z$. 

Let $x_n = u_n(e_n)$ and let $r_n$ denote the unique reflection on $\C^n$ mapping $e_n$ to $x_n$. Therefore, we have $u_{n+1} = r_{n+1} \circ (u_n \oplus 1)$. It is natural to decompose $x_{n+1}$ into the basis given by $\iota(f_1^{(n)}), ..., \iota(f_n^{(n)}), e_{n+1}$, where $\iota : \C^n \to \C^{n+1}$ is the inclusion which maps $(x_1, \ldots, x_n)$ to $(x_1, \ldots, x_n, 0)$. Identifying $f_k^{(n)}$ and $\iota(f_k^{(n)})$, we then have
\[
  x_{n+1} = \sum_{k=1}^n \mu_k^{(n)} f_k^{(n)} + \nu_n e_{n+1}
\]
for some $\mu_k^{(n)}$ ($1 \leq k \leq n$) and $\nu_n$ such that $\abs{\mu_1^{(n)}}^2 + \cdots 
+ \abs{\mu_n^{(n)}}^2 + \abs{\nu_n}^2 = 1$.
Again, is can be convenient to consider 
$\mu_k^{(n)}$ for all $k \in \Z$, by a $n$-periodic extension of the sequence. 
The following result gives the spectral decomposition of $u_{n+1}$ in function of the decomposition of $u_n$ and 
$x_{n+1}$: 
\begin{theorem}[Spectral decomposition] \label{thm:spectrum}
On the event that the coefficients $\mu_1^{(n)}, \dots, \mu_n^{(n)}$ are all different from zero and 
that the $n$ eigenvalues of $u_n$ are all distinct (which holds almost surely under 
the uniform measure on $U^\infty$), the eigenvalues of $u_{n+1}$ are the unique roots of the rational equation
\[
  \sum_{j=1}^n \abs{\mu_j^{(n)}}^2 \frac{\lambda_j^{(n)}}{\lambda_j^{(n)} - z} + \frac{\abs{1 - \nu_n}^2}{1 - z} = 1 - \overline{\nu}_n
\]
on the unit circle. Furthermore, they interlace between $1$ and the eigenvalues of $u_n$
\[
  0 < \theta_1^{(n+1)} < \theta_1^{(n)} < \theta_2^{(n+1)} < \cdots < \theta_n^{(n)} < \theta_{n+1}^{(n+1)} < 2 \pi.
\]
and if the phases of the eigenvectors are suitably chosen, they satisfy the relation
\[
  (h_k^{(n+1)})^{\frac12} f_k^{(n+1)} = \sum_{j=1}^n \frac{\mu_j^{(n)}}{\lambda_j^{(n)} - \lambda_k^{(n+1)}} f_j^{(n)} + \frac{\nu_n - 1}{1 - \lambda_k^{(n+1)}} e_{n+1}
\]
where
\[
  h_k^{(n+1)} = \sum_{j=1}^n \frac{\abs{\mu_j^{(n)}}^2}{\abs{\lambda_j^{(n)} - \lambda_k^{(n+1)}}^2} + \frac{\abs{\nu_n - 1}^2}{\abs{1 - \lambda_k^{(n+1)}}^2}
\]
is the unique strictly positive real number which makes $f_k^{(n+1)}$ a unit vector.
\end{theorem}

\begin{proof}
Let $f$ be an eigenvector of $u_{n+1}$ with corresponding eigenvalue $z$. Then we have
\[
  f = \sum_{j=1}^n a_j f_j^{(n)} + b e_{n+1}
\]
where $a_1, ..., a_n, b$ are (as yet unknown) complex numbers, not all zero. Our goal is to write
these coefficients in terms of $x_{n+1}$ and the eigenvalues of $u_n$.

We have
\begin{align*}
    zf &= u_{n+1} f \\
    &= u_{n+1} \left( \sum_{j=1}^n a_j f_j^{(n)} + b e_{n+1} \right) \\
    &= \sum_{j=1}^n a_j u_{n+1} f_j^{(n)} + b u_{n+1} e_{n+1} \\
    &= \sum_{j=1}^n a_j \lambda_j^{(n)} r_{n+1} f_j^{(n)} + b x_{n+1}.
\end{align*}

We recall that for all $t \in \mathbb{C}^{n+1}$, $r_{n+1}(t)$ is given by
\[
  r_{n+1}(t) = t + \frac{\langle t, x_{n+1} - e_{n+1} \rangle}{\langle e_{n+1}, x_{n+1} - e_{n+1} \rangle} (x_{n+1} - e_{n+1})
\]
so that
\[
  zf = \sum_{j=1}^n a_j \lambda_j^{(n)} \left(f_j^{(n)} + \frac{\langle f_j^{(n)}, x_{n+1} - e_{n+1} \rangle}{\langle e_{n+1}, x_{n+1} - e_{n+1} \rangle} (x_{n+1} - e_{n+1}) \right) + b x_{n+1}.
\]
Now we decompose
\[
  x_{n+1} = \sum_{k=1}^n \mu_k^{(n)} f_k^{(n)} + \nu_n e_{n+1}
\]
and
\begin{align*}
  zf &= \sum_{j=1}^n a_j \lambda_j^{(n)} \left(f_j^{(n)} + \frac{\overline{\mu_j^{(n)}}}{\overline{\nu_n} - 1} (x_{n+1} - e_{n+1}) \right) + b x_{n+1} \\
  &= \sum_{j=1}^n a_j \lambda_j^{(n)} f_j^{(n)} + \left(\sum_{\ell=1}^n a_\ell \lambda_\ell^{(n)} \frac{\overline{\mu_\ell^{(n)}}}{\overline{\nu_n} - 1} \right) (x_{n+1} - e_{n+1}) + b x_{n+1}.
\end{align*}
Because $f_1^{(n)}, ..., f_n^{(n)}, e_{n+1}$ is a basis for $\mathbb{C}^{n+1}$, we deduce the system of $n+1$ equations
\[
  z a_j = a_j \lambda_j^{(n)} + \mu_j^{(n)} \sum_{\ell=1}^n a_\ell \lambda_\ell^{(n)} \frac{\overline{\mu_\ell^{(n)}}}{\overline{\nu_n} - 1} + b \mu_j^{(n)}
\]
for $j = 1, \ldots, n$ and
\[
  z b = b + (\nu_n - 1) \sum_{\ell=1}^n a_\ell \lambda_\ell^{(n)} \frac{\overline{\mu_\ell^{(n)}}}{\overline{\nu_n} - 1} + b (\nu_n - 1).
\]
For $z \notin \{\lambda_1^{(n)}, \dots, \lambda_n^{(n)}, 1\}$, let us consider
the linear transform $Q : \mathbb{C}^{n+1} \to \mathbb{C}^{n+1}$ whose matrix 
representation in the basis $f_1^{(n)}, \ldots, f_n^{(n)}, e_{n+1}$ is
\[
  Q = I + w v^t,
\]
where
\[
  w = \begin{pmatrix} \frac{\mu_1^{(n)}}{\lambda_1^{(n)} - z} \\ \vdots \\ \frac{\mu_n^{(n)}}{\lambda_n^{(n)} - z} \\ \frac{\nu_n - 1}{1 - z} \end{pmatrix}
\]
and
\[
  v^t = \left( \lambda_1^{(n)} \frac{\overline{\mu_1^{(n)}}}{\overline{\nu_n} - 1}, \cdots, \lambda_n^{(n)} \frac{\overline{\mu_n^{(n)}}}{\overline{\nu_n} - 1},  1 \right). 
\]
Then, the above system can be written
\[
  Q f = 0.
\]
Clearly, $\rank Q \in \{n, n+1\}$. If it has full rank then $f = 0$, but we assume a priori that $z$
 is an eigenvalue for $u_{n+1}$ and so has a non-trivial eigenspace. Thus we must have $\rank Q = n$ and
\[
  0 = Q f = f + w (v^t f).
\]
The right hand side can only vanish if $f$ is proportional to $w$, so $f = \alpha w$ for some 
complex constant $\alpha \in \C \setminus \{0\}$ and $v^t w = -1$. In particular,
\[
  \sum_{j=1}^n \frac{\lambda_j^{(n)} \abs{\mu_j^{(n)}}^2}{\lambda_j^{(n)} - z} +
 \frac{\abs{\nu_n - 1}^2}{1 - z} = 1 - \overline{\nu_n},
\]
as required.

Conversely, if $z \notin \{\lambda_1^{(n)}, \dots, \lambda_n^{(n)}, 1\}$
satisfies this equation, then
\[
Q w = w + w (v^t w) = w + w(-1) = 0,
\]
which implies that $w$ is an eigenvector of $u_{n+1}$ for the eigenvalue $z$. 

Let us now show that the eigenvalues $z \notin \{\lambda_1^{(n)}, \dots, \lambda_n^{(n)}, 1\}$ of
$u_{n+1}$ strictly interlace between $1$ and the eigenvalues of $u_n$: since $u_{n+1}$ has at most 
$n+1$ eigenvalues, this implies that $\lambda_1^{(n)}, \dots, \lambda_n^{(n)}, 1$ are not eigenvalues 
of $u_{n+1}$. 

Define the rational function $\Phi : S^1 \to \C \cup \{\infty\}$ by
\[
  \Phi(z) = \sum_{j=1}^n \frac{\lambda_j^{(n)} \abs{\mu_j^{(n)}}^2}{\lambda_j^{(n)} - z} + \frac{\abs{\nu_n - 1}^2}{1 - z} - (1 - \overline{\nu_n})
\]
Note that $\Phi$ vanishes precisely on the eigenvalues of $u_{n+1}$ which are different from 
$\lambda_1^{(n)}, \dots, \lambda_n^{(n)}, 1$. Recalling that $\abs{\mu_1^{(n)}}^2 + 
\cdots + \abs{\mu_n^{(n)}}^2 + \abs{\nu_n}^2 = 1$, we can rearrange the expression of $\Phi$ 
to the equivalent form
\[
  \Phi(z) = \frac{1}{2} \left( \sum_{j=1}^n \abs{\mu_j^{(n)}}^2 \frac{\lambda_j^{(n)} + z}
{\lambda_j^{(n)} - z} + \abs{1 - \nu_n}^2 \frac{1 + z}{1 - z} - \nu_n + \overline{\nu_n} \right). 
\]
Hence, $\Phi$ takes values only in $i \R \cup \{\infty\}$, since for 
all $z \neq z' \in S^1$, $(z+z')/(z-z')$ is purely imaginary (the triangle $(-z',z,z')$ has a right angle at $z$).
Note that for $z \in \{ \lambda_1^{(n)}, \dots, \lambda_n^{(n)}, 1\}$, a unique term of the sum defining 
$\Phi$ is infinite, since by assumption, $\mu_1^{(n)}, \dots, \mu_n^{(n)}, 1 - \nu_n$ are nonzero and 
  $\lambda_1^{(n)}, \dots, \lambda_n^{(n)}, 1$ are distinct: $\Phi(z) = \infty$. 

Next, we consider $t \mapsto \Phi(e^{it})$ in a short interval $(\theta_j^{(n)} - \delta, 
\theta_j^{(n)} + \delta)$. Then, for $t = \theta_j^{(n)} + u$ in this interval,
\[
  \frac{\lambda^{(n)}_j + \lambda^{(n)}_j e^{i u}}{\lambda^{(n)}_j - \lambda^{(n)}_j e^{i u}} = 
\frac{1 + e^{iu}}{1 - e^{iu}} = 2 i u^{-1} + O(1)
\]
while the other terms in $\Phi(e^{it})$ are uniformly bounded as $\delta \to 0$; likewise for the 
interval $(-\delta, \delta)$. In particular, $\Phi \to i \infty$ as $u \to 0$ from the right and
 $\Phi \to -i \infty$ as $u \to 0$ from the left. We therefore conclude, as $\Phi$ is continuous, that 
on each interval of the partition
\[
  (0, \theta_1^{(n)}) \cup (\theta_1^{(n)}, \theta_2^{(n)}) \cup \cdots \cup (\theta_n^{(n)}, 2 \pi)
\]
of the unit circle into $n+1$ parts, $t \mapsto \Phi(e^{it})$ must assume every value on the line $i\R$, and 
in particular must have at least one root. But we know that $\Phi$  has only $n+1$ roots on the
 circle so there must be exactly one root in each part of the partition, which proves the interlacing property. 

It remains to check the expression of the eigenvectors $(f_k^{(n+1)})_{1 \leq k \leq n+1}$ given in the theorem, 
but this expression is immediately deduced from the expression of the 
vector $w$ involved in the operator $Q$ defined above, and the fact that $\|f_k^{(n+1)}\| = 1$. 

\end{proof}

\section{A filtration adapted to the virtual isometry} \label{sec:filtration}
In our proof of convergence of eigenvalues and eigenvectors, we will use some martingale 
arguments. That is why in this section, we introduce a filtration related to our random virtual 
isometry. 
For $n \geq 1$, we define the $\sigma$-algebra $\mathcal{A}_n = \sigma\{\lambda_j^{(m)} \mid 1 
\leq m \leq n, 1 \leq j \leq m\}$
 and its limit $\mathcal{A} = \vee_{n=1}^\infty \mathcal{A}_n$. 
\begin{lemma} \label{lem:sigmaA}
    For all $n \geq 1$, the $\sigma$-algebra $\mathcal{A}_n$ is equal, up to completion, to
 the $\sigma$-algebra generated 
by $u_1$ the variables $\abs{\mu_j^{(m)}}$ and $\nu_m$ for $1 \leq m \leq n - 1$ and $1 \leq j \leq m$.
\end{lemma}

\begin{proof}
By Theorem~\ref{thm:spectrum},  
the eigenvalues of $u_{n+1}$ are almost surely the roots of the equation
\[
\sum_{j=1}^n \abs{\mu_j^{(n)}}^2 \frac{\lambda_j^{(n)}}{\lambda_j^{(n)} - z} + \frac{\abs{1 - \nu_n}^2}{1 - z} = 1 - \overline{\nu}_n.
\]
This equation depends only on $\abs{\mu_j^{(n)}}$, $\lambda_j^{(n)}$ for $j = 1, \dotsc, n$, and $\nu_n$. 
By induction, we deduce that $\lambda_j^{(n+1)}$ is a measurable function of $u_1$,  $\{|\mu_j^{(m)}|\}_{1 \leq j
 \leq m,1 \leq m \leq n}$ and $\{\nu_m\}_{1 \leq m \leq n}$.

Conversely, the above equation with $z$ specialized to $\lambda_1^{(n+1)}, \dotsc, \lambda_{n+1}^{(n+1)}$ can be written in the form
\[
  Rv = w.
\]
Here $w$ is a column vector of $1$s, $v$ is the column vector with entries
\[
  v^t = \begin{pmatrix} \frac{\abs{\mu_1^{(n)}}^2}{1 - \overline{\nu_n}}, & \cdots, & \frac{\abs{\mu_n^{(n)}}^2}{1 - \overline{\nu_n}}, & 1 - \nu_n \end{pmatrix}
\]
and $R$ is an $(n +1) \times (n + 1)$ matrix with entries
\[
  R_{k,j} = \frac{\lambda_j^{(n)}}{\lambda_j^{(n)} - \lambda_k^{(n+1)}}
\]
for $1 \leq j \leq n$ and
\[
  R_{k,n+1} = \frac{1}{1 - \lambda_k^{(n+1)}}.
\]
If we write $\widetilde{R} = R S$, where $S$ is the diagonal matrix with entries
 $S_{jj} = (\overline{\mu_j^{(n)}} \lambda_j^{(n)})^{-1}$ ($1 \leq j \leq n$) and 
$S_{n+1,n+1} = \nu_n - 1$, then we see that the rows of $\widetilde{R}$ are the representations of 
the eigenvectors $f_1^{(n+1)}, \dotsc, f_{n+1}^{(n+1)}$ in the basis $f_1^{(n)}, \dotsc, f_n^{(n)}, e_{n+1}$ up
 to constants, and therefore orthogonal. We conclude that $\widetilde{R}$, and thus $R$, are invertible, so
 $Rv = w$ has a unique solution, which can be written in terms of the eigenvalues $\lambda_j^{(m)}$ for $1 \leq j \leq m$ and $m \leq {n+1}$. Thus we conclude that $\abs{\mu_j^{(n)}}$ and $\nu_n$ are measurable functions of $\lambda_j^{(m)}$ for $1 \leq j \leq m$ and $m \leq {n+1}$ as was to be shown.
\end{proof}

For $1 \leq j \leq n$, we define the phase $\phi_j^{(n)}$ by 
 $\mu_j^{(n)} = \phi_j^{(n)} \abs{\mu_j^{(n)}}$, and 
the $\sigma$-algebras $\mathcal{B}_n = \mathcal{A} \vee \sigma\{\phi_j^{(m)}
 \mid 1 \leq m \leq n-1, 1 \leq j \leq m\}$ and $\mathcal{B} = \vee_{n=1}^\infty \mathcal{B}_n$.
\begin{lemma} \label{lem:sigmaB}
    The $\sigma$-algebra $\mathcal{B}_n$ is equal, up to completion, to the $\sigma$-algebra 
generated by $\mathcal{A}$ and the eigenvectors $f_j^{(m)}$ for $1 \leq j \leq m$ and $1 \leq m \leq n$.
\end{lemma}

\begin{proof}
Again by Theorem~\ref{thm:spectrum}, we can write the 
eigenvectors of $u_{n+1}$ as functions of $\mu_j^{(n)} = 
\phi_j^{(n)} \abs{\mu_j^{(n)}}$ ($1 \leq j \leq n$), $\nu_n$, $\lambda_j^{(m)}$ ($1 \leq j \leq m$,
 $m \in \{n,n+1\}$), and the eigenvectors of $u_n$. Clearly,
 $\abs{\mu_j^{(n)}}$, $\nu_n$, and $\lambda_j^{(m)}$ for $m \in \{n,n+1\}$ are $\mathcal{A}$-measurable
 and  $\phi_j^{(n)}$ is $\mathcal{B}_{n+1}$-measurable.

Conversely, we write
\[
  \langle (h_k^{(n+1)})^{1/2} f_k^{(n+1)}, f_j^{(n)} \rangle = 
\frac{\phi_j^{(n)} \abs{\mu_j^{(n)}}}{\lambda_j^{(n)} - \lambda_k^{(n+1)}}
\]
to find each $\phi_j^{(n)}$ as a function of the other variables.
\end{proof}


\section{Convergence of the eigenangles} \label{sec:eigenvalues}
In order to prove the convergence of the normalized eigenangles of $u_n$ when $n$ goes to infinity, we need the following lemma. 
\begin{lemma} \label{lem:angleestimate}
    Let $\epsilon > 0$. Then, almost surely under
the Haar measure on $U^\infty$, for $n \geq 1$ and $0 < k  \leq n^{1/4}$, we have
    \[
      \frac{\theta_k^{(n+1)} \abs{\mu_k^{(n)}}^2}{\theta_k^{(n)} - \theta_k^{(n+1)}} =
 1 + O(k n^{-\frac13 + \eps})
    \]
and for $n \geq 1$ and $-n^{1/4} \leq  k \leq 0$,
\[
\frac{\theta_k^{(n+1)} \abs{\mu_{k}^{(n)}}^2}{\theta_{k}^{(n)} - \theta_k^{(n+1)}}
= \frac{\theta_k^{(n+1)} \abs{\mu_{k+n}^{(n)}}^2}{\theta_{k}^{(n)} - \theta_k^{(n+1)}}
 = 1 + O((1+|k|)
n^{-\frac13 + \eps}),
\] 
\end{lemma}

\begin{remark} The implied constant in the $O(\cdot)$ notation depends on $(u_m)_{m \geq 1}$ and 
$\eps$: in particular, it is a random variable. However, for given $(u_m)_{m \geq 1}$ and 
$\eps$, it does not depend on $k$ and $n$. 
\end{remark}
\begin{proof}
By symmetry of the situation, we can assume $k > 0$. Moreover, let us fix $\eps \in (0, 0.01)$. We will suppose that the event 
$E := E_0 \, \cap \, E_1 \, \cap \, E_2, \cap 
\, E_3$ holds, where
\begin{align*}
    E_0 &= \{\theta_0^{(1)} \neq 0 \} \cap \{ \forall n \geq 1, \nu_n \neq 0 \} \cap \{\forall n \geq 1, 1 \leq k \leq n, \mu_k^{(n)} \neq 0\} \\
    E_1 &= \{\exists n_0 \geq 1, \forall n \geq n_0, \abs{\nu_n} \leq n^{-\frac12+\eps}\} \\
    E_2 &= \{\exists n_0 \geq 1, \forall n \geq n_0, 1 \leq k \leq n, \abs{\mu_k^{(n)}} \leq n^{-\frac12+\eps}\} \\
    E_3 &= \{\exists n_0 \geq 1, \forall n \geq n_0, k \geq 1, n^{-\frac53 - \eps} \leq \theta_{k+1}^{(n)} - \theta_k^{(n)} \leq n^{-1 + \eps}\}.
\end{align*}
It is possible to do this assumption, since  by the result proven in the appendix of the present paper, the event $E$ occurs almost surely. As we will see now, this \emph{a priori} information on the distribution of the eigenvalues of the random virtual isometry implies strong quantitative bounds on the change in eigenvalues of successive unitary matrices.

Recall from Theorem~\ref{thm:spectrum} that
\[
  \sum_{j=1}^n \frac{\lambda_j^{(n)} \abs{\mu_j^{(n)}}^2}{\lambda_j^{(n)} - \lambda_k^{(n+1)}} + 
\frac{\abs{1 - \nu_n}^2}{1 - \lambda_k^{(n+1)}} = 1 - \overline{\nu}_n.
\]
By using the $n$-periodictiy of $\lambda_j^{(n)}, \mu_j^{(n)}$, $f_j^{(n)}$ with
respect to $j$, 
we can write 
\begin{equation}
  \sum_{j \in J} \frac{\lambda_j^{(n)} \abs{\mu_j^{(n)}}^2}{\lambda_j^{(n)} - \lambda_k^{(n+1)}} + 
\frac{\abs{1 - \nu_n}^2}{1 - \lambda_k^{(n+1)}} = 1 - \overline{\nu}_n, \label{qwertyuiop}
\end{equation}
where $J$ is the random set of $n$ consecutive integers, such that $\theta_k^{(n+1)} - \pi
< \theta_j^{(n)} \leq \theta_k^{(n+1)} + \pi$.
 Iterating the lower bound on the distance between adjacent eigenvalues, given by the definition of 
the event $E_3$, we get, 
for $j \in J \backslash \{k-1,k\}$, 
\[
  \abs{\theta_j^{(n)} - \theta_k^{(n+1)}} \gtrsim \abs{k - j} n^{-\frac53 - \eps},
\]
and then 
\[
  \abs{\lambda_j^{(n)} - \lambda_k^{(n+1)}} \gtrsim \abs{k - j} n^{-\frac53 - \eps}, 
\]
since $\abs{\theta_j^{(n)} - \theta_k^{(n+1)}} \leq \pi$. 

Likewise, we have by $E_3$, $1 - \lambda_k^{(n+1)} = O(k n^{-1 + \eps})$, and by $E_2$, 
$\abs{\mu_j^{(n)}}^2 = O(n^{-1+2 \eps})$, which gives, for 
$j \in J \backslash \{ k-1, k \}$,
\[
  \frac{\lambda_j^{(n)} (1 - \lambda_k^{(n+1)})\abs{\mu_j^{(n)}}^2}{\lambda_j^{(n)} - \lambda_k^{(n+1)}}
 \lesssim \frac{k}{\abs{k - j}} n^{-\frac13+ 4\eps}.
\]
 Summing for $j$ in $J \backslash \{ k-1, k \}$, which is included in 
the interval $[k-1-n, k+n]$, gives 
\[
  \sum_{j \in J \backslash \{ k-1, k \}} \frac{\lambda_j^{(n)} (1 - \lambda_k^{(n+1)})\abs{\mu_j^{(n)}}^2}{\lambda_j^{(n)} - \lambda_k^{(n+1)}}
 = O (k n^{-\frac13+ 4\eps} \log n) = O(k n^{-\frac13+ 5\eps}).
\]
Now, subtracting this equation from the product of \eqref{qwertyuiop} by $1 - \lambda_k^{(n+1)}$, 
and bounding $\nu_n = O(n^{-\frac12 + \eps})$ (by the
property $E_1$) gives us the resulting equation
\[
\frac{\lambda_k^{(n)} (1 - \lambda_k^{(n+1)}) \abs{\mu_k^{(n)}}^2}{\lambda_k^{(n)} - \lambda_k^{(n+1)}} {\bf 1}_{k \in J} + \frac{\lambda_{k-1}^{(n)} (1 - \lambda_k^{(n+1)}) \abs{\mu_{k-1}^{(n)}}^2}{\lambda_{k-1}^{(n)} - \lambda_k^{(n+1)}} {\bf 1}_{k-1 \in J} = - 1 + O(kn^{-\frac13+ 5 \eps}).
\]

Next we estimate the first two terms in terms of the eigenangles. We find
\[
  1 - \lambda_k^{(n+1)} = - i \theta_k^{(n+1)} + O((\theta_k^{(n+1)})^2)
\]
and
\[
  \lambda_j^{(n)} - \lambda_k^{(n+1)} = i(\theta_j^{(n)} - \theta_k^{(n+1)}) \lambda_j^{(n)} + O((\theta_j^{(n)} - \theta_k^{(n+1)})^2)
\]
for $j = k-1$, $k$. Collecting terms and using the trivial bounds gives

\begin{multline}
  \frac{\theta_k^{(n+1)} \abs{\mu_k^{(n)}}^2}{\theta_k^{(n)} - \theta_k^{(n+1)}} \left(
1 + O(kn^{-1 + \eps}) \right) {\bf 1}_{k \in J} \\
+ \frac{\theta_k^{(n+1)} \abs{\mu_{k-1}^{(n)}}^2}{\theta_{k-1}^{(n)} - \theta_k^{(n+1)}}
\left(
1 + O(kn^{-1 + \eps}) \right) {\bf 1}_{k-1 \in J}  
= 1 + O(kn^{-\frac13 + 5 \eps}). \label{qwertyuiop2}
\end{multline}
From Theorem~\ref{thm:spectrum}, the eigenvalues of $u_n$ and $u_{n+1}$ interlace, so for $n$ sufficiently large the real part of the first term is positive and the real part of the second term is negative. The real part of the right hand side tends to $1$ as $n$ grows with $k$ fixed, so the first term has real part bounded below for $n$ sufficiently large. In particular,
\[
\frac{\theta_k^{(n+1)} \abs{\mu_k^{(n)}}^2}{\theta_k^{(n)} - \theta_k^{(n+1)}} \gtrsim 1.
\]
Using the a priori bounds for $\theta_k^{(n+1)}$ and $\abs{\mu_k^{(n)}}^2$, we find
\[
  \theta_k^{(n)} - \theta_k^{(n+1)} \lesssim k n^{-2+3\eps}.
\]
Hence,
\[
  \theta_k^{(n+1)} - \theta_{k-1}^{(n)} = (\theta_{k}^{(n)} - \theta_{k-1}^{(n)}) -
(\theta_k^{(n)} - \theta_k^{(n+1)}) \gtrsim n^{-\frac53 - \eps} - O(kn^{-2+3\eps}) \gtrsim n^{-\frac53 - \eps},
\]
since $ kn^{-2+3\eps}/ n^{-\frac53 - \eps} = O(n^{1/4 - 2 + 0.03 + 5/3 + 0.01}) = o(1)$.
We deduce that the second term of \eqref{qwertyuiop2} is dominated by $k n^{-1/3 + 4\eps}$, and then

\[
  \frac{\theta_k^{(n+1)} \abs{\mu_k^{(n)}}^2}{\theta_k^{(n)} - \theta_k^{(n+1)}} = 1 + O(k n^{-\frac13 + 5 \eps}).
\]
Changing the value of $\eps$ appropriately gives the desired result. 
\end{proof}


This lemma is enough for us to estimate the change in $\theta_k^{(n)}$ as $n$ grows, and in particular to find a limit for the renormalized angle.
\begin{theorem} \label{thm:eigenvalues}
There is a sine-kernel point process $(y_k)_{k \in \Z}$ such that almost surely,
\[
  \frac{n}{2 \pi} \theta_k^{(n)} = y_k + O((1+  k^2) n^{-\frac13 + \eps}),
\]
for all $n \geq 1$, $|k| \leq n^{1/4}$ and $\eps > 0$, where the implied constant may depend on $(u_m)_{m \geq 1}$
and $\eps$, but not on $n$ and $k$.  
\end{theorem}
\begin{proof}
The proof proceeds exactly as in \cite{BNN12}. It is sufficient to 
prove the result for $\eps$ equal to the inverse of an integer: hence, it is enough to show the estimate for 
fixed $\eps$. By symmetry, one can take $k > 0$. 
We rearrange the equation in Lemma~\ref{lem:angleestimate} to find
\[
  \abs{\mu_k^{(n)}}^2 = \left( \frac{\theta_k^{(n)}}{\theta_k^{(n+1)}} - 1\right) (1 + O(kn^{-\frac13 + \eps}))
\]
Because almost surely, $\abs{\mu_k^{(n)}}^2 = O(n^{-1+2 \eps})$, we get
\[
\abs{\mu_k^{(n)}}^2 = \frac{\theta_k^{(n)}}{\theta_k^{(n+1)}} - 1 + O(k n^{-\frac43 + 3 \eps}).
\]
Using the asymptotic $\log(1 - \delta) = -\delta + O(\delta^2)$ for $\delta = o(1)$, we conclude, if 
$\eps$ is small enough,
\[
  \log \frac{\theta_k^{(n)}}{\theta_k^{(n+1)}} = \abs{\mu_k^{(n)}}^2 + O(k n^{-\frac43 + 3\eps}).
\]
Define the random variable $L_k^{(n)} = \log \theta_k^{(n)} + \sum_{j=1}^{n-1} \abs{\mu_k^{(j)}}^2$; we have just shown $L_k^{(n+1)} - L_k^{(n)} = O(k n^{-\frac43 + 3 \eps})$ so 
for $k$ fixed, $L_k^{(n)}$
 converges to a limit $L_k^{(\infty)}$ almost surely as $n \to \infty$, with $\abs{L_k^{(n)} - L_k^{(\infty)}} = O(k n^{-\frac13 + 3 \eps})$. Now,
\begin{align*}
  \exp L_k^{(n)} &= \theta_k^{(n)} \exp \sum_{j=k}^{n-1} \abs{\mu_k^{(j)}}^2 \\
  &= n \theta_k^{(n)} \exp\left(- \log n + \sum_{j=1}^{n-1} \frac1j + \sum_{j=1}^{n-1} (\abs{\mu_k^{(j)}}^2 - \frac1j) \right)
\end{align*}
Recall $-\log n + \sum_{j=1}^{n-1} \frac1j = \gamma + O(n^{-1})$ where $\gamma$ is the Euler-Mascheroni 
constant. Next we define
\[
  M_k^{(n)} := \sum_{j=1}^{n-1} \left( \abs{\mu_k^{(j)}}^2 - \frac1j \right)
\]
and observe that each term of the sum is an independent mean-zero random variable. Therefore, 
for $k$ fixed, $(M_k^{(n)})_{n \geq k}$ is a martingale. We claim that $M_k^{(n)}$ is bounded in $L^2$; in fact,
\[
  \mathbb{E} (\abs{\mu_k^{(n)}}^2 - \frac1n)^2 = O(n^{-2}).
\]
so that
\[
  \mathbb{E} ((M_k^{(\infty)} - M_k^{(n)})^2) = \sum_{j \geq n} \mathbb{E} (\abs{\mu_k^{(n)}}^2 - \frac1n)^2 = O(n^{-1}),
\]
where $M_{k}^{(\infty)}$ is the claimed limit of $M_k^{(n)}$ (this limit exists since $M_k^{(n)}$ is a sum of centered and independent random variables with summable variances). To see this, we write
\[
  \abs{\mu_k^{(n)}}^2 = \frac{e_1}{e_1 + \cdots + e_n}
\]
where the variables $e_r$ are independent standard exponential random variables. Then we compute
\[
  \mathbb{E} (\abs{\mu_k^{(n)}}^2 - \frac1n)^2 = \mathbb{E} \left( \frac{(n-1) e_1 - e_2 - \cdots - e_n}{n(e_1 + \cdots + e_n)} \right)^2.
\]
As shown before in this paper, 
$\mathbb{P}(e_1 + \cdots + e_n \leq \frac{n}{2}) = O(n^{-C})$ for all $C \geq 2$ so that
\begin{align*}
  \mathbb{E} (\abs{\mu_k^{(n)}}^2 - \frac1n)^2 &\leq O(n^{-C}) + \frac4{n^4} \mathbb{E} (((n-1) e_1 - e_2 - \cdots - e_n)^2) \\
  &\leq O(n^{-2})
\end{align*}
Now, by the triangle inequality and Doob's maximal inequality, for $q$ positive integer, $k \leq 2^q$,
\begin{align*}
  \mathbb{E} (\sup_{n \geq 2^q} (M_k^{(\infty)} - M_k^{(n)})^2) &\lesssim \mathbb{E} 
((M_k^{(\infty)} - M_k^{(2^q)})^2) + \mathbb{E} (\sup_{n \geq 2^q} (M_k^{(n)} - M_k^{(2^q)})^2) \\
  &\lesssim \mathbb{E} (M_k^{(\infty)} - M_k^{(2^q)})^2 \\
  &= O(2^{-q}).
\end{align*}
Hence, 
\begin{align*}
 \mathbb{E} \left[ \sup_{2^q \leq n \leq 2^{q+1}} \sup_{k \leq n^{1/4}} 
(M_k^{(\infty)} - M_k^{(n)})^2 \right]
& \leq  \mathbb{E} \left[ \sup_{k \leq 2^{(q+1)/4}} \sup_{2^q \leq n \leq 2^{q+1}} 
(M_k^{(\infty)} - M_k^{(n)})^2 \right]
\\ & \leq \sum_{k \leq 2^{(q+1)/4}}  \mathbb{E}  \left[ \sup_{2^q \leq n \leq 2^{q+1}} 
(M_k^{(\infty)} - M_k^{(n)})^2 \right]
\\ & \lesssim  2^{(q+1)/4} 2^{-q} = O( 2^{-3q/4})
\end{align*}
and 
\begin{align*}
 \mathbb{E} \left[ \sup_{ n \geq 2^{q}} \sup_{k \leq n^{1/4}} 
(M_k^{(\infty)} - M_k^{(n)})^2 \right]
& \leq \sum_{r \geq q} \mathbb{E} \left[ \sup_{2^r \leq n \leq 2^{r+1}} \sup_{k \leq n^{1/4}} 
(M_k^{(\infty)} - M_k^{(n)})^2 \right]
\\ & \lesssim \sum_{r \geq q} 2^{-3r/4} = O (2^{-3q/4}). 
\end{align*}
By Markov's inequality, we get 
\[
\mathbb{P} (\sup_{n \geq 2^q}  \sup_{k \leq n^{1/4}}  \abs{M_k^{(\infty)} - M_k^{(n)}} \geq 2^{-q/3}) \leq 2^{2q/3} \mathbb{E} (\sup_{n \geq 2^q} \sup_{k \leq n^{1/4}} (M_k^{(\infty)} - M_k^{(n)})^2) = O(2^{-q/12}),
\]
which, by Borel-Cantelli lemma, shows that almost surely for some $q_0 \geq 1$,
 all $q \geq q_0$, $n \geq 2^q$ and $k \leq n^{1/4}$ satisfy
 $|M_k^{(\infty)} - M_k^{(n)}| \leq 2^{-q/3}$. Hence, 
\[
  \abs{M_k^{(\infty)} - M_k^{(n)}} = O(n^{-\frac13})
\]
almost surely. Collecting these estimates and applying them to the equation
\[
  \exp L_k^{(n)} = n \theta_k^{(n)} \exp(\gamma + O(n^{-1}) + M_k^{(n)})
\]
gives us
\[
  \exp \left( L_k^{(\infty)} + O(k n^{-\frac13+3 \eps}) \right) = n \theta_k^{(n)}
 \exp(\gamma + M_k^{(\infty)} + O(n^{-\frac13}))
\]
Rearranging,
\[
  n \theta_k^{(n)} = \exp( L_k^{(\infty)} - M_k^{(\infty)} - \gamma) (1 + O(kn^{-\frac13 + 3 \eps}))
 =: 2 \pi y_k (1  + O(k n^{-\frac13 + 
3 \eps}).
\]
Now, by \cite{BNN12}, $(y_k)_{k \in \mathbb{Z}}$ is a determinantal sine-kernel process, so we have 
almost surely the estimate $y_k = O(1 + |k|)$, which proves Theorem \ref{thm:eigenvalues}. 
\end{proof}

\section{Weak convergence and renormalization of the eigenvectors} \label{sec:eigenvectors}

We are now ready to show that the eigenfunctions $f_k^{(n)}$ of $u_n$ converge in a suitable sense. We 
 assume that for a given value of $\eps$, 
 the event $E$ from Section~\ref{sec:eigenvalues} holds, which happens almost surely. Recall 
from Theorem~\ref{thm:spectrum} that we can choose
 representatives $f_k^{(n)}$ for the eigenvectors of each $u_n$ in such a way that
\[
  (h_k^{(n+1)})^{\frac12} f_k^{(n+1)} = \sum_{j=1}^n \frac{\mu_j^{(n)}}{\lambda_j^{(n)} - \lambda_k^{(n+1)}} f_j^{(n)} + \frac{\nu_n - 1}{1 - \lambda_k^{(n+1)}} e_{n+1}
\]
where
\[
  h_k^{(n+1)} = \sum_{j=1}^n \frac{\abs{\mu_j^{(n)}}^2}{\abs{\lambda_j^{(n)} - \lambda_k^{(n+1)}}^2} + \frac{\abs{\nu_n - 1}^2}{\abs{1 - \lambda_k^{(n+1)}}^2}.
\]
For $n = 1$, we adopt the convention $f_1^{(1)} = - e_1$. We deduce, for $n \geq \ell$,
\[
  \langle f_k^{(n+1)}, e_{\ell} \rangle = (h_k^{(n+1)})^{-\frac12} \sum_{j=1}^n \frac{\mu_j^{(n)}}{\lambda_j^{(n)} - \lambda_k^{(n+1)}} \langle f_j^{(n)}, e_{\ell} \rangle.
\]
The invariance by conjugation of the Haar measures implies that each eigenvector $f_k^{(n+1)}$, multiplied by 
an independent random phase of modulus $1$, is a uniform vector on the complex sphere $S^{n+1}$. Hence, 
the scalar product $\langle f_k^{(n+1)}, e_\ell \rangle$ converges
 to zero in probability. In order to get a limit which is different from zero, we need to 
consider a suitable normalization. We introduce the following eigenvectors, for $n \geq k$: 
\[
  g_k^{(n)} := D_k^{(n)} f_k^{(n)},
\]
where $D_k^{(n)} \in \C$ is the random variable
\[
  D_k^{(n)} = \prod_{s=k}^{n-1} (h_k^{(s+1)})^{\frac12} \frac{\lambda_k^{(s)} - \lambda_k^{(s+1)}}{\mu_k^{(s)}}.
\]
We claim that for each renormalized eigenvector $g_k^{(n)}$, the scalar product
 $\langle g_k^{(n)}, e_\ell \rangle$ converges.
\begin{theorem} \label{thm:vector}
For each $k \geq 1$ and $\ell \geq 1$, the sequence $\{\langle g_k^{(n)}, e_\ell \rangle\}_{n \geq k \vee \ell}$
 is a martingale with respect to the filtration $(\mathcal{B}_n)_{n \geq  k \vee \ell}$,
and the conditional expectation of $|\langle g_k^{(n)}, e_\ell \rangle|^2$, given $\mathcal{A}$, is almost surely
 bounded when $n$ varies. 
\end{theorem}
\begin{remark}
More precisely, $\{\langle g_k^{(n)}, e_\ell \rangle\}_{n \geq k \vee \ell}$ is a martingale in the following 
sense: for $n \geq  k \vee \ell$, the conditional expectation of $\langle g_k^{(n+1)}, e_\ell \rangle$ 
given $\mathcal{B}_n$ is almost surely well-defined and equal to $\langle g_k^{(n)}, e_\ell \rangle$. 
However, we do not claim the integrability of $\langle g_k^{(n)}, e_\ell \rangle$ without conditioning. 
\end{remark}

\begin{proof}
From the equation above,
\begin{align}
  \langle g_k^{(n+1)}, e_{\ell} \rangle &= D_k^{(n+1)} (h_k^{(n+1)})^{-\frac12} \sum_{j=1}^n \frac{\mu_j^{(n)}}{\lambda_j^{(n)} - \lambda_k^{(n+1)}} \langle f_j^{(n)}, e_{\ell} \rangle \nonumber \\
  &= D_k^{(n)}  \frac{\lambda_k^{(n)} - \lambda_k^{(n+1)}}{\mu_k^{(n)}}
\sum_{j=1}^n \frac{\mu_j^{(n)}}{\lambda_j^{(n)} - \lambda_k^{(n+1)}} \langle f_j^{(n)}, e_{\ell} \rangle \nonumber \\
  &= \langle g_k^{(n)}, e_{\ell} \rangle + 
D_k^{(n)}  \frac{\lambda_k^{(n)} - \lambda_k^{(n+1)}}{\mu_k^{(n)}}
\sum_{\substack{1 \leq j \leq n \\ j \neq k}} \frac{\mu_j^{(n)}}{\lambda_j^{(n)} - \lambda_k^{(n+1)}} \langle f_j^{(n)}, e_{\ell} \rangle \label{recurrenceg}
\end{align}

Now, recall from Lemma~\ref{lem:sigmaB} that the sequence of eigenvectors $(f_k^{(n)})_{n \geq k}$ is adapted 
to the filtration of $\sigma$-algebras $(\mathcal{B}_n)_{n \geq k}$. If we decompose
\[
  \mu_j^{(n)} = \phi_j^{(n)} \abs{\mu_j^{(n)}}
\]
then  for $n$ fixed, $\{\phi_j^{(n)}\}_{1 \leq j \leq n}$ is a family of iid random phases uniformly distributed on the unit circle, independent of the $\sigma$-algebra $\mathcal{B}_n$. Indeed, $\mathcal{B}_n$ is the 
$\sigma$-algebra generated by $\mathcal{A}$ and $\{\phi_j^{(m)}\}_{1 \leq m \leq n-1, 1 \leq j \leq m}$, and then, by Lemma 
 \ref{lem:sigmaA}, it is the $\sigma$-algebra generated by $u_1$, $(\nu_m)_{m \geq 1}$, $(|\mu_j^{(m)}|)_{m \geq 1, 1 \leq j \leq m}$, and $(\phi_j^{(m)})_{1 \leq m \leq n-1, 1 \leq j \leq m}$. All these variables are 
 independent of  $(\phi_j^{(n)})_{1 \leq j \leq n}$, which are iid uniform on the unit circle, since the vectors $((\mu_j^{(m)})_{1 \leq j \leq m}, \nu_m)$ are independent,  uniform on the unit complex sphere of $\mathbb{C}^{m+1}$.

Now, since $h_k^{(n+1)}$ is real, we have
\[
  \Phi_k^{(n)} := \frac{D_k^{(n)}}{\abs{D_k^{(n)}}} = \prod_{s=k}^{n-1} (\phi_k^{(s)})^{-1} \frac{\lambda_k^{(s)} - \lambda_k^{(s+1)}}{\abs{\lambda_k^{(s)} - \lambda_k^{(s+1)}}}
\]
for $1 \leq k \leq n$, and we can write
\[
  \langle g_k^{(n+1)} - g_k^{(n)}, e_\ell \rangle = \Phi_k^{(n)} \abs{D_k^{(n)}} \frac{\lambda_k^{(n)} - \lambda_k^{(n+1)}}{\abs{\mu_k^{(n)}}} \sum_{\substack{1 \leq j \leq n \\ j \neq k}} \frac{\phi_j^{(n)}}{\phi_k^{(n)}} \frac{\abs{\mu_j^{(n)}}}{\lambda_j^{(n)} - \lambda_k^{(n+1)}} \langle f_j^{(n)}, e_{\ell} \rangle.
\]
We would like to compute the conditional expectation of the difference $\langle g_k^{(n+1)} - g_k^{(n)}, 
e_\ell \rangle$, given the $\sigma$-algebra $\mathcal{B}_n$. We first verify the measurability of each quantity on the right; in particular,
\begin{enumerate}
  \item $\abs{D_k^{(n)}}$, $\frac{\lambda_k^{(n)} - \lambda_k^{(n+1)}}{\abs{\mu_k^{(n)}}}$, $\frac{\abs{\mu_j^{(n)}}}{\lambda_j^{(n)} - \lambda_k^{(n+1)}}$ are $\mathcal{A}$-measurable.
  \item $\Phi_k^{(n)}$ and $\langle f_j^{(n)}, e_{\ell} \rangle$ are $\mathcal{B}_n$-measurable.
  \item $\{\phi_j^{(n)} (\phi_k^{(n)})^{-1}\}_{\substack{1 \leq j \leq n \\ j \neq k}}$ are iid and independent of $\mathcal{B}_n$.
\end{enumerate}
We also have:
\[
  \abs{\langle g_k^{(n+1)} - g_k^{(n)}, e_\ell \rangle} \leq \abs{D_k^{(n)}}  \frac{\abs{\lambda_k^{(n)} - \lambda_k^{(n+1)}}}{\abs{\mu_k^{(n)}}} \sum_{\substack{1 \leq j \leq n \\ j \neq k}} \frac{\abs{\mu_j^{(n)}}}{\abs{\lambda_j^{(n)} - \lambda_k^{(n+1)}}}
\]
which gives an $\mathcal{A}$-measurable bound for the scalar product $\langle g_k^{(n+1)} - g_k^{(n)}, e_\ell \rangle$. Since this bound is almost surely finite, the conditional expectation is almost surely well-defined. By Fubini's theorem and measurability we get
\begin{multline*}
  \mathbb{E} [\langle g_k^{(n+1)} - g_k^{(n)}, e_\ell \rangle \mid \mathcal{B}_n] = \Phi_k^{(n)} \abs{D_k^{(n)}} \frac{\lambda_k^{(n)} - \lambda_k^{(n+1)}}{\abs{\mu_k^{(n)}}} \\ \times \sum_{\substack{1 \leq j \leq n \\ j \neq k}} \mathbb{E} \left[\frac{\phi_j^{(n)}}{\phi_k^{(n)}} \mid \mathcal{B}_n \right] \frac{\abs{\mu_j^{(n)}}}{\lambda_j^{(n)} - \lambda_k^{(n+1)}} \langle f_j^{(n)}, e_{\ell} \rangle.
\end{multline*}
However, by independence,
\[
  \mathbb{E} \left[\frac{\phi_j^{(n)}}{\phi_k^{(n)}} \mid \mathcal{B}_n \right] = \mathbb{E}
 \left[ \frac{\phi_j^{(n)}}{\phi_k^{(n)}} \right]= 0,
\]
so $(\langle g_k^{(n)}, e_\ell \rangle)_{n \geq k \vee \ell}$ is a $(\mathcal{B}_n)_{n \geq k \vee \ell}$-martingale.

To check its conditional boundedness in $L^2$, given $\mathcal{A}$, we need to show
\[
   \mathbb{E} [ \abs{\langle g_k^{(k \vee \ell)}, e_{\ell}  \rangle}^2 \, | \mathcal{A} ] + \sum_{n \geq k \vee \ell} \mathbb{E} [ \abs{\langle g_k^{(n+1)} - g_k^{(n)}, e_\ell \rangle}^2 \, | \mathcal{A} ] < \infty
\]
almost surely. The first term is smaller than or equal to $\|g_k^{(k \vee \ell)}\|^2 
= |D_k^{(k \vee \ell)}|^2$, which is $\mathcal{A}$-measurable and almost surely finite. Hence, it is 
sufficient to bound the sum. Expanding it gives: 
\begin{equation}
  \mathbb{E} [ \abs{\langle g_k^{(n+1)} - g_k^{(n)}, e_\ell \rangle}^2 \, 
| \mathcal{A} ] = \abs{D_k^{(n)}}^2 \frac{\abs{\lambda_k^{(n)} - \lambda_k^{(n+1)}}^2}{\abs{\mu_k^{(n)}}^2} S,
\label{Mn+1}
\end{equation}
where
\[
  S = \sum_{\substack{1 \leq i, j \leq n \\ i,j \neq k}} \frac{\abs{\mu_i^{(n)}}}{\overline{\lambda_i^{(n)} - \lambda_k^{(n+1)}}} \frac{\abs{\mu_j^{(n)}}}{\lambda_{j}^{(n)} - \lambda_k^{(n+1)}}
  \mathbb{E} [ \overline{\phi_i^{(n)}}{\phi_j^{(n)}} \overline{\langle f_i^{(n)}, e_{\ell} \rangle} \langle f_j^{(n)}, e_{\ell} \rangle \mid \mathcal{A} ].
\]
Now
\begin{align*}
    \mathbb{E} [ \overline{\phi_i^{(n)}}{\phi_j^{(n)}} \overline{\langle f_i^{(n)}, e_{\ell} \rangle} \langle f_j^{(n)}, e_{\ell} \rangle \mid \mathcal{A} ] &= \mathbb{E} [ \mathbb{E} [\overline{\phi_i^{(n)}}{\phi_j^{(n)}} \mid \mathcal{B}_n ] \overline{\langle f_i^{(n)}, e_{\ell} \rangle} \langle f_j^{(n)}, e_{\ell} \rangle \mid \mathcal{A} ] \\
    &= \delta_{i,j} \mathbb{E}[ \overline{\langle f_i^{(n)}, e_{\ell} \rangle} \langle f_j^{(n)}, e_{\ell} \rangle \mid \mathcal{A}]
\end{align*}
where $\delta_{i,j}$ is the Kronecker delta. Thus
\begin{equation}
  S = \sum_{\substack{1 \leq j \leq n \\ j \neq k}} \frac{\abs{\mu_j^{(n)}}^2}{\abs{\lambda_{j}^{(n)} - \lambda_k^{(n+1)}}^2} \mathbb{E} [\abs{\langle f_j^{(n)}, e_\ell \rangle}^2 \mid \mathcal{A}].
\label{formulaS}
\end{equation}
In order to effectively bound this sum we need a posteriori information from the convergence of the eigenvalues in Section~\ref{sec:eigenvalues}. 
We define the event $F_k$ to be
\[
  F_k := \{ \exists n_0 \geq 1, \forall n \geq n_0, (\theta_{k+1}^{(n)} - \theta_k^{(n)}) \wedge (\theta_k^{(n)} - \theta_{k-1}^{(n)}) \geq n^{-1-\eps} \}
\]
\begin{lemma}
  $F_k$ is $\mathcal{A}$-measurable and holds with probability one.
\end{lemma}
\begin{proof}
$F_k$ depends only on the eigenangles (hence eigenvalues) and is therefore clearly $\mathcal{A}$-measurable. By Theorem~\ref{thm:eigenvalues} applied to $k-1, k, k+1$, we see that each of the associated eigenangles satisfies
\[
  n \theta_j^{(n)} = 2 \pi y_j + O(n^{-\frac14}).
\]
 In particular,
\[
  \theta_k^{(n)} - \theta_{k-1}^{(n)} = 2 \pi n^{-1} (y_k - y_{k-1}) + O(n^{-\frac54}).
\]
Since $\{y_j\}_{j \in \Z}$ is a sine-kernel point process $y_k - y_{k-1} \gtrsim 1$ almost surely; similar for $k+1$ and $k$.
\end{proof}

Now we condition on $F_k$ and estimate the quantities 
involved in $\mathbb{E} [ \abs{\langle g_k^{(n+1)} - g_k^{(n)}, e_\ell \rangle}^2 \, | \mathcal{A} ]$. We will
 begin by estimating $D_k^{(n)}$. By Lemma~\ref{lem:angleestimate}, almost surely,
\[
  \abs{\theta_k^{(n+1)} - \theta_k^{(n)}} = \theta_k^{(n+1)} \abs{\mu_k^{(n)}}^2 (1 + O(n^{-\frac13 + \eps})),
\]
hence
\begin{equation}
  \abs{\lambda_k^{(n+1)} - \lambda_k^{(n)}} = \theta_k^{(n+1)} \abs{\mu_k^{(n)}}^2 (1 + O(n^{-\frac13 + \eps})).
\label{differencelambda}
\end{equation}
Similarly,
\[
  \frac{\abs{\nu_n - 1}^2}{\abs{1 - \lambda_k^{(n+1)}}^2} = (\theta_k^{(n+1)})^{-2} (1 + O(n^{-\frac12+\eps})).
\]
Now, as in the proof of convergence of the eigenangles, let us consider the set $J$ of indices 
$j$  such that $\theta_k^{(n+1)} - \pi
< \theta_j^{(n)} \leq \theta_k^{(n+1)} + \pi$.
From the event $F_k$, we see that for $j \in J \backslash \{k\}$, and $n$ large enough, 
\begin{align*}
  \abs{\theta_j^{(n)} - \theta_k^{(n+1)}} &  \geq \abs{\theta_k^{(n)} 
- \theta_j^{(n)}} - \abs{\theta_k^{(n)} - \theta_k^{(n+1)}}  \\ & \geq n^{-1-\eps} - 
  O \left( \theta_k^{(n+1)} \abs{\mu_k^{(n)}}^2 (1 + O(n^{-\frac13 + \eps})) \right) 
\\ & \geq n^{-1-\eps} - O(n^{-2 + \eps}) \gtrsim n^{-1-\eps},
\end{align*} 
and then 
\begin{equation}
\abs{\lambda_j^{(n)} - \lambda_k^{(n+1)}} \gtrsim n^{-1-\eps}, \label{differencelambda2}
\end{equation}
if $\eps$ is taken small enough. 
We can get a stronger estimate for $\abs{j - k}$ sufficiently large. In fact, since $E_3$ holds, we have
\begin{equation}
  \abs{\lambda_j^{(n)} - \lambda_k^{(n+1)}} \gtrsim \abs{j - k} n^{-\frac53-\eps}. \label{differencelambda3}
\end{equation}
These two lower bounds let us estimate
\begin{align*}
  \sum_{\substack{1 \leq j \leq n \\ j \neq k}}
 \frac{\abs{\mu_j^{(n)}}^2}{\abs{\lambda_{j}^{(n)} - \lambda_k^{(n+1)}}^2} &\lesssim n^{-1+\eps}
\sum_{\substack{j \in J \\ j \neq k}} \abs{\lambda_{j}^{(n)} - \lambda_k^{(n+1)}}^{-2} \\
  &\lesssim n^{-1+\eps} (n^{2 + \eps} n^{\frac23} + n^{\frac{10}3 + \eps} n^{-\frac23}) \\
  &\lesssim n^{\frac53 + \eps}.
\end{align*}
where we split the sum into the interval with $\abs{j - k} \leq n^{\frac23}$ and its complement.

Now we can estimate $h_k^{(s+1)}$. In fact,
\begin{align*}
 &  h_k^{(s+1)} \frac{\abs{\lambda_k^{(s)} - \lambda_k^{(s+1)}}^2}{\abs{\mu_k^{(s)}}^2} \\ &= 1
 + \frac{\abs{\lambda_k^{(s)} - \lambda_k^{(s+1)}}^2}{\abs{\mu_k^{(s)}}^2} \left( \sum_{\substack{1
 \leq j \leq n \\ j \neq k}} \frac{\abs{\mu_j^{(s)}}^2}{\abs{\lambda_j^{(s)} - 
\lambda_k^{(s+1)}}^2} + \frac{\abs{\nu_s - 1}^2}{\abs{1 - \lambda_k^{(s+1)}}^2} \right) \\
&= 1 + \frac{\abs{\lambda_k^{(s)} - \lambda_k^{(s+1)}}^2}{\abs{\mu_k^{(s)}}^2} 
 \left(O(s^{\frac53+\eps}) 
+ (\theta_k^{(s+1)})^{-2} (1 + O(s^{-\frac12+\eps}))\right) 
\end{align*}
Now, since almost surely, $\theta_k^{(s+1)} = O(1/s)$, $s^{\frac53+\eps} = O((\theta_k^{(s+1)})^{-2} 
s^{-\frac13 + \eps})$, and then
\[
h_k^{(s+1)} \frac{\abs{\lambda_k^{(s)} - \lambda_k^{(s+1)}}^2}{\abs{\mu_k^{(s)}}^2}
 =  1 + \frac{(\theta_k^{(s+1)})^{-2}\abs{\lambda_k^{(s)} - \lambda_k^{(s+1)}}^2}{\abs{\mu_k^{(s)}}^2} 
\left( 1 + O(s^{-\frac13+\eps})\right).
\]
Now, \begin{align*}
\abs{\lambda_k^{(s)} - \lambda_k^{(s+1)}}^2 &  = 
\abs{\theta_k^{(s)} - \theta_k^{(s+1)}}^2 \left(  1 + O(|\theta_k^{(s)} - \theta_k^{(s+1)}|) \right)
\\ & = \abs{\theta_k^{(s)} - \theta_k^{(s+1)}}^2 \left(  1 + O(s^{-1}) \right),
\end{align*}
Using Lemma~\ref{lem:angleestimate}, one obtains:
\begin{align*}
 h_k^{(s+1)} \frac{\abs{\lambda_k^{(s)} - \lambda_k^{(s+1)}}^2}{\abs{\mu_k^{(s)}}^2}
& = 1 +  \frac{(\theta_k^{(s+1)})^{-2}\abs{\theta_k^{(s)} - \theta_k^{(s+1)}}^2}{\abs{\mu_k^{(s)}}^2} 
\left( 1 + O(s^{-\frac13+\eps})\right).
\\ & = 1 + \abs{\mu_k^{(s)}}^2 \left( 1 + O(s^{-\frac13+\eps})\right).
\end{align*}
Applying the bound on $\abs{\mu_k^{(s)}}^2$ given by $E_2$ and changing the value of $\eps$ gives 
\[
h_k^{(s+1)} \frac{\abs{\lambda_k^{(s)} - \lambda_k^{(s+1)}}^2}{\abs{\mu_k^{(s)}}^2} = 
1 + \abs{\mu_k^{(s)}}^2  + O(s^{-\frac43+\eps}).
\]
Thus from the expression
\[
  \abs{D_k^{(n)}}^2 = \prod_{s=k}^{n-1} h_k^{(s+1)} \frac{\abs{\lambda_k^{(s)} - \lambda_k^{(s+1)}}^2}{\abs{\mu_k^{(s)}}^2},
\]
we deduce
\begin{align*}
  \abs{D_k^{(n)}}^2 &= \prod_{s=k}^{n-1} (1 + \abs{\mu_k^{(s)}}^2 + O(s^{-\frac43 + \eps})) \\
  &= \exp \left( \sum_{s=k}^{n-1} \frac1s \right) \exp \left( \sum_{s=k}^{n-1} \abs{\mu_k^{(s)}}^2 - \frac1s \right) \exp \sum_{s=k}^{n-1} O(s^{-\frac43+\eps}).
\end{align*}
As before,
\[
  \exp  \sum_{s=k}^{n-1} \frac1s = k^{-1} n \exp [\gamma (1 + O(n^{-1}))]
\]
where $\gamma$ is the Euler-Mascheroni constant,
\[
  \exp \sum_{s=k}^{n-1} \left( \abs{\mu_k^{(s)}}^2 - \frac1s \right)
\]
is equal to $\exp(M_k^{(n-1)} - M^{(k)}_k)$ from Section~\ref{sec:eigenvalues}, and the last
 term converges to a limit $N_\infty$ with error $O(n^{-\frac13 + \eps})$.

Thus, we have
\begin{align}
  \abs{D_k^{(n)}}^2 &= k^{-1} n \exp (\gamma + M_k^{(\infty)} - M_k^{(k)} + N_\infty) (1 + O(n^{-\frac13 + \eps})) 
\nonumber \\
  &=: D_k n (1 + O(n^{-\frac13 + \eps})). \label{normalizationD}
\end{align}
where $D_k$ is a non-zero random variable that depends only on $k$.

We are now ready to estimate $\mathbb{E} [ \abs{\langle g_k^{(n+1)} - g_k^{(n)}, e_\ell \rangle}^2 \, | \mathcal{A} ]$. In fact we have
\begin{align*}
\mathbb{E} [ \abs{\langle g_k^{(n+1)} - g_k^{(n)}, e_\ell \rangle}^2 \, | \mathcal{A} ]
&= \abs{D_k^{(n)}}^2 \frac{\abs{\lambda_k^{(n)} - \lambda_k^{(n+1)}}^2}{\abs{\mu_k^{(n)}}^2} S \\
&\lesssim n \abs{\mu_k^{(n)}}^2 (\theta_k^{(n+1)})^2 S \\
&\lesssim n^{-2+\eps} S
\end{align*}
and
\begin{align*}
    S &= \sum_{\substack{1 \leq j \leq n \\ j \neq k}} \frac{\abs{\mu_j^{(n)}}^2}{\abs{\lambda_{j}^{(n)} - \lambda_k^{(n+1)}}^2} \mathbb{E} [\abs{\langle f_j^{(n)}, e_\ell \rangle}^2 \mid \mathcal{A}] \\
    &\lesssim n^{-1+\eps} \sum_{\substack{1 \leq j \leq n \\ j \neq k}} \abs{\lambda_{j}^{(n)} - \lambda_k^{(n+1)}}^{-2} \mathbb{E} [\abs{\langle f_j^{(n)}, e_\ell \rangle}^2 \mid \mathcal{A}] \\
    &\lesssim n^{-1+\eps} \sum_{j \in J, 0 < \abs{j - k} < n^{\frac23}} n^{2+\eps} \mathbb{E} [\abs{\langle f_j^{(n)}, e_\ell \rangle}^2 \mid \mathcal{A}] \\
    &\qquad + n^{-1+\eps} \sum_{j \in J, \abs{j - k} \geq n^{\frac23}} \abs{j-k}^{-2} n^{\frac{10}{3}+\eps} \mathbb{E} [\abs{\langle f_j^{(n)}, e_\ell \rangle}^2 \mid \mathcal{A}]
\end{align*}
Therefore, it is now sufficient to prove that for some $\eps > 0$, and almost surely, 
\begin{align*}
 & \sum_{n \geq k}
n^{-3+\eps} \left( \sum_{j \in J, 0 < \abs{j - k} < n^{\frac23}} n^{2+\eps} \mathbb{E}
 [\abs{\langle f_j^{(n)}, e_\ell \rangle}^2 \mid \mathcal{A}] \right.  \\
    & \left. \qquad + \sum_{j \in J, \abs{j - k} \geq n^{\frac23}} \abs{j-k}^{-2} n^{\frac{10}{3}+\eps} 
\mathbb{E} [\abs{\langle f_j^{(n)}, e_\ell \rangle}^2 \mid \mathcal{A}] \right) < \infty.
\end{align*}
It is then sufficient to prove that the expectation of the left-hand side is finite. Since $J
\subset [k-n-1, k+n]$, one deduces that it is enough to have
\begin{align*}
 & \sum_{n \geq k}
n^{-3+\eps} \left( \sum_{ 0 < \abs{j - k} < n^{\frac23}} n^{2+\eps} \mathbb{E}
 [\abs{\langle f_j^{(n)}, e_\ell \rangle}^2] \right.  \\
    & \left. \qquad + \sum_{j \in J, n^{\frac23} \leq  \abs{j - k} \leq n+1} \abs{j-k}^{-2} n^{\frac{10}{3}+\eps} 
\mathbb{E} [\abs{\langle f_j^{(n)}, e_\ell \rangle}^2] \right) < \infty.
\end{align*}
Now, since $f_j^{(n)}$ is, up to a phase of modulus $1$, uniform on the sphere $S^{n}$, one has 
$$\mathbb{E} [\abs{\langle f_j^{(n)}, e_\ell \rangle}^2] = 1/n,$$
and then one needs only to check: 
$$\sum_{n \geq k}
n^{-3+\eps} \left( \sum_{ 0 < \abs{j - k} < n^{\frac23}} n^{1+\eps} 
   + \sum_{j \in J, n^{\frac23} \leq  \abs{j - k} \leq n+1} \abs{j-k}^{-2} n^{\frac{7}{3}+\eps} 
 \right) < \infty,$$
which is easy. 
\end{proof}
Because this martingale is bounded in $L^2$, we have the following immediate corollary.
\begin{corollary} \label{cor:gkl}
 Almost surely, for all $k \in \Z$ and $\ell \geq 1$, the scalar product $\langle g_k^{(n)}, e_\ell \rangle$ 
converges to a limit $g_{k, \ell}$ when $n$ goes to infinity. 
\end{corollary}
For each $k  \in \Z$, the infinite sequence $g_k := (g_{k, \ell})_{\ell \geq 1} \in \mathbb{C}^{\infty}$
 can be considered as the weak limit of the eigenvector $g_k^{(n)}$ of $u_n$, when $n$ goes to infinity. 
At the end of the paper, we construct a random operator whose eigenvectors are given 
by the sequences $g_k$, $k \in \Z$. In order to ensure that these eigenvectors are nontrivial, we need
 to check that $g_k$ is almost surely non vanishing. This will be done in the next section. 

\section{The law of the coefficients of the eigenvectors of the flow} \label{sec:normalizationsqrtn}

In the previous section, we have proven the almost sure convergence of each coordinate of the 
eigenvectors of $u_n$, after normalization by a factor $D_k^{(n)}$. Using the estimate 
\eqref{normalizationD}, we see that almost surely, $|D_k^{(n)}|$ is equivalent to a nonzero constant 
times $\sqrt{n}$ when $n$ goes to infinity. It is then more elegant to formulate the 
result of convergence 
by taking eigenvectors of norm exactly $\sqrt{n}$. Moreover, such a normalization provides the 
distribution of the limiting coordinates of the eigenvectors. 
The full statement we prove is the following: 
\begin{theorem}
Let $(u_n)_{n \geq 1}$ be a virtual isometry, following the Haar measure. 
For $k \in \mathbb{Z}$ and $n \geq 1$, let $v_k^{(n)}$ be a unit eigenvector corresponding 
to the $k$th smallest nonnegative eigenangle of $u_n$ for $k \geq 1$, and the 
$(1-k)$th largest strictly negative eigenangle of $u_n$ for $k \leq 0$. 
Then for all $k \in \mathbb{Z}$, there almost surely exist some 
 complex numbers $(\psi_k^{(n)})_{n \geq 1}$ of modulus 1, and 
 a sequence $(t_{k,\ell})_{\ell \geq 1}$, such that for all $\ell \geq 1$, 
 $$ \sqrt{n} \, \langle \psi_k^{(n)} v_k^{(n)}, e_{\ell} \rangle \underset{n \rightarrow \infty}
 {\longrightarrow}
 t_{k,\ell}.$$
Almost surely, for all $k \in \mathbb{Z}$, the sequence $(t_{k,\ell})_{\ell \geq 1}$ depends, up
 to a multiplicative factor of modulus one, only on the virtual rotation $(u_n)_{n \geq 1}$. Moreover, 
 if  $(\psi_k)_{k \in \mathbb{Z}}$ is a sequence of iid,
uniform variables on $\mathbb{U}$, independent of $(t_{k,\ell})_{\ell \geq 1}$, then 
 $(\psi_k t_{k,\ell})_{k \in \mathbb{Z}, \ell \geq 1}$ is an iid family of 
 standard complex gaussian variables ($\mathbb{E} [|\psi_k t_{k,\ell}|^2] = 1$). 
\end{theorem}
\begin{remark}
 The vectors $v_k^{(n)}$ are equal to $f_k^{(n)}$, up to a multiplicative factor of modulus $1$. 
The independent phases $\psi_k$ introduced in the last part of the theorem are needed in 
order to get iid complex gaussian variables. This is not the case, for example, if we normalize 
$(t_{k,\ell})_{\ell \geq 1}$ in such a way that $t_{k,1} \in \mathbb{R}_+$. 
\end{remark}

\begin{proof}
For fixed $k \in \mathbb{Z}$, let us first show the existence
 of the sequence $(t_{k,\ell})_{\ell \geq 1}$. By symmetry, one can assume $k \geq 1$: in
 this case, for all $n \geq k$, there exists $\tau_k^{(n)}$ of modulus $1$ such that 
$f_k^{(n)} = \tau_k^{(n)} v_k^{(n)}$. Hence, 
 $$ |D_k^{(n)}| \langle \Phi_k^{(n)} \tau_k^{(n)} v_k^{(n)}, e_{\ell} \rangle$$
 almost surely converges when $n$ goes to infinity. Now, by the estimate \eqref{normalizationD}, 
 $|D_k^{(n)}| / \sqrt{n}$ converges almost surely to a strictly positive constant. 
 One deduces the existence of $(t_{k,\ell})_{\ell \geq 1}$, by taking 
$$\psi_k^{(n)} :=  \Phi_k^{(n)} \tau_k^{(n)}.$$
Let us now check the uniqueness, by supposing that two sequences $(t_{k,\ell})_{\ell \geq 1}$ and
$(t'_{k,\ell})_{\ell \geq 1}$ can be constructed from the same virtual rotation $(u_n)_{n \geq 1}$.
In this case, there exist, for all $n \geq 1$, two unit eigenvectors $w_k^{(n)}$ and 
$w_k^{(n)'}$ corresponding to the same eigenvalue, and such that for all $\ell \geq 1$, 
\begin{equation}
 \sqrt{n} \, \langle w_k^{(n)}, e_{\ell} \rangle \underset{n \rightarrow \infty}
 {\longrightarrow}
 t_{k,\ell} \label{wkn1}
 \end{equation}
 and 
 $$ \sqrt{n} \, \langle w_k^{(n)'}, e_{\ell} \rangle \underset{n \rightarrow \infty}
 {\longrightarrow}
 t'_{k,\ell}.$$
Since the eigenvalues are almost surely simple, for all $n \geq 1$, there exists 
$\chi_k^{(n)} \in \mathbb{U}$ such that $w_k^{(n)'}=\chi_k^{(n)}w_k^{(n)}$, which implies
\begin{equation}
 \chi_k^{(n)} \sqrt{n} \, \langle w_k^{(n)}, e_{\ell} \rangle \underset{n \rightarrow \infty}
 {\longrightarrow}
 t'_{k,\ell}. \label{wkn2}
 \end{equation}
By comparing \eqref{wkn1} and \eqref{wkn2}, one deduces that $t'_{k, \ell} = \chi_k t_{k, \ell}$, 
where $\chi_k \in \mathbb{U}$ denotes the limit of any converging subsequence of $(\chi_k^{(n)})_{n \geq 1}$. 

Moreover, let us choose the random vectors $(w_k^{(n)})_{k \in \mathbb{Z},n \geq 1}$ and the
random variables $(t_{k,\ell})_{k \in \mathbb{Z}, \ell \geq 1}$ as measurable functions
of $(u_n)_{n \geq 1}$, in such a way that \eqref{wkn1} is satisfied almost surely. 
Let $(\psi_k)_{k \in \mathbb{Z}}$ be iid random variables, independent of 
$(u_n)_{n \geq 1}$. For all $n \geq 1$, the invariance by conjugation of the
 Haar measure on $U(n)$ implies that the family of eigenvectors 
$(\psi_k w_k^{(n)})_{1 \leq k \leq n}$ of $u_n$ forms a Haar-distributed unitary matrix 
in $U(n)$. 
One deduces that if $L$ is a finite set of strictly positive integers, and 
if $K$ is a finite set of integers, then 
$$\sqrt{n} ( \langle \psi_k w_k^{(n)}, e_{\ell} \rangle)_{k \in K, \ell \in L}$$
converges in law to a family of iid standard complex gaussian variables. 
Since this family of variables also converges almost surely, one deduces that the
limiting variables $(\psi_k t_{k,\ell})_{k \in K, \ell \in L}$ are iid standard complex and
gaussian. Since the finite sets $K$ and $L$ can be taken arbitrarily, we are done. 
\end{proof}
One deduces immediately the following: 
\begin{corollary}
 The limiting coordinates $g_{k, \ell}$ introduced in Corollary \ref{cor:gkl} are almost surely different from 
zero. 
\end{corollary}
\begin{proof}
We know that for a suitable normalization of $t_{k, \ell}$, 
$$ g_{k, \ell} = \underset{n \rightarrow \infty}{\lim}
|D_k^{(n)}| \langle \Phi_k^{(n)} f_k^{(n)}, e_{\ell} \rangle,$$
  $$t_{k, \ell} = \underset{n \rightarrow \infty}{\lim}
\sqrt{n} \langle \Phi_k^{(n)} f_k^{(n)}, e_{\ell} \rangle,$$
and by \eqref{normalizationD}, 
$$\sqrt{D_k} =  \underset{n \rightarrow \infty}{\lim} |D_k^{(n)}|/\sqrt{n}.$$
Combining these limits gives
\begin{equation}
g_{k, \ell} = \sqrt{D_k} t_{k, \ell}, \label{gdkt}
\end{equation}
which is almost surely nonzero, since $t_{k, \ell}$ is a standard complex gaussian variable up to multiplication by an independent uniform random variable on the unit circle. 
\end{proof}

The eigenspaces of $u_n$, generated by the vectors $f_k^{(n)}$, 
can also be considered as elements of the projective space
$\mathbb{P}^{n-1}(\mathbb{C})$. Moreover, one can define the infinite-dimensional 
projective space $\mathbb{P}^{\infty}(\mathbb{C})$, as the space of nonzero infinite 
sequences of complex numbers, quotiented by scalar multiplication.
The convergence of renormalized eigenvectors proven above can be viewed as a convergence of 
the corresponding points on the projective spaces. 

There exists a uniform measure on all these projective spaces, obtained by taking 
the equivalence class of a sequence of iid standard complex gaussian variables. 
For $n \geq 1$ finite, the uniform measure on $\mathbb{P}^{n}(\mathbb{C})$ can also be obtained from 
a uniform point on the sphere in $\mathbb{C}^{n+1}$. 
For $m < n \in \mathbb{N} \cup \{\infty\}$, there exists a natural projection $\Pi_{n,m}$ from 
$\mathbb{P}^{n}(\mathbb{C})$ to $\mathbb{P}^{m}(\mathbb{C})$, obtained by taking only the $m+1$ first 
coordinates of the sequences, and this projection is well-defined when these coordinates 
are not all vanishing: in particular, almost surely under the uniform measure 
on $\mathbb{P}^{n}(\mathbb{C})$. Note that the image of this measure by $\Pi_{n,m}$ is the 
uniform measure on $\mathbb{P}^{m}(\mathbb{C})$. 
Moreover, we can define the notion of weak convergence on the projective spaces as follows.
Let $(x_n)_{n \geq 1}$ be a sequence such that $x_n \in \mathbb{P}^{n}(\mathbb{C})$ for all 
$n \geq 1$ and let $x_{\infty} \in \mathbb{P}^{\infty}(\mathbb{C})$. We say
 that  $(x_n)_{n \geq 1}$ weakly converges to $x_{\infty}$ if and only if the following holds:
 for all $m \geq 1$ such that the $m+1$ first coordinates of $x_{\infty}$ are not 
 all vanishing, the projection $\Pi_{n,m}(x_n) \in \mathbb{P}^{m} (\mathbb{C})$ 
 is well-defined for $n$ large enough and tends to $\Pi_{\infty,m}(x_{\infty})$ 
 when $n$ goes to infinity. 
 From the previous result, we can easily deduce the following: 
 \begin{theorem}
 Let $(u_n)_{n \geq 1}$ be a virtual rotation, following the Haar measure. 
For $k \in \mathbb{Z}$ and $n \geq 1$, let $x_k^{(n)} \in \mathbb{P}^{n-1}(\mathbb{C})$ be the
 eigenspace corresponding to the $k$th smallest nonnegative eigenangle of $u_n$ for $k \geq 1$, and the 
$(1-k)$th largest strictly negative eigenangle of $u_n$ for $k \leq 0$ (this eigenspace is 
almost surely one-dimensional). Then, 
there almost surely exists some random points $(x_k^{(\infty)})_{k \in \mathbb{Z}}$
in $\mathbb{P}^{\infty}(\mathbb{C})$ such that 
for all $k \in \mathbb{Z}$, $x_k^{(n)}$ weakly converges to $x_k^{(\infty)}$ when $n$ goes
 to infinity. The points $(x_k^{(\infty)})_{k \in \mathbb{Z}}$ are represented by 
 the sequences $(t_{k, \ell})_{\ell \geq 1}$, $k \in \mathbb{Z}$ given above: they are independent and 
 uniform on $\mathbb{P}^{\infty}(\mathbb{C})$.
 \end{theorem}
\section{A flow of operators on a random space}  \label{sec:flowoperators}
For each $\alpha \in \mathbb{R}$, let $(\alpha_n)_{n \geq 1}$ be a sequence such that 
$\alpha_n$ is equivalent to $\alpha n$ when $n$ goes to infinity. For $n \geq 1$, $k \in \Z$, 
we have
$$u_n^{\alpha_n} f_k^{(n)} = e^{i \theta_k^{(n)} \alpha_n}  f_k^{(n)}.$$
Now, $e^{i \theta_k^{(n)} \alpha_n}$ tends to $e^{2 i \pi \alpha y_k}$ and 
after normalization, the coordinates of $f_k^{(n)}$ tend to the corresponding coordinates of the 
sequence $(t_{k,\ell})_{\ell \geq 1}$. It is then natural to expect that, in a sense which needs to be 
made precise, $u_n^{\alpha_n}$ tends to some operator $U$, acting on some infinite sequences, such that
$$U  ((t_{k, \ell})_{\ell \geq 1}) = e^{2 i \pi \alpha y_k} (t_{k, \ell})_{\ell \geq 1}.$$
This motivates the following definition: 
\begin{definition}
 The space $\mathcal{E}$ is the random vector subspace of $\mathbb{C}^{\infty}$, consisting of all finite linear combinations of the sequences $(t_{k, \ell})_{\ell \geq 1}$, or equivalently, $(g_{k, \ell})_{\ell \geq 1}$, for 
$k \in \Z$. 
\item For $\alpha \in \mathbb{R}$, the operator $U^{\alpha}$ is the unique linear application from 
$\mathcal{E}$ to $\mathcal{E}$ such that for all $k \in \Z$, 
$$U^{\alpha} ((t_{k, \ell})_{\ell \geq 1}) = e^{2 i \pi \alpha y_k} (t_{k, \ell})_{\ell \geq 1},$$
or equivalently, 
$$U^{\alpha} ((g_{k, \ell})_{\ell \geq 1}) = e^{2 i \pi \alpha y_k} (g_{k, \ell})_{\ell \geq 1}.$$
\end{definition}
\begin{remark}
 For each $k$, the sequence $(t_{k, \ell})_{\ell \geq 1}$ is almost surely well-defined up to a multiplicative
constant of modulus $1$. Moreover, the sequences $(t_{k, \ell})_{\ell \geq 1}$ for 
$k \in \Z$ are a.s. linearly independent, since for a suitable normalization,
$(t_{k, \ell})_{k \in \Z, \ell \geq 1}$ are iid standard complex gaussian. This ensures that the 
definition of $U^{\alpha}$ given above is meaningful. The notation $U^{\alpha}$ is motivated by 
the immediate fact that $(U^{\alpha})_{\alpha \in \mathbb{R}}$ is a flow of operators on 
$\mathcal{E}$, i.e. $U^0 = I_{\mathcal{E}}$ and $U^{\alpha + \beta} = U^{\alpha} U^{\beta}$ for
all $\alpha, \beta \in \mathbb{R}$.   
\end{remark}
As suggested before, we expect that $U^{\alpha}$ is a kind of limit for $u_n^{\alpha_n}$ when 
$n$ goes to infinity. Of course, these operators do not act on the same space, so we need to 
be more precise. 
\begin{theorem} \label{thm:flowoperators}
 Almost surely, for any sequence $(s_{\ell})_{\ell \geq 1}$ in $\mathcal{E}$ and for all integers
$m \geq 1$, 
$$\left[ u_n^{\alpha_n} ((s_{\ell})_{1 \leq \ell \leq n} )\right]_m 
\underset{n \rightarrow \infty}{\longrightarrow} \left[ U^{\alpha} ((s_{\ell})_{\ell \geq 1}) \right]_m,$$
where $[\, \cdot \,]_m$ denotes the $m$th coordinate of a vector or a sequence. 
\end{theorem}
By linearity, it is sufficient to show the theorem for $(s_{\ell})_{\ell \geq 1} = (g_{k, \ell})_{\ell \geq 1}$.
 Hence, Theorem \ref{thm:flowoperators} 
can be deduced from the following proposition: 
\begin{proposition} \label{prop:eigenvectorconvergence}
For all $k \in \mathbb{Z}$, $\ell \geq 1$, one has almost surely: 
\[
\langle u_n^{\alpha_n} (g_k[n]), e_\ell \rangle \to e^{2 \pi i \alpha y_k}  g_{k, \ell}
\]
as $n \to \infty$, for $g_k[n] := (g_{k, \ell})_{1 \leq \ell \leq n}$. 
\end{proposition}
In the next section, we will give a more intrinsic way to define a flow of operators similar to 
$(U^{\alpha})_{\alpha \in \mathbb{R}}$. In order to make this construction, we need 
a more precise and stronger result than Proposition \ref{prop:eigenvectorconvergence}, which is given by the 
 following two propositions: 
\begin{proposition} \label{prop:distanceeigenvectors}
Let $\eps > 0$. Almost surely, for all $k \in \mathbb{Z}$, we have the following.
\begin{enumerate}
 \item The euclidian norm $\|g_k[n]\|$ is equivalent to a strictly positive random variable times $\sqrt{n}$, when 
$n$ goes to infinity. 
\item $\|g_k[n] - g_k^{(n)} \| = O_\eps(n^{\frac{1}{3} + \eps})$.
\item For any $T > 0$ and $\delta \in (0, 1/6)$, 

$$\underset{\alpha \in [-T,T]}{\sup} \; \underset{\alpha_n \in [n(\alpha - n^{-\delta}), 
n(\alpha + n^{-\delta})]}{\sup} 
\|u_n^{\alpha_n} g_k[n] - e^{2 \pi i \alpha y_k} g_k[n]\| =  O(n^{\frac{1}{2} - \delta}).$$
\end{enumerate}
\end{proposition}
\begin{proposition} \label{prop:distancecomponentseigenvectors}
 Almost surely, for all $k \in \mathbb{Z}$, $\ell \geq 1$, $\alpha, \gamma \in \mathbb{R}$, and 
for all sequences $(\alpha_n)_{n \geq 1}$ and $(\gamma_n)_{n \geq 1}$ such that 
$\alpha_n/n = \alpha + o(n^{-\delta})$ and $\gamma_n/n = \gamma  + o(n^{-\delta})$ for some $\delta 
\in [0,1/6)$, 
$$ \langle u_n^{\alpha_n} (g_k[n]) -  e^{2 \pi i \alpha y_k} g_k[n], u_n^{\gamma_n} (e_\ell) \rangle 
= o (n^{-\delta}),$$
when $n$ goes to infinity. Moreover, for $\delta \in (0,1/6)$, we get the uniform estimate: 
$$ \sup_{\substack{\alpha_n \in [n(\alpha - n^{-\delta}), 
n(\alpha + n^{-\delta})] \\ \gamma_n \in [n(\gamma - n^{-\delta}), 
n(\gamma + n^{-\delta})]}}
 \langle u_n^{\alpha_n} (g_k[n]) -  e^{2 \pi i \alpha y_k} g_k[n], u_n^{\gamma_n} (e_\ell) \rangle 
= O (n^{-\delta}).$$
\end{proposition}
\begin{proof} We now prove Propositions \ref{prop:distanceeigenvectors}
  and \ref{prop:distancecomponentseigenvectors}: it is clear that this last proposition 
implies Proposition \ref{prop:eigenvectorconvergence} (by taking $\delta = \gamma_n = 0$). 

Using \eqref{gdkt}, we get the following: 
$$\|g_{k}[n]\|^2 = |D_k| \sum_{\ell = 1}^n |t_{k, \ell}|^2$$
where $(t_{k, \ell})_{\ell \geq 1}$ are up to independent and uniform random phases iid standard complex gaussian variables. By the law of 
large numbers, $\|g_k[n]\|^2$ is a.s.~equivalent to $n |D_k|$ when $n$ goes to infinity, which shows 
the first item of Proposition \ref{prop:distanceeigenvectors}. 

Now, the third item can be quickly  deduced from the second one. Indeed, if we have the estimate 
$\|g_k[n] - g_k^{(n)} \| = O (n^{\frac{1}{3} + \eps})$ for all $\eps > 0$, then, since $\delta < 1/6$, we
also have $\|g_k[n] - g_k^{(n)} \| = O(n^{\frac{1}{2} - \delta})$, which implies, for 
all $\alpha \in [-T,T]$ and $\alpha_n \in [n(\alpha - n^{-\delta}), 
n(\alpha + n^{-\delta})]$, 
\begin{align*}
& \|u_n^{\alpha_n} g_k[n] - e^{2 \pi i \alpha y_k} g_k[n]\|
\\ & \leq   \|u_n^{\alpha_n} (g_k[n] - g_k^{(n)}) \| + \|u_n^{\alpha_n} g_k^{(n)} - e^{2 \pi i \alpha y_k} g_k^{(n)}\|
+\| e^{2 \pi i \alpha y_k} (g_k[n] - g_k^{(n)}) \|
\\ & = 2 \|g_k[n] - g_k^{(n)}\| + |e^{i \alpha_n \theta_{k}^{(n)} } -e^{2 \pi i \alpha y_k}  | \, \| g_k^{(n)}\|.
\end{align*}
Now, $$\|g_k[n] - g_k^{(n)}\| = O(n^{\frac{1}{2} - \delta}),$$
 $$\|g_k^{(n)} \| \leq \|g_k[n]\| +  \|g_k[n] - g_k^{(n)}\|
= O(\sqrt{n}) + O(n^{\frac{1}{2} - \delta}) = O(\sqrt{n}),$$
and 
\begin{align*}
|e^{i \alpha_n \theta_{k}^{(n)} } -e^{2 \pi i \alpha y_k}  |
& \leq |\alpha_n \theta_k^{(n)} - 2 \pi \alpha y_k|
\\ & \leq  |\theta_k^{(n)}| |\alpha_n - \alpha n| + 
|\alpha|  |n \theta_k^{(n)} - 2 \pi y_k| 
\\ & = O(1/n) O(n^{1 - \delta}) + O(n^{-1/4}) = O(n^{-\delta}),
\end{align*} 
where the implied constant does not depend on $\alpha$ and $\alpha_n$ (recall that $|\alpha|$ is assumed to be 
uniformly bounded by $T$). Note that 
we used Theorem \ref{thm:eigenvalues} (for $k$ fixed and $\eps = 1/12$) at the last step of the computation.
This gives: 
$$ \|u_n^{\alpha_n} g_k[n] - e^{2 \pi i \alpha y_k} g_k[n]\| = 
O(n^{\frac{1}{2} - \delta}),$$
uniformly with respect to $\alpha$ and $\alpha_n$, i.e. the third item of Proposition 
\ref{prop:distanceeigenvectors}. In order to complete the proof of this proposition, it then 
remains to show the second item, which needs several intermediate steps.
 
 \begin{lemma} \label{L2bounds}
For fixed $k \geq 1$, $\eps > 0$, there exists a $\mathcal{A}$-measurable random variable $M > 0$, 
such that, almost surely, for all $N \geq n \geq k$, 
$$\mathbb{E} [\| g_k^{(n)} - g_k[n] \|^2 \, | \mathcal{A} ] \leq M \,  n^{\frac{2}{3} + \eps},$$
and
$$\mathbb{E} [\|g_k^{(n)} - g_{k,n}^{(N)}\|^2 \, | \mathcal{A} ] \leq M \, (N-n) \, n^{-\frac13 + \eps},$$
where $g_{k,n}^{(N)}$ denotes the vector obtained by taking the $n$ first coordinates of 
 $g_k^{(N)}$.
 \end{lemma}
 \begin{proof}
 One has, for $n \geq 1$, and $N \in \{n,n+1,n+2,... \} \cup \{\infty\}$,
 $$ \mathbb{E} [\| g_k^{(n)} - g_{k,n}^{(N)}\|^2 \, | \mathcal{A} ]
 = \sum_{\ell = 1}^n \mathbb{E} [| \langle g_k^{(n)}, e_{\ell} \rangle - \langle g_k^{(N)}, e_{\ell} \rangle |^2
  \, |\mathcal{A}],$$
for $g_{k,n}^{(\infty)} := g_k [n]$. 
  By the martingale property satisfied by the scalar products $\langle g_k^{(m)}, e_{\ell} \rangle$ for
    $m \geq k$,
   $$ \mathbb{E} [| \langle g_k^{(n)}, e_{\ell} \rangle - \langle g_k^{(N)}, e_{\ell} \rangle |^2
  \, |\mathcal{A}] = \sum_{n \leq m < N}
   \Delta_{\ell}^{(m)} $$
   where
     $$  \Delta_{\ell}^{(m)} = \mathbb{E} [| \langle g_k^{(m+1)}, e_{\ell} \rangle - \langle g_k^{(m)}, e_{\ell} \rangle |^2
  \, |\mathcal{A}], $$
   and then by \eqref{Mn+1} and \eqref{formulaS},
 \begin{align}
  \Delta_{\ell}^{(m)} & =
|D_k^{(m)}|^2  \frac{|\lambda_k^{(m)} - \lambda_k^{(m+1)}|^2}{|\mu_k^{(m)}|^2}
\sum_{1 \leq j \leq m, \, j \neq k} \frac{|\mu_{j}^{(m)}|^2}{|\lambda_{j}^{(m)} 
- \lambda_k^{(m+1)}|^2} \mathbb{E} \left[ | \langle f_{j}^{(m)}, e_{\ell} \rangle |^2 \,
 | \mathcal{A} \right] \nonumber \\ 
& \leq  M_1 \, m^{-3 + \eps} \sum_{j \in J_m \backslash \{k\}}
|\lambda_{j}^{(m)} - \lambda_k^{(m+1)}|^{-2}\mathbb{E} \left[ | \langle f_{j}^{(m)}, e_{\ell} 
\rangle |^2 \, | \mathcal{A} \right] \label{estimateDelta}
\end{align}
where $J_m$ is the random set of $m$ consecutive integers, such that $\theta_k^{(m+1)} - \pi
< \theta_j^{(m)} \leq \theta_k^{(m+1)} + \pi$, and where
$$M_1 := \sup_{m \geq k, 1 \leq j \leq m, j \neq k} 
m^{3 - \eps} \, |D_k^{(m)}|^2  \frac{|\lambda_k^{(m)} - \lambda_k^{(m+1)}|^2 \, |\mu_{j}^{(m)}|^2}{|\mu_k^{(m)}|^2}.$$
From the estimates \eqref{differencelambda}, \eqref{normalizationD}, the fact that 
$|\mu_{j}^{(m)}|^2$ and $|\mu_{k}^{(m)}|^2$ are almost surely dominated by $m^{-1+ \frac{\eps}{2}}$ and the 
bound $ \theta_{k}^{(m+1)} = O(1/m)$, one deduces that the $\mathcal{A}$-measurable 
quantity $M_1$ is almost surely finite. 
Similarly, one has, for all $j \in J_m \backslash \{k\}$,
\begin{equation}|\lambda_{j}^{(m)} - \lambda_k^{(m+1)}|^{-2} \leq M_2 \, \left(m^{2+ \eps} \wedge \frac{m^{\frac{10}{3}
+ \eps}}{|k-j|^2} \right)
\leq M_2 \, \frac{m^{\frac83 + \eps}}{|k-j|}, \label{estimateDelta2}
\end{equation}
where 
$$M_2 := \sup_{m \geq k, j \in J_m \backslash \{k\}} 
|\lambda_{j}^{(m)} - \lambda_k^{(m+1)}|^{-2} \left(m^{2 + \eps} \wedge \frac{m^{\frac{10}{3} + \eps
}}{|k-j|^2} \right)^{-1}.$$
Now, $M_2$ is $\mathcal{A}$-measurable and, by the estimates \eqref{differencelambda2} and 
\eqref{differencelambda3}, it is almost surely finite. From \eqref{estimateDelta} and
\eqref{estimateDelta2}, we deduce (with the change of variable $p = j-k$):
$$\Delta_{\ell}^{(m)} \leq \, M_1 \, M_2 \, m^{-\frac 1 3 + 2\eps}
 \sum_{p \in \{-m-1,-m+1, \dots, -1,1, \dots m+1 \}} 
 \frac{1}{|p|} \mathbb{E} \left[ | \langle f_{k+p}^{(m)}, e_{\ell} \rangle |^2 \, | \mathcal{A} \right].$$
Therefore, 
 \begin{align*}
& \mathbb{E} [| \langle g_k^{(n)}, e_{\ell} \rangle - \langle g_k^{(N)}, e_{\ell} \rangle |^2
  \, |\mathcal{A}] \\ & \leq 
   M_1M_2 \, \sum_{n \leq m < N} m^{-\frac 1 3 + 2 \eps} \sum_{p \in \{-m,-m+1, \dots, -1,1, \dots m \}} 
 \frac{1}{|p|} \mathbb{E} \left[ | \langle f_{k+p}^{(m)}, e_{\ell} \rangle |^2 \, | \mathcal{A} \right],
\end{align*}
 and then, summing for $\ell$ between $1$ and $n$, 
\begin{align*}
& \mathbb{E} [\| g_k^{(n)} - g_{k,n}^{(N)}\|^2 \, | \mathcal{A} ]
\\ & \leq   M_1 M_2 \, \sum_{n \leq m < N} m^{-\frac 1 3 + 2 \eps} \sum_{p \in \{-m,-m+1, \dots, -1,1, \dots m \}} 
 \frac{1}{|p|}  \mathbb{E} \left[ \sum_{\ell=1}^{n} | \langle f_{k+p}^{(m)}, e_{\ell} \rangle |^2 \, 
| \mathcal{A} \right].
\end{align*}
Let us first suppose that $N$ is finite. One has 
$$\sum_{\ell=1}^{n} | \langle f_{k+p}^{(m)}, e_{\ell} \rangle |^2 \leq 
\sum_{\ell=1}^{m} | \langle f_{k+p}^{(m)}, e_{\ell} \rangle |^2 = \|f_{k+p}^{(m)}\|^2 = 1,$$
which implies, 
\begin{align*}
\mathbb{E} [\|\langle g_k^{(n)} - g_{k,n}^{(N)}\|^2 \, | \mathcal{A} ]
& \leq   M_1 \, M_2 \, \sum_{n \leq m < N} m^{-\frac 1 3 + 2 \eps} \sum_{p \in \{-m,-m+1, \dots, -1,1, \dots m \}}
 \frac{1}{|p|} \\ &  \leq M_3 \, \sum_{n \leq m < N} m^{-\frac 1 3 + 3 \eps} \leq M_3 (N-m)m^{-\frac 1 3 + 3 \eps}
\end{align*}
 where 
 $$M_3 := M_1 M_2 \sup_{m \geq 1} \left( m^{-\eps} \sum_{p \in \{-m,-m+1, \dots, -1,1, \dots m \}}
 \frac{1}{|p|} \right) < \infty$$
 is $\mathcal{A}$-measurable. Hence, we have the desired bound, after changing $\eps$ and taking $M = M_3$. 
 In the case where $N$ is infinite, we use Lemma \ref{lem:uniformbeta}
to get the estimate 
$$\mathbb{E} \left[| \langle f_{k+p}^{(m)}, e_{\ell} \rangle |^2 \, 
| \mathcal{A} \right] \leq M_4 m^{-1 + \eps},$$
where $$M_4 := \sup_{m \geq 1, 1 \leq k, \ell \leq m} 
m^{1-\eps} \mathbb{E} \left[| \langle f_{k}^{(m)}, e_{\ell} \rangle |^2 \, 
| \mathcal{A} \right]$$
is $\mathcal{A}$-measurable and almost surely finite. 
Hence, 
\begin{align*}
& \mathbb{E} [\| g_k^{(n)} - g_{k}[n]\|^2 \, | \mathcal{A} ]
\\ & \leq   M_1 M_2 M_4\, \sum_{m \geq n} m^{-\frac 4 3 + 3 \eps}  \sum_{p \in \{-m,-m+1, \dots, -1,1, \dots m \}} 
 \frac{n}{|p|}, 
\end{align*}
which is easily dominated by a finite, $\mathcal{A}$-measurable quantity, multiplied by
 $n^{\frac 2 3 + 4 \eps}$. 
 \end{proof}
We have now an $L^2$ bound on $\|g_k^{(n)} - g_{k,n}^{(\infty)}\|$, conditionally on $\mathcal{A}$. 
The next goal is to deduce an almost sure bound, by using Borel-Cantelli lemma. This cannot be made directly, since
the corresponding probabilities do not decay sufficiently fast, but one can solve this problem by using 
subsequences.
\begin{lemma} \label{subsequence1}
For fixed $k \geq 1$, $\eps > 0$, let $\nu := 1 + \lfloor 3/\eps \rfloor$, and for
$r \geq 1$, let $n_r := k + r^{\nu}$. Then, almost surely, 
$$\|g_k^{(n_r)} - g_{k}[n_r]\|^2 = O(n_r^{\frac{2}{3} + \eps}). $$
\end{lemma} 
\begin{proof}
From Lemma \ref{L2bounds} (applied to $\eps/2$ instead of $\eps$), there exists  
 $M'$ almost surely finite and $\mathcal{A}$-measurable 
such that for all $r \geq 1$, 
\begin{align*}
\mathbb{P} [\|g_k^{(n_r)} - g_{k}[n_r]\|^2  \geq n_r^{\frac{2}{3} + \eps} | \mathcal{A} ]
&  \leq n_r^{-\frac{2}{3} - \eps} \mathbb{E} [\|g_k^{(n_r)} - g_{k}[n_r]\|^2  | \mathcal{A} ]
\\& \leq M' n_r^{-\frac{2}{3} - \eps} n_r^{\frac{2}{3} + \frac{\eps}{2}} \leq M' n_r^{-\eps/2} 
\\ & =
 M' (k+r^{\nu})^{-\eps/2} \leq M' (r^{3/\eps})^{-\eps/2} = M'r^{-3/2}.
\end{align*}
and then
$$ \frac{1}{M'+1} \, \mathbb{E} \left[ 
  \,  \sum_{r \geq 1} {\bf 1}_{\|g_k^{(n_r)} - g_{k}[n_r]\|^2  \geq n_r^{\frac{2}{3} + \eps
}} \; | \mathcal{A}\right] \leq \frac{M'}{M'+1} \sum_{r \geq 1} r^{-3/2}  \leq 3.$$
By taking the expectation and using the fact that $1/(M'+1)$ is $\mathcal{A}$-measurable, one deduces
$$ \mathbb{E} \left[  \frac{1}{M'+1} \,  \sum_{r \geq 1} {\bf 1}_{\|g_k^{(n_r)} - g_{k}[n_r]\|^2  \geq n_r^{\frac{2}{3} + 
\eps}} \right] \leq 3,$$
and then 
$$\frac{1}{M'+1} \,  \sum_{r \geq 1} {\bf 1}_{\|g_k^{(n_r)} -  g_{k}[n_r]\|^2  \geq n_r^{\frac{2}{3} + 
\eps}} 
< \infty$$ 
almost surely. Hence, almost surely, $\|g_k^{(n_r)} - g_{k}[n_r]\|^2  \leq n_r^{\frac{2}{3} + \eps}$ 
for all but finitely many $r \geq 1$.

\end{proof}
 The next lemma gives a way to go from a given subsequence to a less sparse subsequence.
 We will say that a value $\delta \in (0,1]$ is {\it good}
 if the conclusion of  Lemma \ref{subsequence1} remains valid after replacing the sequence  
 $(k+r^{\nu})_{r \geq 1}$ by $(k+\lfloor r^{1/\delta} \rfloor)_{r \geq 1}$, the
 brackets denoting the integer part: from Lemma \ref{subsequence1}, we know that $1/\nu$ is good.
 \begin{lemma} \label{subsequence2}
 If $\delta \in [1/\nu, (\nu- 1)/\nu]$ is good, then $\delta + 1/\nu$ is good. 
 \end{lemma}
 \begin{proof}
 For $r \geq 1$, let $n_r := k + \lfloor r^{1/(\delta + 1/\nu)} \rfloor$ and let 
 $N_r$ the smallest element of the sequence $(k + \lfloor s^{1/\delta } \rfloor)_{s \geq 1}$, such that 
 $N_r \geq n_r$. The sequence $(N_r)_{r \geq 1}$, as $(n_r)_{r \geq 1}$, tends 
 to infinity with $r$: moreover, it is a subsequence of $(k + \lfloor s^{1/\delta }
 \rfloor)_{s \geq 1}$ (but some terms of 
 this sequence may appear several times in $(N_r)_{r \geq 1}$). By assumption, one
 has almost surely: $$\|g_k^{(N_r)} - g_{k}[N_r]\|^2 = O(N_r^{\frac{2}{3} + \eps}).$$
 Now, it is easy to check that $N_r = O(n_r)$ (with a constant depending only on $\delta$). Hence, by 
  restricting the vectors to their $n_r$ first coordinates, one deduces, almost surely,
 \begin{equation}
 \|g_{k,n_r}^{(N_r)} - g_{k} [n_r]\|^2 \leq \|g_k^{(N_r)} - g_{k}[N_r]\|^2 = 
 O(n_r^{\frac{2}{3} + \eps}). \label{firstterm}
 \end{equation}
 On the other hand, the distance between two consecutive terms of the sequence
 $(k + \lfloor s^{1/\delta } \rfloor)_{s \geq 1}$
 satisfies the following:
 \begin{align*}
(k + \lfloor (s+1)^{1/\delta } \rfloor) - (k + \lfloor s^{1/\delta } \rfloor )
 & = (s+1)^{1/\delta} - s^{1/\delta} + O(1) \\ &  \lesssim  s^{1/\delta  - 1} \lesssim 
(k + \lfloor s^{1/\delta } \rfloor)^{1 - \delta},
\end{align*}
which implies 
$$N_r - n_r = O(n_r^{1-\delta}),$$
where the implied constant depends only on $k$ and $\delta$. 
 One deduces, by Lemma \ref{L2bounds}:
 $$
 \mathbb{E} [\|g_k^{(n_r)} - g_{k,n_r}^{(N_r)}\|^2 \, | \mathcal{A} ] \leq M' \, (N_r-n_r) \, (n_r)^{-\frac{1}{3}
+ \frac{\eps}{2}}
   \lesssim M' n_r^{\frac{2}{3} + \frac{\eps}{2} - \delta},$$
   and then
   $$ \mathbb{P} [\|g_k^{(n_r)} - g_{k,n_r}^{(N_r)}\|^2  \geq n_r^{\frac{2}{3} + \eps}\, | \mathcal{A} ] 
   \lesssim M' n_r^{-\frac{\eps}{2}  - \delta} \lesssim M' r^{\left(-\frac{\eps}{2} - \delta \right)/
\left(\delta + \frac{1}{\nu} \right)},$$
   where the impled constant depends only on $k$ and $\delta$. 
   Since the exponent of $r$ is strictly smaller than $-1$, one deduces, by using Borel-Cantelli
  lemma similarly as in the proof of Lemma \ref{subsequence1}, that almost surely, 
   \begin{equation}
   \|g_k^{(n_r)} - g_{k,n_r}^{(N_r)}\|^2 = O(n_r^{\frac{2}{3} + \eps}). \label{secondterm}
   \end{equation}
  Combining \eqref{firstterm} and \eqref{secondterm} gives the desired result. 
  
 \end{proof}
 \noindent
By applying Lemma \ref{subsequence1} and ($\nu - 1$ times) Lemma \ref{subsequence2}, one deduces that 
$1$ is good, which gives the second item of Proposition \ref{prop:distanceeigenvectors} for $k \geq 1$: 
the situation for $k \leq 0$ is similar. 

Let us now prove Proposition \ref{prop:distancecomponentseigenvectors}. By the triangle inequality, we
get:
\begin{align*}
 & | \langle u_n^{\alpha_n} (g_k[n]) - e^{2 \pi i \alpha y_k}  g_k[n], u_n^{\gamma_n} (e_{\ell}) \rangle | 
\\ & \leq | \langle u_n^{\alpha_n} (g_k[n] - g_k^{(n)}),   u_n^{\gamma_n} (e_{\ell}) \rangle | 
+ | \langle u_n^{\alpha_n}  g_k^{(n)} - e^{2 \pi i \alpha y_k}  g_k^{(n)}, 
  u_n^{\gamma_n} (e_{\ell}) \rangle | 
\\ & + | \langle e^{2 \pi i \alpha y_k} (g_k^{(n)}  - g_k[n]) ,   u_n^{\gamma_n} (e_{\ell}) \rangle | 
\leq |\langle g_k^{(n)}  - g_k[n],  u_n^{\gamma_n - \alpha_n} (e_{\ell}) \rangle|
\\ & + |e^{i \alpha_n \theta_k^{(n)}} - e^{2 \pi i \alpha y_k}| |\langle g_k^{(n)}, u_n^{\gamma_n} (e_{\ell})
 \rangle |  + | \langle g_k^{(n)}  - g_k[n],  u_n^{\gamma_n} (e_{\ell}) \rangle|.
\end{align*}
To prove the first part or the proposition, since the sequence $(\alpha_n - \gamma_n)_{n \geq 1}$ satisfies 
exactly the same assumptions as 
$(\gamma_n)_{n \geq 1}$ (replacing $\gamma$ by $\gamma - \alpha$), it is sufficient to prove the almost 
sure estimates: 
\begin{equation}
|e^{i \alpha_n \theta_k^{(n)}} - e^{2 \pi i \alpha y_k}| |\langle g_k^{(n)}, u_n^{\gamma_n} (e_{\ell})
 \rangle | = o(n^{-\delta}) \label{estimatealphagamma}
\end{equation}
and 
\begin{equation}
| \langle g_k^{(n)}  - g_k[n],  u_n^{\gamma_n} (e_{\ell}) \rangle| = o(n^{-\delta}). \label{estimategamma}
\end{equation} 
For the uniform part of the Proposition  \ref{prop:distancecomponentseigenvectors}, it is
 sufficient to prove the same estimates, with $o(n^{-\delta})$ replaced by $O(n^{-\delta})$, 
this bound being uniform with respect to $\alpha_n \in [n(\alpha - n^{-\delta}), n (\alpha + n^{-\delta})]$, 
and $\gamma_n \in [n(\gamma - 2 n^{-\delta}), n (\gamma + 2 n^{-\delta})]$ (the factor $2$ comes
from the term where $\gamma_n$ plays the role of $\gamma_n - \alpha_n$). 
Now, under the assumption of the first part of the proposition, 
\eqref{estimatealphagamma} is a consequence of the two following estimates:
\begin{align*}
|e^{i \alpha_n \theta_{k}^{(n)} } -e^{2 \pi i \alpha y_k}  |
& \leq |\alpha_n \theta_k^{(n)} - 2 \pi \alpha y_k|
\\ & \leq  |\theta_k^{(n)}| |\alpha_n - \alpha n| + 
|\alpha|  |n \theta_k^{(n)} - 2 \pi y_k| 
\\ & = O(1/n) o(n^{1 - \delta}) + O(n^{-1/4}) = o(n^{-\delta}),
\end{align*}
 and 
\begin{align*}
|\langle g_k^{(n)}, u_n^{\gamma_n} (e_{\ell}) \rangle| 
& = |\langle u_n^{-\gamma_n} g_k^{(n)},  e_{\ell}\rangle| 
= |\langle e^{-i \gamma_n \theta_{k}^{(n)} } g_k^{(n)},  e_{\ell}\rangle| 
\\ & = |\langle g_k^{(n)},  e_{\ell}\rangle| = O(1),
\end{align*}
the last estimate coming from the almost sure convergence of $\langle g_k^{(n)},  e_{\ell}\rangle$ 
towards $g_{k, \ell}$. This computation is also available (after replacing small $o$ by big $O$) 
for the second part of the proposition, since 
all the estimates are easily checked to be uniform with respect to
 $\alpha_n \in [n(\alpha - n^{-\delta}), n (\alpha + n^{-\delta})]$
and  $\gamma_n  \in [n(\gamma - 2 n^{-\delta}), n (\gamma + 2 n^{-\delta})]$.

It remains to prove \eqref{estimategamma}. From the second item of Proposition \ref{prop:distanceeigenvectors}
and the fact that $\delta < 1/6$ we have the estimate: 
$$\|g_k^{(n)}  - g_k[n]\| = o(n^{\frac{1}{2} - \delta'})$$
for some $\delta' > \delta$. Hence, in order to complete the proof of 
 Proposition  \ref{prop:distancecomponentseigenvectors},
 is sufficient to show  the following property of delocalization: 
\begin{equation}
\underset{\gamma_n \in [n(\gamma - 2 n^{-\delta}), n (\gamma + 2 n^{-\delta})]}{\sup}
\frac{| \langle g_k^{(n)}  - g_k[n],  u_n^{\gamma_n} (e_{\ell}) \rangle|}{\|g_k^{(n)}  - g_k[n]\|} 
= O(n^{-\frac{1}{2} + \delta' - \delta}). \label{estimatedelocalization}
\end{equation}
This will be a consequence of the following lemma: 

\begin{lemma} \label{lem:delocalization}
For all $n \geq 1$, the quotient
\[
\frac{\abs{\langle g_k^{(n)} - g_k[n], u_n^{\gamma_n} e_\ell \rangle}^2}{\norm{g_k^{(n)} - g_k[n]}^2}
\]
is a $\beta$ random variable of parameters $1$ and $n-1$, i.e. it has the same law as 
$|x_1|^2$, where $(x_1, \dots, x_n)$ is a uniform vector on the complex sphere $S^n$. 
\end{lemma}

\begin{proof}
  Let $m \geq 1$, let $\sigma$ be a random matrix in $U(m)$, independent of $(u_n)_{n \geq 1}$ and 
  following the Haar measure. Let us define $(u'_n)_{n \geq 1}$ as the 
  unique virtual isometry such that for all $n \geq m$, 
  $$u'_n = \left( \begin{array}{cc}
\sigma & 0 \\
0 & I_{n-m}  \end{array} \right) \; u_n \;\left( \begin{array}{cc}
\sigma & 0 \\
0 & I_{n-m}  \end{array} \right)^{-1}.$$
The invariance by conjugation of the Haar measure on the space of virtual isometries implies that 
$(u'_n)_{n \geq 1}$ has the same law as $(u_n)_{n \geq 1}$. Let $(g^{(n)'}_k)_{n \geq 1}$ be the sequence 
of eigenvectors constructed from $(u'_n)_{n \geq 1}$ in the same way as $(g^{(n)}_k)_{n \geq 1}$ is constructed 
from $(u_n)_{n \geq 1}$. Since $(u'_n)_{n \geq 1}$ has the same law as $(u_n)_{n \geq 1}$, one 
deduces that almost surely, each coordinate of $g^{(n)'}_k$ with a given index converges to 
a limit when $n$ goes to infinity: let $g'_k$ be the corresponding limiting sequence, which is 
the analog of $g_k$ when $(u_n)_{n \geq 1}$ is replaced by $(u'_n)_{n \geq 1}$. 
One knows that for $n \geq m$, 
$$\tilde{g}_k^{(n)} := \left( \begin{array}{cc}
\sigma & 0 \\
0 & I_{n-m}  \end{array} \right) \, g_k^{(n)}$$
is an eigenvector of $(u'_n)_{n \geq 1}$, corresponding to the eigenvalue $\lambda_k^{(n)}$ (recall that
for $n \geq m$, $u_n$ and $u'_n$ have the same eigenvalues). Hence, almost surely, there exists $\kappa_n 
\in \mathbb{C}^*$ such that $g^{(n)'}_k = \kappa_n  \tilde{g}_k^{(n)}$. Moreover, 
from \eqref{recurrenceg}, and from the fact that $g^{(n)}_k$ is orthogonal to $f^{(n)}_j$ for 
$j \neq k$, one obtains that 
$$\langle g^{(n+1)}_{k,n} - g^{(n)}_k, g^{(n)}_k \rangle = \langle g^{(n+1)'}_{k,n}-g^{(n)'}_k
, g^{(n)'}_k \rangle = 0,$$
and then 
$$ \langle \kappa_{n+1}  \tilde{g}^{(n+1)}_{k,n} - \kappa_{n} \tilde{g}^{(n)}_k, \kappa_{n} \tilde{g}^{(n)}_k
\rangle = 0.$$
Applying $\diag (\sigma^{-1}, I_{n-m})$ to the vectors in the scalar product and 
dividing by $|\kappa_n|^2$ gives:
$$\left\langle \frac{\kappa_{n+1} \, g^{(n+1)}_{k,n} }{\kappa_n}  - g^{(n)}_k, g^{(n)}_k \right\rangle = 0,$$
and then
\begin{align*}
& \left( \frac{\kappa_{n+1}}{\kappa_n}  - 1 \right) \, \langle  g^{(n+1)}_{k,n}, g^{(n)}_k \rangle
\\ & = \left\langle \frac{\kappa_{n+1} \, g^{(n+1)}_{k,n} }{\kappa_n}  - g^{(n)}_k, g^{(n)}_k \right\rangle - 
\langle   g^{(n+1)}_{k,n} - g^{(n)}_k, g^{(n)}_k \rangle = 0,
\end{align*}
which implies $\kappa_{n+1} = \kappa_n$, since $$\langle  g^{(n+1)}_{k,n}, g^{(n)}_k \rangle
= \langle  g^{(n+1)}_{k,n} - g^{(n)}_k , g^{(n)}_k \rangle + \| g^{(n)}_k\|^2  = \| g^{(n)}_k\|^2 > 0.$$
Hence, one has $g^{(n)'}_k = \kappa_m  \tilde{g}_k^{(n)}$ for all $n \geq m$: by taking the 
limit for $n \rightarrow \infty$, one deduces that $g'_k = \kappa_m \tilde{g}_k$, 
where $\tilde{g}_k$ is the infinite sequence obtained from $g_k$ by 
applying $\sigma$ to the vector formed by the $m$ first coordinates, and by letting 
the coordinates fixed for the indices strictly larger than $m$. 
One deduces the following (with obvious notation):
\begin{align*} \langle g_k^{(m)'} - g'_{k}[m], (u'_m)^{\gamma_m} (e_{\ell_0}) \rangle 
& = \kappa_m \, \langle \tilde{g}_k^{(m)} - \tilde{g}_{k}[m], (u'_m)^{\gamma_m} (e_{\ell_0}) \rangle
\\ & = \kappa_m \, \langle \sigma (g_k^{(m)}) - \sigma( g_{k}[m]), \sigma (u_m)^{\gamma_m}
\sigma^{-1} (e_{\ell_0}) \rangle
\\ & =  \kappa_m \, \langle g_k^{(m)} - g_{k}[m], (u_m)^{\gamma_m}
\sigma^{-1} (e_{\ell_0}) \rangle
\\ & =  \kappa_m \, \langle (u_m)^{-\gamma_m} (g_k^{(m)} - g_{k}[m]), \sigma^{-1} (e_{\ell_0}) \rangle.
\end{align*}
Similarly, 
\begin{align*}
\|g_k^{(m)'} - g'_{k}[m]\|^2 & = |\kappa_m|^2  \, \|\tilde{g}_k^{(m)} - \tilde{g}_{k}[m]\|^2
= |\kappa_m|^2 \,  \|\sigma (g_k^{(m)})  - \sigma (g_{k}[m])\|^2
\\ & =  |\kappa_m|^2 \,  \|g_k^{(m)}  - g_{k}[m]\|^2
=  |\kappa_m|^2 \,  \| (u_m)^{-\gamma_m} (g_k^{(m)}  - g_{k}[m])\|^2.
\end{align*}
Hence, 
\begin{equation}
\frac{ |\langle g_k^{(m)'} - g_{k}'[m], (u'_m)^{\gamma_m} (e_{\ell_0}) \rangle  |^2}{
\|g_k^{(m)'} - g'_{k}[m]\|^2} = \frac{ | \langle x, y \rangle |^2}{\|x\|^2}, \label{beta}
\end{equation}
where 
$ x = (u_m)^{-\gamma_m} (g_k^{(m)}  - g_{k}[m])$ and $y = \sigma^{-1} (e_{\ell_0})$ is 
independent of $x$ (since $\sigma$ is independent of $(u_n)_{n \geq 1}$) and uniform on the unit 
sphere of $\mathbb{C}^m$ (since $\sigma$ is uniform on $U(m)$). One deduces that the 
left-hand side of \eqref{beta} is a beta random variable 
with parameters $1$ and $m-1$. Since $(u'_n)_{n \geq 1}$ and $(u_n)_{n \geq 1}$ have the same distribution, 
one can remove the primes from this left-hand side, which gives the announced result. 
\end{proof}
Now, applying Lemma \ref{lem:delocalization}, we get:
\begin{multline*}
\mathbb{P} 
\left( \underset{\gamma_n \in [n(\gamma - 2 n^{-\delta}), n (\gamma + 2 n^{-\delta})]}{\sup}
\frac{| \langle g_k^{(n)}  - g_k[n],  u_n^{\gamma_n} (e_{\ell}) \rangle|}{\|g_k^{(n)}  - g_k[n]\|} 
 \geq n^{-\frac{1}{2} + \delta' - \delta} \right) \\
\leq (1+ 4 n^{-\delta}) \mathbb{P} [\beta(1,n-1) \geq n^{-1 + 2(\delta' - \delta)} ]
= O\left( (1+ 4 n^{-\delta}) \, e^{- n^{2(\delta' - \delta)}/6} \right),
\end{multline*}
which by the Borel-Cantelli lemma, gives the estimate
\eqref{estimatedelocalization}, and then completes the proof of Proposition \ref{prop:distancecomponentseigenvectors}. 
\end{proof}

\section{An intrinsic inner product on the domain of the operator}
%

We return now to the original space $\mathcal{E}$ consisting of the finite linear combinations of the eigenvectors of the flow.
Most of the infinite-dimensional operators 
which are considered in the literature are defined on Hilbert spaces. 
Here, the space $\mathcal{E}$ has not a Hilbert structure, since infinite linear combinations of the basis vectors are not permitted. However, it is possible to construct an inner product on $\mathcal{E}$. 
Since eigenspaces of Hermitian and unitary operators are pairwise orthogonal, it is natural to define our scalar product in such a way that the sequences $(t_{k, \ell})_{\ell \geq 1}$, $k \in \Z$ are orthogonal. If we also suppose that these sequences have norm $1$, we then define a scalar product 
on $\mathcal{E}$ as follows: 
$$\langle w, w' \rangle = \sum_{k \in \Z}
\lambda_k \overline{\lambda'_k }$$
for $$w_{\ell} = \sum_{k \in Z} \lambda_k
 t_{k, \ell}, w'_{\ell} = \sum_{k \in Z}\lambda'_k
 t_{k, \ell},$$
 where $(\lambda_k)_{k \in \Z}$ and 
 $(\lambda'_k)_{k \in \Z}$ are sequences containing finitely many non-zeros terms. 
This definition does not depend on the phases of the vectors
$(t_{k, \ell})_{\ell \geq 1}$, $k \in \Z$
 which are chosen: indeed, if 
$t_{k, \ell}$ is multiplied by $z_k \in \mathbb{U}$, for all $k \in \Z$, $\ell \geq 1$, then $\lambda_k$ and $\lambda'_k$ are both multiplied by $z_k^{-1}$, and 
$\lambda_k \overline{\lambda'_k }$ is not changed. 
Hence, one can choose the phases in such a way that the variables
$(t_{k, \ell})_{\ell \geq 1, k \in \mathbb{Z}}$ are 
iid, complex gaussian. In particular, the 
vectors 
$(t_{k, \ell})_{\ell \geq 1}$, $k \in \Z$ are 
linearly independent, and then the sequences 
$(\lambda_k)_{k \in \Z}$ and $(\lambda'_k)_{k \in \Z}$  are uniquely determined by 
$w, w' \in \mathcal{E}$. 

In fact, the scalar product $\langle w, w' \rangle$ we have defined 
 can almost surely be written as a function of the coordinates of $w$ and $w'$, without  referring to the sequences
$(t_{k, \ell})_{\ell \geq 1}$, $k \in \Z$: 
\begin{proposition} \label{LLNE}
Let $(w_{\ell})_{\ell \geq 1}$ and $(w'_{\ell})_{\ell \geq 1}$ be two vectors in $\mathcal{E}$. Then 
$$\langle w, w' \rangle = 
\underset{n \rightarrow \infty}{\lim} 
\frac{1}{n} \sum_{\ell=1}^n w_{\ell} \overline{w'_{\ell}}
= \underset{s \rightarrow 1, s < 1}{\lim}
(1-s) \sum_{\ell = 1}^{\infty} s^{\ell - 1} w_{\ell} \overline{w'_{\ell}}.$$
\end{proposition}
\begin{proof}
By linearity, it is sufficient to show the convergence of the two limits and the equality almost surely for $w = (t_{k, \ell})_{\ell \geq 1}$ and $w' = (t_{k', \ell})_{\ell \geq 1}$ for every $k, k' \in \mathbb{Z}$. The first equality can then be written as follows: 
$$\frac{1}{n} \sum_{\ell = 1}^n  X_{k, k', \ell} \underset{n \rightarrow \infty}{\longrightarrow} 0,$$
where $X_{k, k', \ell} = t_{k, \ell} \overline{t_{k', \ell}}
- 1_{k = k'}$. This is now a consequence of 
the law of large numbers, since the variables 
$(X_{k, k', \ell})_{\ell \geq 1}$ are iid, integrable and centered. It remains to prove that 
$$(1-s) \sum_{\ell = 1}^{\infty} s^{\ell - 1} X_{k, k', \ell} \underset{s  \rightarrow 1, s < 1}{\longrightarrow} 0.$$
The sum written here is bounded in $L^1$, and then a.s. finite for all $s \in (0,1)$: moreover, it is equal to 
$$(1-s) \sum_{n = 1}^{\infty} 
(s^{n - 1} - s^{n}) \left( \sum_{\ell = 1}^n X_{k,k', \ell}
\right).$$
For any $\epsilon > 0$, there exists $n_0 \geq 1$ such that 
for any $n \geq  n_0$, 
$$\left| \sum_{\ell = 1}^n X_{k,k' \ell}  \right| 
\leq \epsilon n.$$
Hence, 
$$\left|(1-s) \, \sum_{\ell = 1}^{\infty} s^{\ell - 1} X_{k, k', \ell}\right| \leq  (1-s) \left(\sum_{n = 1}^{n_0} 
(s^{n - 1} - s^{n}) \left| \sum_{\ell = 1}^n X_{k,k', \ell}
\right| + \epsilon \sum_{n = n_0 + 1}^{\infty} n 
(s^{n - 1} - s^{n}) \right).$$
The right-hand side of this inequality tends to $\epsilon$, hence $$\underset{s \rightarrow 1, s <1}{\limsup} 
\left|(1-s) \, \sum_{\ell = 1}^{\infty} s^{\ell - 1} X_{k, k', \ell}\right| \leq \epsilon,$$
and we are done by sending $\eps \rightarrow 0$.
\end{proof}
The space $\mathcal{E}$ that we have constructed is not equipped with a complete metric topology. One could attempt to construct a completion $\overline{\mathcal{E}}$ (along with an implicit topology) as the family of formal series
$$\sum_{k \in \mathbb{Z}} \lambda_k \, (t_{k, \ell})_{\ell 
\geq 1},$$
where $(\lambda_k)_{k \in \mathbb{Z}}$ is in $\ell^2( \mathbb{Z})$. 
Unfortunately, it is not true that all such sequences in
$\overline{\mathcal{E}}$ would converge to 
sequences of complex numbers. For example, if 
$\lambda_k$ is chosen in such a way that 
$|\lambda_k| = 1/(1+|k|)$ and $\lambda_k t_{k, 1}$ is 
a nonnegative real number, then the first coordinate
would be
$$\sum_{k \in \mathbb{Z}} \lambda_k t_{k,1} 
= \sum_{k \in \mathbb{Z}} |\lambda_k t_{k,1}|
= \sum_{k \in \mathbb{Z}} \frac{|t_{k,1}|}{1 + |k|},$$
which is almost surely infinite, since by dominated 
convergence, 
\begin{align*}
\mathbb{E} \left[e^{-  \sum_{k \in \mathbb{Z}} \frac{|t_{k,1}|}{1 + |k|}} \right]
& = \underset{n \rightarrow \infty}{\lim}
\mathbb{E} \left[e^{- \sum_{|k| \leq n} \frac{|t_{k,1}|}{1 + |k|}} \right] \\
&= \prod_{k \in \mathbb{Z}} \mathbb{E} \left[e^{- \frac{|t_{k,1}|}{1 + |k|}} \right] 
\\ & \leq  \prod_{k \in \mathbb{Z}} \left[ 
\mathbb{P} (|t_{1,1}| \leq 1)
+ e^{-1/(1 + |k|)} \mathbb{P} (|t_{1,1}| > 1) \right] \\
&\leq \prod_{k \neq 0} [1 - c \abs{k}^{-1} + O(k^{-2})] \\
&= 0.
\end{align*}
We can avoid this problem by restricting, for $\delta > 0$, to the subspace $\mathcal{E}_\delta$ of $\overline{\mathcal{E}}$ given by combinations $(\lambda_k)$ such that
$$\sum_{k \in \mathbb{Z}} (1 + |k|^{1 + \delta})
|\lambda_k|^2 < \infty.$$
Indeed, under this assumption, for all $\ell \geq 1$, by Cauchy-Schwarz
\begin{equation}\sum_{k \in \mathbb{Z}} |\lambda_k t_{k,\ell} |
\leq \left(\sum_{k \in \mathbb{Z}} (1 + |k|^{1 + \delta})
|\lambda_k|^2 \right)^{1/2} \left(
\sum_{k \in \mathbb{Z}} \frac{|t_{k,\ell}|^2}{
1 + |k|^{1 + \delta}} \right)^{1/2}. \label{lambdatkl}
\end{equation}
The first factor is finite from the definition of 
$\mathcal{E}_{\delta}$ and the second factor is almost surely 
finite, since 
$$\mathbb{E} \left[ \sum_{k \in \mathbb{Z}} 
\frac{|t_{k,\ell}|^2}
{1 + |k|^{1 + \delta}}\right] = \sum_{k \in \mathbb{Z}}
\frac{1}{1 + |k|^{1 + \delta}} < \infty.
$$
One then has the following: 
\begin{proposition} \label{LLNEdelta}
Let $w$ and $w'$ be two sequences in $\mathcal{E}_{\delta}$, 
such that 
$$w_{\ell} = \sum_{k \in \mathbb{Z}} \lambda_k t_{k, \ell}, 
\; w'_{\ell} = \sum_{k \in \mathbb{Z}} \lambda'_k t_{k, \ell}, $$
where
\begin{equation}\sum_{k \in \mathbb{Z}}  (1+ |k|^{1+\delta})(|\lambda_k|^2
+ |\lambda'_k|^2) < \infty. \label{normEdelta}
\end{equation}
Then, for 
\begin{equation}\langle w, w' \rangle := \sum_{k \in \mathbb{Z}}
\lambda_k \overline{\lambda'_k},
\label{scalarproductEdelta}
\end{equation}
the conclusion of Proposition \ref{LLNE} is satisfied. 
\end{proposition}
\begin{remark}
For $w, w' \in \mathcal{E}_{\delta}$, it is a priori 
not obvious that the coordinates $(\lambda_k)_{k \in 
\mathbb{Z}}$ and $(\lambda'_k)_{k \in 
\mathbb{Z}}$ are uniquely determined, and then the definition given by the formula 
\eqref{scalarproductEdelta} is a priori ambiguous. However, 
the present and the previous propositions show that 
for all $w \in \mathcal{E}_{\delta}$, 
one has $$\lambda_k = \langle w, (t_{k, \ell})_{\ell \geq 1}
\rangle = \underset{n \rightarrow \infty}{\lim } \frac{1}{n} \sum_{\ell = 1}^n w_{\ell} \overline{t_{k, \ell}},$$
which implies that $(\lambda_k)_{k \in \Z}$ is uniquely determined by $w$.
Notice also that the convergence of the series $\sum_{k
\in \mathbb{Z}}
\lambda_k \overline{\lambda'_k}$ is an immediate consequence of the assumption \eqref{normEdelta}.

\end{remark}
\begin{proof}
Let us first show that almost surely, for $w_{\ell} = \sum_{k \in \mathbb{Z}} \lambda_k t_{k, \ell}$ and
$$\|\lambda\|_{\delta}^2 := \sum_{k \in \mathbb{Z}}
(1+ |k|^{1+\delta}) |\lambda_k|^2 < \infty,$$
one has
\begin{equation}
\underset{n  \rightarrow \infty}{\limsup} \, 
\frac{1}{n} \sum_{\ell = 1}^n |w_{\ell}|^2 \leq C_{\delta}
\|\lambda\|_{\delta}^2, \label{boundLLN}
\end{equation}
where $C_{\delta} > 0$ depends only on $\delta$.
Indeed, by \eqref{lambdatkl}, one has 
$$|w_{\ell}|^2 \leq \|\lambda\|_{\delta}^2 \mathcal{W}_{\ell},$$
where 
$$\mathcal{W}_{\ell} = \sum_{k \in \mathbb{Z}} \frac{|t_{k, \ell}|^2}
{1 + |k|^{1+ \delta} }.$$
Now, the variables $(\mathcal{W}_{\ell})_{\ell \geq 1}$
are iid, positive, with expectation: 
$$C_{\delta} = \sum_{k \in \mathbb{Z}} \frac{1}
{1 + |k|^{1 + \delta}}< \infty.$$
 By the law of large numbers, one deduces that almost surely, 
 \eqref{boundLLN} holds for all $(\lambda_k)_{k \in \mathbb{Z}}$ such that $\|\lambda\|_{\delta}^2 < \infty$. 
 
By Cauchy-Schwarz, for 
$w_{\ell} = \sum_{k \in \mathbb{Z}} \lambda_k t_{k, \ell}$ 
and $w'_{\ell} = \sum_{k \in \mathbb{Z}} \lambda'_k t_{k, \ell}$, one deduces
$$\underset{n \rightarrow \infty}{\limsup} \, 
\frac{1}{n} \left| \sum_{\ell = 1}^{n} 
w_{\ell} \overline{w'_{\ell}} \, \right| \leq C_{\delta}
\|\lambda\|_{\delta} \|\lambda'\|_{\delta}.$$
Now, for $K \geq 1$,  let use define:
$$w_{\ell, K} =\sum_{|k| \leq K} \lambda_k t_{k, \ell}
, \; 
w'_{\ell, K} =\sum_{|k| \leq K} \lambda'_k t_{k, \ell},$$
note that the corresponding sequences are in $\mathcal{E}$. 
One has: 
$$\frac{1}{n} \sum_{\ell= 1}^n w_{\ell} \overline{w'_{\ell}} = 
\frac{1}{n} \sum_{\ell= 1}^n w_{\ell, K} \overline{w'_{\ell, K}} + \frac{1}{n} \sum_{\ell= 1}^n (w_{\ell} - w_{\ell, K}) \overline{w'_{\ell, K}}
+  \frac{1}{n} \sum_{\ell= 1}^n w_{\ell}  (\overline{w'_{\ell}} - \overline{w'_{\ell, K} }).$$
By Proposition \ref{LLNE}, the first mean tends to
$$\langle (w_{\ell,K})_{\ell \geq 1}, (w'_{\ell,K})_{\ell \geq 1} \rangle  = \sum_{|k| \leq K} \lambda_k 
\overline{\lambda'_k}$$
when $n$ goes to infinity. 
The second  mean is bounded by 
$$C_{\delta} \left( \sum_{|k| > K} 
(1 + |k|^{1+ \delta}) |\lambda_k|^2 \right)^{1/2}
 \left( \sum_{|k| \leq K} 
(1 + |k|^{1+ \delta}) |\lambda'_k|^2 \right)^{1/2},$$
and the third mean is bounded by
$$C_{\delta} \|\lambda\|_{\delta} 
\left( \sum_{|k| > K} 
(1 + |k|^{1+ \delta}) |\lambda'_k|^2 \right)^{1/2}.$$
We deduce
\begin{align*}
\underset{n \rightarrow \infty}{\limsup}
\left|  \frac{1}{n} \sum_{\ell= 1}^n w_{\ell} \overline{w'_{\ell}} - \sum_{k \in \mathbb{Z}}
\lambda_k \overline{\lambda'_k} \right|
& \leq  \sum_{|k| > K} \lambda_k 
\overline{\lambda'_k} + 
C_{\delta} \|\lambda\|_{\delta} 
\left( \sum_{|k| > K} 
(1 + |k|^{1+ \delta}) |\lambda'_k|^2 \right)^{1/2}
\\ & + C_{\delta} \|\lambda'\|_{\delta} 
\left( \sum_{|k| > K} 
(1 + |k|^{1+ \delta}) |\lambda_k|^2 \right)^{1/2}.
\end{align*}
Letting $K \longrightarrow \infty$ gives the first equality in Proposition \ref{LLNE}. 
The second equality is proven if we show that for $s \in (0,1)$,
and $X_{\ell}  = w_{\ell} \overline{w'_{\ell} } -
\langle w, w' \rangle$,
\begin{equation}
(1-s) \sum_{\ell = 1}^{\infty} s^{\ell-1} X_{\ell} 
\label{sumXell}
\end{equation}
is convergent and tends to $0$ when $s$ goes to $1$. 
For $N \geq 1$, 
$$(1-s) \sum_{\ell = 1}^{N} s^{\ell-1} X_{\ell} 
=  (1-s) \sum_{n = 1}^{N} 
(s^{n - 1} - s^{n}) \left( \sum_{\ell = 1}^n X_{\ell}
\right)
 +  (1-s) s^{N} \sum_{\ell = 1}^N X_{\ell}. $$
 Since we know that 
 $$\sum_{\ell = 1}^n X_{\ell} = o(n),$$
 the series \eqref{sumXell} converges to the sum of the 
 series:
 $$(1-s) \sum_{n = 1}^{\infty} 
(s^{n - 1} - s^{n}) \left( \sum_{\ell = 1}^n X_{\ell}
\right),$$
which is absolutely convergent. We can then show that 
this sum tends to zero when $s$ goes to $1$, in the same way as in the proof of Proposition \ref{LLNE}. 
\end{proof}

The scalar product we have defined on $\mathcal{E}$, and then in $\mathcal{E}_{\delta}$, can be compared with the following 
situation. Let $(B^{(k)}_t)_{t \in [0,1]}$, $k \in \Z$,  be independent Brownian motions. 
If $(\alpha_k)_{k \in \Z}$ is a family of 
real numbers, such that $\alpha_k = 0$ for 
all but finitely many indices $k \in \Z$, then 
one can consider the stochastic 
process 
$$\left(B_t^{(\alpha_k)_{k \in \Z}} 
:= \sum_{k \in \Z} \alpha_k B^{(k)}_t \right)_{t \in [0,1]}.$$
For two sequences $(\alpha_k)_{k \in \Z}$ and $(\beta_k)_{k \in \Z}$ containing finitely non-zero terms, the quadratic covariation of $B^{(\alpha_k)_{k \in \Z}} $ and $B^{(\beta_k)_{k \in \Z}} $ is given by
$$ \langle B^{(\alpha_k)_{k \in \Z}},
B^{(\beta_k)_{k \in \Z}} \rangle = 
\sum_{k \in \Z} \alpha_k \beta_k.$$
Its defines a scalar product on the vector space 
of stochastic processes of the form $(B_t^{(\alpha_k)_{k \in \Z}})_{t \in [0,1]}$.

 Let us now go back to the vector space $\mathcal E_{\delta}$. 
This space contains some 
infinite sequences of complex numbers. It is natural to ask if it is possible to embed $\mathcal E_{\delta}$ into a space which has a richer structure. An example is obtained by considering the space of analytic functions on the open unit disc. Indeed, it is possible to identify the sequence $w = (w_{\ell})_{\ell \geq 1}$ with the function 
$$F(w)\; : \; z \mapsto \sum_{\ell \geq 1} w_{\ell} z^{\ell - 1}.$$ 
The series for $F$ converges absolutely on the unit disc, since we have
\[
  \abs{F(w)} = \abs{\sum_{\ell \geq 1} w_\ell z^{\ell - 1}} \leq \left( \sum_{\ell \geq 1} \abs{w_\ell}^2 \abs{z}^{\ell - 1} \right)^{1/2} \left( \sum_{\ell \geq 1} \abs{z}^{\ell - 1} \right)^{1/2} < \infty
\]
by Cauchy-Schwarz inequality and the proof of Proposition~\ref{LLNEdelta}.
It is then possible to express the scalar 
product $\langle w, w' \rangle$ on $\mathcal{E}_{\delta}$ in terms of 
integrals involving the holomorphic functions 
$F(w)$ and $F(w')$, as follows: 
\begin{proposition}
For all $w, w' \in \mathcal{E}_{\delta}$,
$$\langle w, w' \rangle = 2 \,\underset{s \rightarrow 1, s <1}{\lim} (1-s) \int_{0}^{2 \pi} F(w)(s e^{i \theta})
\overline{F(w')(s e^{i \theta})} \, \frac{d \theta}{2 \pi}.$$
\end{proposition}
\begin{proof}
We expand the integral,
$$\int_{0}^{2 \pi} F(w)(s e^{i \theta})
\overline{F(w')(s e^{i \theta})} \, \frac{d \theta}{2 \pi}
= \int_{0}^{2 \pi} \left( \sum_{\ell= 1}^{\infty}
w_{\ell} s^{\ell-1} e^{i \theta (\ell-1)} \right)
\left( \sum_{\ell'= 1}^{\infty}
\overline{w_{\ell'}} s^{\ell'-1} e^{-i \theta (\ell'-1)} \right)
\, \frac{d \theta}{2 \pi}.
$$
Since 
$$\sum_{\ell = 1}^{\infty} (|w_{\ell}| + |w'_{\ell}|)
s^{\ell-1} < \infty,$$
we can apply Fubini's theorem, which gives
$$\int_{0}^{2 \pi} F(w)(s e^{i \theta})
\overline{F(w')(s e^{i \theta})} \, \frac{d \theta}{2 \pi}
= \sum_{\ell=1}^{\infty} s^{2(\ell-1)}w_{\ell} \overline{w'_{\ell}}.$$
By Proposition \ref{LLNEdelta}, one deduces: 
$$(1-s^2) \int_{0}^{2 \pi} F(w)(s e^{i \theta})
\overline{F(w')(s e^{i \theta})} \, \frac{d \theta}{2 \pi}
\, \underset{s \rightarrow 1, s< 1}{\longrightarrow}
\langle w, w' \rangle,$$
which proves the proposition, since 
$2(1-s)/ (1-s^2)= 2/(1+s)$ tends to $1$ when $s$ goes 
to $1$. 
\end{proof}

The flow $(U^{\alpha})_{\alpha \in \mathbb{R}}$ can
be naturally extended to the space $\mathcal{E}_{\delta}$, by setting:
$$U^{\alpha} \left( \sum_{k \in \mathbb{Z}} 
\lambda_k t_{k, \ell} \right) 
= \sum_{k \in \mathbb{Z}} e^{2i \pi \alpha y_k}
\lambda_k t_{k, \ell}.$$
Note that for all $\alpha \in \mathbb{R}$, the extension of $U^{\alpha}$ preserves both the 
norm $$w \mapsto \langle w, w \rangle^{1/2}
= \left(\sum_{k \in \mathbb{Z}} |\lambda_k|^2 \right)^{1/2},$$
and the norm
$$w \mapsto \left(\sum_{k \in \mathbb{Z}} 
(1+ |k|^{1 + \delta})|\lambda_k|^2 \right)^{1/2}$$
defining the space $\mathcal{E}_{\delta}$.

\section{A more intrinsic definition of the limiting operator} \label{sec:intrinsic}

The random space $\mathcal{E}$ and the flow $(U^{\alpha})_{\alpha \in \mathbb{R}}$ of random operators on 
$\mathcal{E}$ defined previously have the disadvantage of involving explicitly the limiting eigenvectors of the virtual isometry $(u_n)_{n \geq 1}$. This makes the definition artificial. In this section, we construct a random space of sequences which contains $\mathcal{E}$ and a flow of operators which restricts to $U^\alpha$ on $\mathcal{E}$. This random space is not constructed directly from the eigenvectors, but is rather the space of sequences on which the action of the finite matrices $(u_n)_{n=1}^\infty$ converges in a suitable way.

All probabilistic statements in this section are assumed to hold almost surely. For this section, let $\overline{B}(x,r) \subset \R$ denote the closed interval in $\R$ with center $x$ and radius $r$. For any sequence $v$, we write $v[n]$ for the vector $(v_\ell)_{1 \leq \ell \leq n} \in \C^n$ for all $n \geq 1$.

We define our random space as follows.

\begin{definition}
Let $\mathcal{F}$ denote the space of sequences $(w_{\ell})_{\ell \geq 1}$ such that for all 
$\alpha \in \mathbb{R}$, there exists a sequence $V^{\alpha} w$, satisfying the following properties.
 
\begin{enumerate}
\item For all $\ell \geq 1$, $\alpha, \gamma \in \mathbb{R}$ and $0 < \delta' < \delta < 1/6$, there is a constant $C > 0$ such that
 $$ \sup_{\substack{\alpha_n \in \overline{B}(\alpha n, n^{1-\delta}) \\ \gamma_n \in \overline{B}(\gamma n, n^{1 - \delta})}} |\langle u_n^{\alpha_n} (w[n]) - (V^{\alpha} w)[n], u_n^{\gamma_n} (e_{\ell}) \rangle |
\leq C n^{-\delta'}.$$
\item For all $T > 0$ and $0 < \delta' < \delta < 1/6$, there is a constant $C > 0$ such that
$$ \sup_{\alpha \in [-T,T]} \sup_{\alpha_n \in \overline{B}(\alpha n, n^{1-\delta})}
\norm{u_n^{\alpha_n} w [n] - (V^{\alpha} w)[n]} \leq  C n^{\frac{1}{2} -\delta'}.$$
\end{enumerate}
Here the constants may depend on the choices of $\ell, \alpha, \gamma, \delta', \delta,$ and $T$.
\end{definition}

Note that the first condition implies, upon taking $\alpha_n = \lfloor \alpha n \rfloor$ and
 $\gamma_n = 0$, that for
$w \in \mathcal{F}$, $(V^{\alpha} w)_{\ell}$ is the limit of 
$ \langle u_n^{\lfloor \alpha n \rfloor} (w[n]) , e_{\ell} \rangle$ when $n$ goes to infinity. Hence
$V^{\alpha} w$ is uniquely determined.

The above definition gives the random space $\mathcal{F}$ and the family of operators $V^\alpha$ together. It turns out that this definition gives a flow of operators on a vector space which restricts to $\mathcal{E}$ in the natural way.

\begin{theorem}
The set $\mathcal{F}$ is a vector space. There is a family $(V^\alpha)_{\alpha \in \R}$ of linear maps, $V^\alpha : \mathcal{F} \to \mathcal{F}$, given by the correspondence $w \mapsto V^\alpha w$. The family satisfies the semigroup properties $V^0 = \text{id}$ and $V^\alpha V^\beta = V^{\alpha + \beta}$ for all $\alpha, \beta \in \R$.  Moreover, almost surely for all $k \in \mathbb{Z}$ and $\alpha \in \mathbb{R}$, 
one has $(g_{k, \ell})_{\ell \geq 1} \in \mathcal{F}$ and $V^{\alpha} ((g_{k, \ell})_{\ell \geq 1}) = (e^{2 \pi i \alpha y_k} g_{k, \ell})_{\ell \geq 1}$, so that $\mathcal{E}$ is a subspace of $\mathcal{F}$ and that $U^{\alpha}$ is the restriction of $V^{\alpha}$ to $\mathcal{E}$. 
\end{theorem}

\begin{proof}
We begin by showing that $\mathcal{F}$ is a vector space. Clearly it suffices, for $w_1, w_2 \in \mathcal{F}$ and $\lambda \in \mathbb{C}$, to show that $\lambda w_1 + w_2 \in \mathcal{F}$. Let $w = \lambda w_1 + w_2$ as sequences and define $w_\alpha = \lambda V^\alpha w_1 + V^\alpha w_2$. Now consider all $\ell \geq 1$, $\alpha, \gamma \in \R$ and $0 < \delta < 1/6$. For any $\alpha_n \in \overline{B}(\alpha n, n^{1-\delta})$ and $\gamma_n \in \overline{B}(\gamma n, n^{1 - \delta})$, by the triangle inequality
\begin{multline*}
\abs{\langle u_n^{\alpha_n} (w[n]) - w_{\alpha}[n], u_n^{\gamma_n} (e_{\ell}) \rangle}
 \leq \abs{\lambda} \abs{\langle u_n^{\alpha_n} (w_1[n]) - (V^{\alpha} w_1)[n], u_n^{\gamma_n} (e_{\ell}) \rangle}
\\ + \abs{\langle u_n^{\alpha_n} (w_2[n]) - (V^{\alpha} w_2)[n], u_n^{\gamma_n} (e_{\ell}) \rangle}.  
\end{multline*}
Since each $w_1, w_2 \in \mathcal{F}$, there are constants $C_1, C_2$, depending on $\ell, \alpha, \gamma, \delta'$, and $\delta$, such that
\[
  \abs{\langle u_n^{\alpha_n} (w[n]) - w_{\alpha}[n], u_n^{\gamma_n} (e_{\ell}) \rangle} \leq C_1 n^{-\delta'} + C_2 n^{-\delta'}
\]
so that the first condition is satisfied with constant $C_1 + C_2$. Similarly, we have
\begin{align*}
\norm{u_n^{\alpha_n} w [n] - w_{\alpha}[n]} &\leq |\lambda| \|u_n^{\alpha_n} w_1 [n] - (V^{\alpha} w_1)[n] \|
 + \|u_n^{\alpha_n} w_2 [n] - (V^{\alpha} w_2)[n] \| \\
 &\leq C_1' n^{\frac12 - \delta'} + C_2' n^{\frac12 - \delta'}
\end{align*}
for some constants $C_1'$ and $C_2'$, which implies the second condition with constant $C_1' + C_2'$. We deduce that $w \in \mathcal{F}$ with $V^\alpha w = w_\alpha$ and so $\mathcal{F}$ is a vector space as required.

Now we consider the correspondence $w \mapsto V^\beta w$ for $\beta \in \R$. The linearity of such a map follows from the above argument, so it suffices to show that it maps $\mathcal{F}$ to $\mathcal{F}$ and that it obeys the semigroup properties.

To see that $V^0 = \text{id}$, it suffices to note that $(V^0 w)_\ell$ is the limit of $\langle u_n^{\floor{0 n}} (w[n]), e_\ell \rangle = \langle w, e_\ell \rangle$ as $n$ goes to $\infty$.

Now, let $w \in \mathcal{F}$, $\alpha, \beta \in \mathbb{R}$, and let us show that $V^{\beta} w \in \mathcal{F}$
and $V^{\alpha}(V^{\beta} w) = V^{\alpha + \beta} w$. We thus need to show
\begin{enumerate}
 \item For all $\ell \geq 1$, $\alpha, \gamma \in \mathbb{R}$ and $0 < \delta < \delta'< 1/6$, 
 $$ \sup_{\substack{\alpha_n \in \overline{B}(\alpha n, n^{1 - \delta}) \\ \gamma_n \in \overline{B}(\gamma n, n^{1 - \delta})}} |\langle u_n^{\alpha_n} (V^{\beta} w)[n]) - (V^{\alpha+ \beta} w)[n], u_n^{\gamma_n} (e_{\ell}) \rangle |
= O(n^{-\delta'}).$$
\item For all $T > 0$ and $0 < \delta < \delta' < 1/6$, 
$$ \sup_{\alpha \in [-T,T]} \sup_{\alpha_n \in \overline{B}(\alpha n, n^{1 - \delta})}
\|u_n^{\alpha_n} (V^{\beta} w) [n] - (V^{\alpha+ \beta} w)[n] \| = O \left(n^{\frac{1}{2} -\delta'} \right).$$
\end{enumerate}
Note that the semigroup property follows from this because $V^{\alpha + \beta}w$ satisfies the conditions of the definition for $V^\alpha (V^\beta w)$.

Let us therefore fix the parameters $\ell \geq 1$, $\alpha, \gamma \in \R$, and $0 < \delta' < \delta < 1/6$. Consider any choice of $\alpha_n \in \overline{B}(\alpha n, n^{1- \delta})$ and $\gamma_n \in \overline{B}(\gamma n, n^{1 - \delta})$. Define the sequence $(\beta_n)_{n \geq 1}$, $\beta_n \in \Z$ by
\[
  \beta_n = \begin{cases} \floor{\beta n}, & \alpha_n \geq \alpha n \\ \floor{\beta n} + 1, & \alpha_n < \alpha n. \end{cases}
\]
It is easy to check that $\beta_n \in \overline{B}(\beta n, n^{1 - \delta})$ and $\alpha_n + \beta_n \in \overline{B}((\alpha + \beta) n, n^{1 - \delta})$. Now, by the triangle inequality,
\begin{multline*}
\abs{\langle u_n^{\alpha_n} (V^{\beta} w)[n] - (V^{\alpha+ \beta} w)[n], u_n^{\gamma_n} (e_{\ell}) \rangle} \leq \abs{ \langle u_n^{\alpha_n} (V^{\beta} w)[n] - u_n^{\alpha_n + \beta_n}(w[n]), u_n^{\gamma_n} (e_{\ell}) \rangle} \\ + \abs{\langle u_n^{\alpha_n + \beta_n}(w[n]) - (V^{\alpha+ \beta} w)[n], u_n^{\gamma_n} (e_{\ell}) \rangle}.
\end{multline*}
By unitary invariance, we see that
\[
\abs{ \langle u_n^{\alpha_n} (V^{\beta} w)[n] - u_n^{\alpha_n + \beta_n}(w[n]), u_n^{\gamma_n} (e_{\ell}) \rangle} = \abs{\langle (V^{\beta} w)[n] - u_n^{ \beta_n}(w[n]), u_n^{\gamma_n-\alpha_n} (e_{\ell}) \rangle}
\]
and therefore, taking suprema,
\begin{multline*}
\sup_{\substack{\alpha_n \in \overline{B}(\alpha n, n^{1-\delta}) \\ \gamma_n \in \overline{B}(\gamma n, n^{1 - \delta})}}  
 |\langle u_n^{\alpha_n} (V^{\beta} w)[n] - (V^{\alpha+ \beta} w)[n], u_n^{\gamma_n} (e_{\ell}) \rangle |
 \\ \leq \sup_{\substack{\alpha_n \in \overline{B}(\alpha n, n^{1-\delta}) \\ \gamma_n \in \overline{B}(\gamma n, n^{1 - \delta})}}
|\langle (V^{\beta} w)[n] - u_n^{ \beta_n}(w[n]), u_n^{\gamma_n-\alpha_n} (e_{\ell}) \rangle |
\\ + \sup_{\substack{\alpha_n \in \overline{B}(\alpha n, n^{1-\delta}) \\ \gamma_n \in \overline{B}(\gamma n, n^{1 - \delta})}}
 | \langle u_n^{\alpha_n + \beta_n}(w[n]) - (V^{\alpha+ \beta} w)[n], u_n^{\gamma_n} (e_{\ell}) \rangle |.
\end{multline*}
Denoting $\alpha'_n = \alpha_n + \beta_n$ and $\gamma'_n = \gamma_n - \alpha_n$, we have 
$|\beta_n - \beta n| \leq 1$, 
$|\alpha'_n - (\alpha + \beta)n| \leq n^{1 - \delta}$ and 
$ |\gamma'_n - (\gamma - \alpha)n| \leq 2 n^{1- \delta}$. Hence,
\begin{multline*}
\sup_{\substack{\alpha_n \in \overline{B}(\alpha n, n^{1 - \delta}) \\ \gamma_n \in \overline{B}(\gamma n, n^{1 - \delta})}} |\langle u_n^{\alpha_n} (V^{\beta} w)[n] - (V^{\alpha+ \beta} w)[n], u_n^{\gamma_n} (e_{\ell}) \rangle |
\\ \leq \sup_{\substack{\beta_n \in \overline{B}(\beta n, 1) \\ \gamma'_n \in \overline{B}((\gamma - \alpha)n, 2 n^{1 - \delta})}}
|\langle (V^{\beta} w)[n] - u_n^{ \beta_n}(w[n]), u_n^{\gamma'_n} (e_{\ell}) \rangle |
\\ + \sup_{\substack{\alpha'_n \in \overline{B}((\alpha + \beta) n, n^{1 - \delta}) \\ \gamma_n \in \overline{B}(\gamma n, n^{1 - \delta})}}
 | \langle u_n^{\alpha'_n}(w[n]) - (V^{\alpha+ \beta} w)[n], u_n^{\gamma_n} (e_{\ell}) \rangle | 
\end{multline*}
For $\delta'' \in (\delta', \delta)$, we deduce for $n$ large enough (so that
$2n^{-\delta} \leq n^{-\delta''})$, 
\begin{multline*}
\sup_{\substack{\alpha_n \in \overline{B}(\alpha n, n^{1 - \delta}) \\ \gamma_n \in \overline{B}(\gamma n, n^{1 - \delta})}} |\langle u_n^{\alpha_n} (V^{\beta} w)[n] - (V^{\alpha+ \beta} w)[n], u_n^{\gamma_n} (e_{\ell}) \rangle |
\\ \leq \sup_{\substack{\beta_n \in \overline{B}(\beta n, n^{1 - \delta''}) \\ \gamma'_n \in \overline{B}((\gamma - \alpha) n, n^{1 - \delta''})}} |\langle (V^{\beta} w)[n] - u_n^{ \beta_n}(w[n]), u_n^{\gamma'_n} (e_{\ell}) \rangle |
\\ + \sup_{\substack{\alpha'_n \in \overline{B}((\alpha + \beta) n, n^{1 - \delta}) \\ \gamma_n \in \overline{B}(\gamma n, n^{1 - \delta})}} |\langle u_n^{\alpha'_n}(w[n]) - (V^{\alpha+ \beta} w)[n], u_n^{\gamma_n} (e_{\ell}) \rangle | 
\end{multline*}
Now, this last expression is dominated by $n^{-\delta'}$, by 
definition of the space $\mathcal{F}$ and the maps $V^{\beta}$ and $V^{\alpha + \beta}$. This verifies the first condition.

The second condition follows from a similar computation. Indeed, we have
\begin{multline*}
\| u_n^{\alpha_n} (V^{\beta} w)[n] - (V^{\alpha+ \beta} w)[n] \| \leq \|u_n^{\alpha_n} (V^{\beta} w)[n] - u_n^{\alpha_n + \beta_n} w[n]\| 
\\ 
+ \|u_n^{\alpha_n + \beta_n} w[n] - (V^{\alpha+ \beta} w)[n] \|
\end{multline*}
which, again by unitary invariance, reduces to
\[
\| u_n^{\alpha_n} (V^{\beta} w)[n] - (V^{\alpha+ \beta} w)[n] \| \leq \|(V^{\beta} w)[n] - u_n^{\beta_n} w[n]\|
 + \|u_n^{\alpha_n + \beta_n} w[n] - (V^{\alpha+ \beta} w)[n] \|
\]
Defining $\alpha' = \alpha + \beta$, $\alpha'_n = \alpha_n + \beta_n$, $T' =  T + |\beta|$ and 
taking suprema for $\alpha \in [-T,T]$ and $\alpha_n \in \overline{B}(\alpha n, n^{1 - \delta})$, we get
\begin{multline*}
\sup_{\alpha \in [-T,T]} \sup_{\alpha_n \in \overline{B}(\alpha n, n^{1 - \delta})}
\|u_n^{\alpha_n} (V^{\beta} w) [n] - (V^{\alpha+ \beta} w)[n] \|
\\ \leq \sup_{\beta_n \in \overline{B}(\beta n, 1)} \|(V^{\beta} w)[n] - u_n^{\beta_n} w[n]\| 
\\ + \sup_{\alpha' \in [-T',T']} \sup_{\alpha'_n \in \overline{B}(\alpha' n, n^{1 - \delta})} \|u_n^{\alpha'_n} w[n] - (V^{\alpha'} w)[n] \|
\end{multline*}
Both of the quantities on the right hand side are bounded by $O(n^{\frac12 - \delta'})$. We have therefore proven the stability of $\mathcal{F}$ by the family of maps $(V^{\alpha})_{\alpha \in
\mathbb{R}}$ and the semigroup property.

It remains to show that the eigenvectors are contained in $\mathcal{F}$ and $V^\alpha$ restricts to $U^\alpha$ there, but this is an immediate consequence of 
Propositions \ref{prop:distanceeigenvectors} and \ref{prop:distancecomponentseigenvectors}. 
\end{proof}

We have now defined a random space $\mathcal{F}$ containing $\mathcal{E}$, on which the 
flow $(U^{\alpha})_{\alpha \in \mathbb{R}}$ is extended to a flow of operators
$(V^{\alpha})_{\alpha \in \mathbb{R}}$. The definition of $\mathcal{F}$ is quite complicated, but 
it has the strong advantage, compared with the case of $\mathcal{E}$, to be given intrinsically
 in terms of $(u_n)_{n \geq 1}$,
without referring explicitly to eigenvectors. A natural question which can now be asked is the following: 
by extending the space $\mathcal{E}$ to $\mathcal{F}$, do there exist eigenvectors of the flow which are not contained in $\mathcal{E}$? The answer is negative, in the following precise sense.
\begin{definition}
 For $w \in \mathcal{F}$ different from zero, we say that $w$ is an eigenvector of
 the flow $(V^{\alpha})_{\alpha \in \mathbb{R}}$, if and only if there exists $\chi \in \mathbb{R}$ such 
that $V^{\alpha} (w) = e^{2 i \pi \alpha \chi} w $ for all $\alpha \in \mathbb{R}$. 
\end{definition}
Using this definition, we see that for all $k \in \mathbb{Z}$, 
the sequences $(t_{k, \ell})_{\ell \geq 1}$ are eigenvectors of $(V^{\alpha})_{\alpha \in \mathbb{R}}$, 
corresponding to $\chi = y_k$. The following result shows that these sequences are the only eigenvectors of the flow: 
\begin{theorem} \label{thm:onlyeigenvectors}
 The only eigenvectors of $(V^{\alpha})_{\alpha \in \mathbb{R}}$ are the non-zero
 sequences which are proportional to 
$(t_{k, \ell})_{\ell \geq 1}$ (or $(g_{k, \ell})_{\ell \geq 1}$) for some $k \in \mathbb{Z}$. 
\end{theorem}
From this theorem, we deduce in particular that the set of parameters $\chi$ associated to the flow $(V^{\alpha})_{\alpha \in \mathbb{R}}$ is a determinantal sine-kernel process.

Before we begin the proof of this theorem, it is convenient to separately establish two partial results which will aid in the proof. For convenience, we will define
\[
  M_p(\lambda) := \frac{1}{p} \sum_{j = 0}^{p-1} \lambda^j = \frac{1 - \lambda^p}{p(1 - \lambda)}
\]
for $p \geq 1$ and $\lambda \in \C$, $\abs{\lambda} = 1$.

First, we show that a non-zero element of $\mathcal{F}$ cannot be too small in a precise sense. This in particular shows that $\mathcal{F} \cap \ell^2 = \{0\}$.
\begin{proposition} \label{prop:nosmallvectors} 
 Let $w$ be an element of $\mathcal{F}$. Suppose there exists a $\delta > 0$ and a strictly increasing sequence $(n_q)_{q \geq 1}$ of integers such that $\|w[n_q]\| = O(n_q^{\frac{1}{2} - \delta})$. Then, $w$ is identically equal to zero. 
\end{proposition}

\begin{proof}
In the first part of the definition of $\mathcal{F}$, let us take $(\delta/3) \wedge (1/7)$ instead 
of $\delta$, and $\alpha = \gamma_n = 0$. Since $V^0 w = w$, we deduce, for all $\ell \geq 1$, 
$$\underset{|\alpha_n| \leq n^{1- \frac{\delta}{3}}}{\sup} |\langle u_n^{\alpha_n} (w[n]) - w[n], e_{\ell}
\rangle| = O(n^{-\delta'}),$$
for (say) $\delta' = (\delta/4) \wedge (1/8) > 0$. Hence, if we decompose $w[n]$ in terms of eigenvectors 
of $u_n$: 
$$w[n] = \sum_{k = 1}^n \eta_k^{(n)} f_k^{(n)},$$
we get: 
$$\underset{|\alpha_n| \leq n^{1- \frac{\delta}{3}}}{\sup} 
\left| w_{\ell} - \sum_{k=1}^n (\lambda_k^{(n)})^{\alpha_n} \eta_k^{(n)} \langle  f_k^{(n)}, e_{\ell} 
\rangle \right| = O(n^{-\delta'}),$$
Taking $p = \lfloor n^{1-\frac{\delta}{3}} \rfloor + 1$ and averaging for 
$\alpha_n \in \{0, 1, \dots, p-1\}$ gives:
$$w_{\ell} = \underset{n \rightarrow \infty}{\lim}  \sum_{k = 1}^n  M_p(\lambda_k^{(n)})
\eta_k^{(n)} \langle  f_k^{(n)}, e_{\ell} \rangle.$$
Moreover, for all $n \geq 1$ and $k \in \{1, \dots, n\}$,
   $|\langle  f_k^{(n)}, e_{\ell} \rangle|^2$ is a beta random variable of parameters $1$ and $n-1$. 
From the Borel-Cantelli lemma, we deduce that almost surely, $$\langle  f_k^{(n)}, e_{\ell} \rangle
 = O\left(n^{-\frac{1}{2} + \frac{\delta}{3}}\right),$$
and then, using 
Cauchy-Schwarz inequality, we get: 
\begin{align}
 |w_{\ell}| & \lesssim  n^{-\frac{1}{2} + \frac{\delta}{3}} \, \left(\sum_{k = 1}^n  |M_p(\lambda_k^{(n)})|^2
\right)^{1/2} \left(\sum_{k = 1}^n  |\eta_k^{(n)}|^2 \right)^{1/2} \nonumber
\\ & =  n^{-\frac{1}{2} + \frac{\delta}{3}} \|w[n]\| \left(\sum_{k = 1}^n  |M_p(\lambda_k^{(n)})|^2
\right)^{1/2}. \label{boundwl}
\end{align}
Let us now assume the following bound 
\begin{equation}
\sum_{k = 1}^n  |M_p(\lambda_k^{(n)})|^2 = O(n^{\delta}). \label{boundMp}
\end{equation}
Then, for $n = n_q$ ($q \geq 1$), the right-hand side of \eqref{boundwl} is dominated by $n_q^{-\delta/6}$, which 
is only possible if $w_{\ell} = 0$. Hence, Proposition~\ref{prop:nosmallvectors} is proven if we are able to 
show \eqref{boundMp}. 
Now, we have
$$|M_p(\lambda)|^2 = \frac{1}{p^2} \sum_{j = -p+1}^{p-1} (p - |j|) \lambda^j,$$
which implies 
$$\sum_{k = 1}^n  |M_p(\lambda_k^{(n)})|^2 = 
\frac{1}{p^2} \sum_{j = -p+1}^{p-1} (p - |j|) \tr(u_n^j).$$
By \cite{DS94}, it is known, that for all integers $j_1, j_2, \dots, j_r$
such that $|j_1|+|j_2| + \dots + |j_r| \leq n$, one has
$$\mathbb{E}\left[ \prod_{s = 1}^r \tr(u_n^{j_s})
\right] = \mathbb{E}\left[ \prod_{s = 1}^r Z_{j_s}
\right],$$
where $Z_0 = n$, $(Z_j/\sqrt{j})_{j \geq 1}$ are iid standard complex gaussian variables
($\mathbb{E} [|Z_j|^2] = j$), and 
$Z_{-j} = \overline{Z_j}$ for $j \geq 1$. 
For a fixed integer $A > 0$, we deduce, from the fact that $p = o(n)$, that for $n$ large enough: 
$$\mathbb{E} \left[ \left(\sum_{k = 1}^n  |M_p(\lambda_k^{(n)})|^2 \right)^A \right]
 = \mathbb{E} [W^A],$$
where 
$$ W :=  \frac{1}{p^2} \sum_{j = -p+1}^{p-1} (p - |j|) Z_j.$$
The variable $W$ is a real-valued gaussian variable, with expectation 
$n/p$ (coming from the term $j=0$), and variance 
$$\frac{1}{p^4} \sum_{1 \leq |j|  \leq p-1} |j| (p - |j|)^2
\leq \frac{1}{p^4} \sum_{1 \leq |j|  \leq p-1} p^3 \leq 1.$$
Hence 
$$\|W\|_{L^A} \leq (n/p) + \|W - (n/p) \|_{L^A}
\leq (n/p) + \|N(0,1)\|_{L^A} \leq c(A) + n/p,$$
where $c(A) > 0$ depends only on $A$. We deduce (for $A$ fixed):
$$\mathbb{E} \left[ \left(\sum_{k = 1}^n  |M_p(\lambda_k^{(n)})|^2 \right)^A \right] 
\lesssim (n/p)^A  \lesssim n^{A\delta/3}.$$
Using Markov's inequality, we get: 
$$\mathbb{P} \left[ \sum_{k = 1}^n  |M_p(\lambda_k^{(n)})|^2 \geq n^{\delta}\right] 
\leq n^{-A\delta} \mathbb{E} \left[ \left(\sum_{k = 1}^n  |M_p(\lambda_k^{(n)})|^2 \right)^A \right] 
\lesssim n^{-2A \delta/3}.$$
Taking any $A$ strictly larger that $3/2\delta$ and using the Borel-Cantelli lemma gives the estimate \eqref{boundMp}. 
\end{proof}

We also require the following easy technical lemma.
\begin{lemma} \label{lem:boundMnc}
 There exists a universal constant $c> 0$, such that for all $\lambda \in \C$, $\abs{\lambda} = 1$, and
$n \geq 1$,
$$|M_n(\lambda)|^2 \leq 1 - c((n|\lambda-1|) \wedge 1)^2.$$
\end{lemma}

\begin{proof} If $|\lambda - 1| \geq 3/n$, we have  
$$|M_n(\lambda)|^2 = \frac{|1 - \lambda^n|^2}{n^2 |1 - \lambda|^2} \leq 4/9,$$
and the inequality is true for $c = 5/9$. 
If $|\lambda - 1| \leq 3/n$, we can write $\lambda = e^{i \theta}$
where $$|\lambda - 1| \leq |\theta| \leq \pi |\lambda - 1|/2 \leq 3 \pi/2n .$$ Then, 
\begin{align*}
1 - |M_n(\lambda)|^2 
& = \frac{1}{n^2} \sum_{j = -n+1}^{n-1} (n-|j|)( 1- \lambda^j ) \\ 
& = \frac{1}{n^2} \sum_{j = -n+1}^{n-1} (n-|j|) ( 1- \cos(j\theta)),
\end{align*}
where $|j \theta| \leq 3 \pi/2$, which implies that $1 - \cos(j\theta) \gtrsim j^2 \theta^2$. Hence, 
\begin{align*} 1 - |M_n(\lambda)|^2 & \gtrsim \frac{1}{n^2} \sum_{j = -n+1}^{n-1} (n-|j|)  j^2 \theta^2 
\\ & \gtrsim n^2 \theta^2 \geq |\lambda-1|^2 n^2. \qedhere
\end{align*}
\end{proof}

\begin{proof}[Proof of Theorem~\ref{thm:onlyeigenvectors}]
Let $w$ be an eigenvector of $(V^{\alpha})_{\alpha \in \mathbb{R}}$, and let $\chi$ be the corresponding 
eigenvalue. Taking, in the second part of the definition of $\mathcal{F}$, $T = 1$, $\alpha = j/n$ 
for $j \in \{0,1, \dots, n-1\}$, $\alpha_n = \alpha n = j$, we obtain, for any $\delta \in (0, 1/6)$: 
$$\underset{0 \leq j \leq n-1}{\sup} \|u_n^j w[n] - e^{2 i \pi \chi j/n} w[n]\| = 
O\left(n^{\frac{1}{2} - \delta} \right),$$
or equivalently, 
$$w[n] = u_n^j e^{-2 i \pi \chi j/n} w[n] + v_{n,j},$$
where 
$$\underset{0 \leq j \leq n-1}{\sup} \|v_{n,j}\|  = 
O\left(n^{\frac{1}{2} - \delta} \right).$$
Decomposing in the eigenvector basis of $u_n$ gives, with the notation of the proof of Proposition~\ref{prop:nosmallvectors}: 
$$w[n] =  v_{n,j} + \sum_{k=1}^n (\lambda_k^{(n)} e^{-2 i \pi \chi/n} )^j \eta_k^{(n)} f_k^{(n)}.$$
Averaging with respect to $j \in \{0, \dots, n-1\}$ gives: 
$$w[n] = v_n + \sum_{k = 1}^n M_n(\lambda_k^{(n)} e^{-2 i \pi \chi/n} ) \eta_k^{(n)} f_k^{(n)},$$
where $$\|v_n\| = O\left(n^{\frac{1}{2} - \delta}\right).$$
Hence, 
\begin{equation}
\|w[n]\| =  \left(\sum_{k = 1}^n | M_n(\lambda_k^{(n)} e^{-2 i \pi \chi/n} )|^2  |\eta_k^{(n)}|^2 
\right)^{1/2} + O\left(n^{\frac{1}{2} - \delta}\right). \label{expressionwnMn}
\end{equation}
First let us assume that $\chi$ is not one of the values of $y_k$ for $k \in \mathbb{Z}$. Then, there exists
$k \in \mathbb{Z}$, $\chi_1, \chi_2 \in \mathbb{R}$, such 
that $$y_k < \chi_1 < \chi < \chi_2 < y_{k+1}.$$
Hence, for $n$ large enough, 
$$\frac{n \theta_k^{(n)}}{2 \pi} <  \chi_1 < \chi < \chi_2 < \frac{n \theta_{k+1}^{(n)}}{2 \pi},$$
and there are no eigenangles of $u_n$ between $2 \pi \chi_1/n$ and $2 \pi \chi_2/n$. 
One deduces that for all $k \in \{1, \dots, n\}$, 
$$ | \lambda_k^{(n)} e^{-2 i \pi \chi/n} - 1| \gtrsim 1/n,$$
and by Lemma \ref{lem:boundMnc}, there exists $d < 1$ such that for all $n$ large enough, and 
all $k \in \{1, \dots, n\}$, 
$$| M_n(\lambda_k^{(n)} e^{-2 i \pi \chi/n} )| \leq d.$$
Therefore, \eqref{expressionwnMn} gives, for $n$ large enough: 
$$\|w[n]\| \leq  d  \left( \sum_{k = 1}^n  |\eta_k^{(n)}|^2 \right)^{1/2} +  O\left(n^{\frac{1}{2} - \delta}\right)
= d \|w[n]\| +  O\left(n^{\frac{1}{2} - \delta}\right),$$
and, since $1-d > 0$ is independent of $n$,  
$$\|w[n]\| \lesssim (1-d) \|w[n]\| = O\left(n^{\frac{1}{2} - \delta}\right).$$
By Proposition~\ref{prop:nosmallvectors}, $w$ is identically zero, which implies that $(V^{\alpha})_{\alpha \in 
\mathbb{R}}$ has no eigenvectors for $\chi \notin \{y_k, k \in \mathbb{Z}\}$. 

Now, let us assume that $\chi = y_k$ for $k \geq 1$ (the case $k \leq 0$ is similar) and let $w_n$ be
 the projection of $w[n]$ on the orthogonal 
of $f_k^{(n)}$. We have 
$$w[n] = w_n + \eta_k^{(n)} f_k^{(n)},$$
and then
$$u_n^j w[n] - e^{2 i \pi \chi j/n} w[n] =[ u^j (w_n) - e^{2 i \pi \chi j/n} w_n] + 
\eta_k^{(n)} (e^{2i \pi j \theta_k^{(n)}} -e^{2i \pi j \chi/n})  f_k^{(n)},$$
where the two terms in the last sum are orthogonal vectors. Hence, 
$$ \underset{0 \leq j \leq n-1}{\sup} \|u^j (w_n) - e^{2 i \pi \chi j/n} w_n\| \leq
 \underset{0 \leq j \leq n-1}{\sup} \|u_n^j w[n] - e^{2 i \pi \chi j/n} w[n]\| 
\lesssim n^{\frac{1}{2} - \delta}.$$
Now, decomposing $w_n$ in the eigenvector basis of $u_n$, and performing the same computation as for 
$w[n]$, we obtain an estimate which is similar to \eqref{expressionwnMn}:
\begin{equation}
\|w_n\| =  \left(\sum_{1 \leq k' \leq n, k' \neq k}
 | M_n(\lambda_{k'}^{(n)} e^{-2 i \pi \chi/n} )|^2  |\eta_k^{(n)}|^2 
\right)^{1/2} + O\left(n^{\frac{1}{2} - \delta}\right). \label{expressionwnMn2}
\end{equation}
Notice that the term $k' = k$ is not in the sum, since $w_n$ is obtained from $w[n]$ by removing the component 
proportional to $f_k^{(n)}$. Now, it is easy to check that for $k' \in \{1, \dots n\} \backslash \{k\}$, one
has the estimate 
$$|\lambda_k^{(n)} e^{-2 i \pi \chi/n} - 1 | \gtrsim 1/n$$
(note that this estimate is not true for $k' = k$, since $n \theta_k^{(n)}/2 \pi$ tends to $\chi = y_k$ when 
$n$ goes to infinity). Hence, for $n$ large enough, 
 $| M_n(\lambda_{k'}^{(n)} e^{-2 i \pi \chi/n} )|$ is uniformly bounded by a quantity which strictly smaller 
than $1$. As above for $w[n]$, this implies the estimate: 
$$\|w_n\| = O  \left(n^{\frac{1}{2} - \delta}\right).$$ 
If we set $\kappa_n := \eta_k^{(n)} / D_k^{(n)}$, we deduce
$$\|w[n] - \kappa_n g_k^{(n)} \| = O  \left(n^{\frac{1}{2} - \delta}\right),$$
and then 
\begin{equation}
\|w[n] -  \kappa_n g_k[n] \| = O \left((1 + |\kappa_n|)n^{\frac{1}{2} - \delta}\right), \label{equationkappan}
\end{equation}
since $$\|g_k[n] - g_k^{(n)}\| = O  \left(n^{\frac{1}{2} - \delta}\right).$$
Now, for any integer $m$ such that $n \leq m \leq 2n$, we have
$$\|w[m] -  \kappa_m g_k[m] \| \lesssim (1 + |\kappa_m|)m^{\frac{1}{2} - \delta} 
= O \left((1 + |\kappa_m|)n^{\frac{1}{2} - \delta}\right),$$
and taking the $n$ first components, we obtain:  
\begin{equation}
\|w[n] -  \kappa_m g_k[n] \| = O \left((1 + |\kappa_m|)n^{\frac{1}{2} - \delta}\right). \label{equationkappam}
\end{equation}
Comparing \eqref{equationkappan} and \eqref{equationkappam} gives the following: 
$$|\kappa_m - \kappa_n| \|g_k [n]\| = O \left((1 + |\kappa_n| + |\kappa_m|)n^{\frac{1}{2} - \delta}\right),$$
and then 
\begin{equation}
|\kappa_m - \kappa_n|  = O \left((1 + |\kappa_n| + |\kappa_m|)n^{- \delta}\right), \label{Cauchykappa}
\end{equation}
since $\|g_k[n]\|$ is equivalent to a strictly positive constant times $\sqrt{n}$. 
In particular, for $n$ large enough and $n \leq m \leq 2n$, 
$$|\kappa_m| - |\kappa_n| \leq |\kappa_m - \kappa_n| \leq  \frac{1}{2} (1 + |\kappa_n| + |\kappa_m|),$$
which implies 
$$ |\kappa_m| \leq 1 + 3 |\kappa_n| = O(1 + |\kappa_n|),$$
and \eqref{Cauchykappa} can be replaced by
\begin{equation}
|\kappa_m - \kappa_n|  = O \left((1 + |\kappa_n|)n^{- \delta}\right). \label{Cauchykappa2}
\end{equation}
Hence, 
$$|\kappa_m| = |\kappa_n| +  O \left((1 + |\kappa_n|)n^{- \delta}\right),$$
$$\sup_{n \leq m \leq 2n} (1 + |\kappa_m|) =  (1+ |\kappa_n|) (1 + O(n^{-\delta})),$$
and 
$$S_{q+1} \leq S_q (1 + O(2^{-\delta q})),$$
where $S_q$ denotes the supremum of $1 + |\kappa_m|$ for $2^q \leq m \leq 2^{q+1}$. 
We deduce that the sequence $(S_q)_{q \geq 0}$, and then the sequence $(\kappa_n)_{n \geq 1}$, are
bounded. The estimate  \eqref{Cauchykappa2} becomes: 
$$|\kappa_m - \kappa_n| = O \left(n^{- \delta}\right),$$
for $n \leq m \leq 2n$. If for $q \geq 1$, $2^q n \leq m \leq 2^{q+1} n$, we also get:
$$|\kappa_m - \kappa_n|  \leq |\kappa_m - \kappa_{2^{q}n}| 
+ \sum_{r = 0}^{q-1} |\kappa_{2^{r+1}n} - \kappa_{2^r n}| 
\lesssim \sum_{r = 0}^q (2^r n)^{-\delta} = O(n^{-\delta}).$$
Hence $(\kappa_n)_{n \geq 1}$ is a Cauchy sequence, and if $\kappa$ denotes its limit, one has 
$$|\kappa - \kappa_n| = O(n^{-\delta}).$$
Using \eqref{equationkappan}, we deduce:
\begin{align*}
 \| (w - \kappa g_k) [n] \| & \leq \| w[n] - \kappa_n g_k[n] \| + |\kappa - \kappa_n| \|g_k[n]\|
\\ & \lesssim (1 + |\kappa_n|) n^{\frac{1}{2} - \delta} + n^{-\delta} \sqrt{n} 
= O\left(n^{\frac{1}{2}- \delta}\right).
\end{align*}
Now, $w - \kappa g_k$ is a sequence in $\mathcal{F}$: by Proposition~\ref{prop:nosmallvectors}, $w - \kappa g_k = 0$, which implies that $w$ is proportional to $g_k$.
\end{proof}

\section{Further questions}

A natural question which can be asked is whether one 
could construct a similar flow of operators for 
other ensembles of random matrices.
If after a suitable scaling, the joint law 
of the eigenvalues and the eigenvectors converges 
to the law of a determinantal sine-kernel process and 
an independent family of iid complex gaussian variables, then one can make a coupling for which the convergence is almost sure, by Skorokhod's theorem.
Then, one can deduce the construction of a flow of 
operators similar to $(V^{\alpha})_{\alpha \in \mathbb{R}}$. 
This situation applies for a number of unitary invariant ensembles of hermitian and unitary matrices, 
as the GUE. However, the coupling given by Skorokhod's theorem is artificial. We are therefore led to the following.
\begin{question}
    Is there a natural coupling between all random matrix ensembles with determinental eigenvalue statistics so that the eigenvalues converge almost surely and the eigenvectors converge to independent gaussians?
\end{question}
Another problem consists to see if there exists a natural version of the flow $(V^{\alpha})_{\alpha \in \mathbb{R}}$, which is defined on a random space with a Hilbert structure. In this case, $V^{\alpha}$ would be 
a "true" unitary operator on this space. 
Moreover, it would be possible to define 
a self-adjoint operator $H$, whose spectrum is 
 the determinantal sine-kernel process $(y_m)_{m \in \Z}$, and which is equal to $1/2i\pi$ times the infinitesimal generator of the flow $(V^{\alpha})_{\alpha \in \mathbb{R}}$. 
 Note that $H$ would be an unbounded operator, and then its domain would no be the whole Hilbert space where $(V^{\alpha})_{\alpha \in 
 \mathbb{R}}$ is defined. However, the operator $H^{-1}$ would be a bounded, and even compact operator.

\appendix 
\section{A priori estimates for unitary matrices}\label{sec:apriori}
 Let us fix $\eps > 0$. The goal of this section is the proof that the event 
 $E := E_0 \, \cap \, E_1 \,  \cap E_2 \,  \cap \,  E_3$ holds almost surely under the Haar measure on the space of virtual isometries, for 
\begin{align*}
    E_0 &= \{\theta_0^{(1)} \neq 0 \} \cap \{ \forall n \geq 1, \nu_n \neq 0 \} \cap \{\forall n \geq 1, 1 \leq k \leq n, \mu_k^{(n)} \neq 0\} \\
    E_1 &= \{\exists n_0 \geq 1, \forall n \geq n_0, \abs{\nu_n} \leq n^{-\frac12+\eps}\} \\
    E_2 &= \{\exists n_0 \geq 1, \forall n \geq n_0, 1 \leq k \leq n, \abs{\mu_k^{(n)}} \leq n^{-\frac12+\eps}\} \\
    E_3 &= \{\exists n_0 \geq 1, \forall n \geq n_0, k \geq 1, n^{-\frac53 - \eps} \leq \theta_{k+1}^{(n)} - \theta_k^{(n)} \leq n^{-1 + \eps}\}.
\end{align*}

\begin{remark}
    In \cite{BNN12} an analogous event was defined to avoid small values of the dimension $n$ where the behavior of the eigenvalues can be more erratic. This event will serve the same purpose, although we have chosen the exponents more carefully to sharpen our results.
\end{remark}

We begin by showing that for any fixed basis of $\C^n$, the coefficients of a uniform random vector on the unit sphere are almost surely $O(n^{-\frac12+\eps})$ for any $\eps > 0$.
\begin{lemma} \label{lem:nobigmu}
    Suppose $v_1, ..., v_n \in \C^n$ is an orthonormal basis and $x \in \C^n$, $\norm{x} = 1$ is chosen uniformly from the unit sphere. Then
 if we write $x = x_1 v_1 + \cdots + x_n v_n$, we have  bound
    \[
      \mathbb{P}(\abs{x_j}^2 > \delta) = O(\exp(-\delta n/6)),
    \]
for all $\delta > 0$ and $j = 1, ..., n$.
\end{lemma}
\begin{remark}
We will prove this statement for deterministic vectors $v_1, \dots, v_n \in \mathbb{C}^n$. 
However, by conditioning, one deduces that the result remains true if the vectors $v_1, \dots, v_n$ are random, as soon as they are independent of $x$. 
\end{remark}
\begin{proof}
Let us assume $\delta < 1$, since the probability is zero otherwise. 
Let $\phi_1, ..., \phi_n$ be iid uniform random phases in $S^1$ and let 
$e_1, ..., e_n$ be independent standard exponential random variable. Then, it is
 well-known that $x$ has the same distribution as the random vector $y = y_1 v_1 + \cdots + y_n v_n$ where
\[
  y_j = \phi_j \sqrt{\frac{e_j}{e_1 + \cdots + e_n}}.
\]

Now, 
\[
\mathbb{P}(e_1 + \cdots + e_n < \frac{n}{2}) \leq e^{n/2} \mathbb{E} [\exp(-e_1 - \dots - e_n)] \leq e^{n/2} 2^{-n} \leq e^{-n/6},
\]
 so
\begin{align*}
  \mathbb{P}(\abs{x_j}^2 > \delta) &\leq \mathbb{P}(e_j > \frac{\delta n}{2}) + O(\exp(-n/6)) \\ &= O(\exp(-\delta n/6)),
\end{align*}
since $\delta \in (0,1)$. 
\end{proof}
From this estimate, we deduce the following bound on the 
coordinates of the eigenvectors $f_k^{(n)}$. 
\begin{lemma} \label{lem:uniformbeta}
Let $\eps > 0$. Then, almost surely, we have 
\[
\sup_{1 \leq j, \ell \leq n}  \abs{\langle f_j^{(n)}, e_\ell \rangle}^2 = O( n^{-1+\eps})
\]
and 
\[
\sup_{1 \leq j, \ell \leq n}  \mathbb{E} \left[
\abs{\langle f_j^{(n)}, e_\ell \rangle}^2 | \mathcal{A} \right] = O( n^{-1+\eps}),
\]
where the implied constant may depend on $\epsilon$ and $(u_m)_{m \geq 1}$. 
\end{lemma}
\begin{proof}
Consider the vector $f_j^{(n)}$ for each fixed $j$ and $n$. By the invariance by conjugation of the Haar measure on $U(n)$, this eigenvector is, up to multiplication by a complex of modulus $1$, a uniform vector on the unit sphere of $\mathbb{C}^n$. More precisely, if 
$\xi \in \mathbb{C}$ is uniform on the unit circle, and independent of $f_j^{(n)}$, 
then $\xi f_j^{(n)}$ is uniform on the unit sphere. One deduces that for all $n, j, \ell$, 
\[
  \mathbb{P}(\abs{\langle f_j^{(n)}, e_\ell \rangle}^2 > n^{-1+\eps}) = O(\exp(- n^\eps/6)).
\]
Using the Borel-Cantelli lemma gives the first result. Moreover, 
\begin{align*}
 \mathbb{E} [ | \langle f_j^{(n)}, e_{\ell} \rangle|^{8 / \eps}]
 & = \int_{0}^{\infty} \mathbb{P} [| \langle f_j^{(n)}, e_{\ell} \rangle|^{2} \geq
\delta^{\eps/4} ] d \delta
 \\ & \lesssim \int_0^{\infty} e^{- n \delta^{\eps/4}/6} d \delta
\\ & = \int_0^{\infty} e^{-z^{\eps/4}/6} d (z/n^{4/\eps}) = 
O(n^{-4/\eps}).  
\end{align*}
We deduce:
\begin{align*}
\mathbb{P}\left( \mathbb{E} [ | \langle f_j^{(n)}, e_{\ell} \rangle|^{8 / \eps} | \mathcal{A} ] \geq n^{4 - \frac{4}{\eps}}
\right) & \leq  n^{\frac{4}{\eps} -4} 
\mathbb{E}\left[ \mathbb{E} [ | \langle f_j^{(n)}, e_{\ell} \rangle|^{8 / \eps} | \mathcal{A} ]
\right] \\ & = n^{\frac{4}{\eps} - 4} \mathbb{E} [ | \langle f_j^{(n)}, e_{\ell} \rangle|^{8 / \eps} ] 
= O(n^{-4}). 
\end{align*}
By the Borel-Cantelli lemma, for all but finitely many $n \geq 1$, $1 \leq j, \ell \leq n$, 
\[
  \mathbb{E} [ \abs{\langle f_j^{(n)}, e_{\ell} \rangle}^{8 / \eps} | \mathcal{A} ] \leq n^{4 - \frac{4}{\eps}}.
\]
By the H\"older inequality applied to the conditional expectation, for $\eps$ sufficiently small,  
\[
\mathbb{E} [ | \langle f_j^{(n)}, e_{\ell} \rangle|^{2} | \mathcal{A} ] \leq \left( \mathbb{E} [ | \langle f_j^{(n)}, e_{\ell} \rangle|^{8 / \eps} | \mathcal{A} ] \right)^{\eps/4} \leq n^{- 1 + \eps}. \qedhere
\]
\end{proof}
Another consequence of Lemma~\ref{lem:nobigmu} is the following: 
\begin{proposition}
The events $E_0$, $E_1$, $E_2$ all hold almost surely.
\end{proposition}

\begin{proof}
We apply Lemma~\ref{lem:nobigmu} to the decomposition
\[
  x_{n+1} = \sum_{j=1}^n \mu_j^{(n)} f_j^{(n)} + \nu_n e_{n+1}, 
\]
which gives
\[
  \mathbb{P}(\abs{\mu_k^{(n)}}^2 > n^{-1+\eps}) = O(\exp(-n^{\eps}/6))
\]
so, in particular,
\[
  \sum_{n \geq 1} \sum_{1 \leq k \leq n} \mathbb{P}(\abs{\mu_k^{(n)}}^2 > n^{-1+\eps}) = O(1).
\]
Therefore, by the Borel-Cantelli lemma, almost surely only a finite number of the events
 $\{\abs{\mu_k^{(n)}}^2 > n^{-1+\eps}\}$ hold simultaneously. A similar argument 
controls the coefficients $\nu_n$.
\end{proof}

Before we can control $E_3$, we require some estimates on the eigenvalues of a Haar unitary random matrix. 
Recall that if $u_n$ is distributed according to the Haar measure, then one can define, for $1 \leq p \leq n$, 
the $p$-point correlation function $\rho_p^{(n)}$ of the eigenangles, as follows: for any bounded, measurable 
function $\phi$ from $\mathbb{R}^p$ to $\mathbb{R}$, 
\begin{align*}
& \mathbb{E} \left[ \sum_{1 \leq j_1 \neq \dots \neq j_p \leq n} \phi(\theta_{j_1}^{(n)}, \dots, 
\theta_{j_p}^{(n)}) \right] \\ &  = \int_{[0, 2 \pi)^p} \rho_p^{(n)} (t_1, \dots, t_p)
\phi(t_1, \dots, t_p)
 dt_1 \dots dt_p.
\end{align*}
Moreover, if the kernel $K$ is defined by 
\[
  K(t) := \frac{\sin(n t/2)}{2 \pi \sin(t/2)}
\]
then the $p$-point correlation function can be given by 
\[
  \rho_p^{(n)}(t_1, ..., t_n) = \det \begin{pmatrix} K(t_j - t_k) \end{pmatrix}_{j,k=1}^p.
\]

Let us first show that the gaps between eigenvalues cannot be asymptotically much larger than average.
\begin{lemma}\label{lem:nobiggaps}
    Let $I \subseteq [0,2\pi)$ be Lebesgue measurable. Then
    \[
      \mathbb{P}(\text{all of the eigenvalues of } u_n \text{ are in } I) \leq \exp(- \frac{\abs{I^c}}{2 \pi} n).
    \]
\end{lemma}

\begin{proof} 
We recall the Andreiev-Heine identity \cite{Sos00}, which says that
\[
  \mathbb{P}(\text{all of the eigenvalues of } u_n \text{ are in } I) = \det M^I
\]
where $M^I$ is an $n \times n$ matrix with entries
\[
M_{j,k}^I = \int_I \exp(i(j-k) t) \frac{dt}{2 \pi}
\]
for $j,k$ between $1$ and $n$. Note that the matrix $\begin{pmatrix} \exp(i(j-k) t) \end{pmatrix}_{j,k=1}^n$ is hermitian and positive, $M^I$ is also; likewise $M^{I^c}$. Moreover, by computing the entries of 
 $M^I + M^{I^c}$, one checks that this sum is the identity matrix: hence, $M^I$, $M^{I^c}$ have the same eigenvectors and, if we denote by $(\tau_j)_{1 \leq j \leq n}$  the eigenvalues of $M^{I^c}$, then $(1 - \tau_j)_{1 \leq j \leq n}$ are the eigenvalues of $M^I$. The eigenvalues of each matrix must lie in the interval $[0,1]$, as otherwise one of the eigenvalues of the other matrix would be negative. Now,
\[
  \det M^I = \prod_{j=1}^n (1 - \tau_j) \leq \exp(- \sum_{j=1}^n \tau_j) = \exp(- \tr M^{I^c}) = \exp(- \frac{I^c}{2 \pi} n)
\]
as was to be shown.
\end{proof}

Note that the previous lemma applies to all measurable subsets, although we will only need to apply it to intervals.

Next we control the gaps between eigenvalues from below.

\begin{lemma} \label{lem:nosmallgaps}
 
Suppose $t_1, ..., t_p \in I$ lie in an interval of length $\abs{I} = \delta \leq 1/n$. Then we
 have the estimate
\[
  \rho_p^{(n)}(t_1, ..., t_p) = O_p(\delta^{2p-2} n^{3p-2}).
\]
\end{lemma}

\begin{proof}
We have
\[
  \rho_p^{(n)}(t_1, \dots, t_p)= \det \begin{pmatrix} K(t_i - t_j) \end{pmatrix}_{i,j = 1}^p.
\]
The Taylor series for the sine function shows that for $|t| \leq 1/n$, 
\[
  K(t) = \frac{n}{2 \pi} \left(1 - \frac{1}{24} (n^2 - 1) t^2 + O(n^4t^4) \right).
\]
Thus, we have:
\[
  \rho_p^{(n)}(t_1, \dots, t_p) = \frac{n^p}{(2 \pi)^p} \det \begin{pmatrix} 1 - \frac{1}{24} (n^2 - 1) (t_i - t_j)^2 + O(n^4(t_i - 
t_j)^4) \end{pmatrix}_{i,j=1}^p
\]
Let $A$ denote the $p \times p$ matrix in the last display, let $1$ denote the column vector of all ones and let $w_j$ denote the column vector whose $i$th entry is
\[
  (w_j)_i = 1 - A_{ij} = \frac{1}{24} (n^2 - 1) (t_i - t_j)^2 + O(n^4
(t_i - t_j)^4).
\]
Then by multilinearity and the inclusion-exclusion principle,
\[
  \det A = \sum_{\sigma \subset [p]} (-1)^{\abs{\sigma}} \det \begin{pmatrix} v_1 & \cdots & v_p \end{pmatrix}, \qquad \text{where } v_j = \begin{cases} w_j, & j \in \sigma \\ 1, & \text{otherwise} \end{cases}
\]
Clearly each term is zero if more than one of the columns is equal to $1$, so we get
\[
  \det A = (-1)^{p-1} \sum_{j=1}^p \det M_j +  (-1)^p \det M
\]
where $M$ is the matrix with columns $w_1, ..., w_p$ and $M_j$ is $M$ with the $j$th column replaced with $1$. Then in the expansion of each determinant we can bound each term by $O_p( (n^2 \delta^2)^{p-1})$, and the conclusion follows.
\end{proof}

\begin{proposition}
    The event $E_3$ holds almost surely.
\end{proposition}

\begin{proof}
Fix $n \geq 1$. The probability that two adjacent eigenvalues of $u_n$ differ by at least $2 \delta$ is bounded above by the probability that one of the parts of the partition
\[
  (0, \delta) \cup (\delta, 2 \delta) \cup \dotsb \cup (\lfloor 2 \pi \delta^{-1} \rfloor \delta,
\lfloor 2 \pi \delta^{-1} + 1 \rfloor \delta)
\]
contains no eigenvalue. This, by Lemma~\ref{lem:nobiggaps}, is bounded by
\[
  \lfloor 2 \pi \delta^{-1} + 1 \rfloor \exp(- \delta n/ 2 \pi).
\]
Now we let $\delta = n^{-1+\eps}$ and apply the Borel-Cantelli lemma to show that at most a finite number of the $u_n$ have gaps larger than $n^{-1 + \eps}$.

Next, we see by Lemma~\ref{lem:nosmallgaps} that the probability that two adjacent eigenvalues of $u_n$ differ by at most $\delta \leq 1/n$ is bounded by
\[
  \iint_{\abs{t_1 - t_2} < \delta} \rho_2^{(n)}(t_1, t_2) \, dt_1 \, dt_2 = O(n^4 \delta^3)
\]
which, when we specialize $\delta = n^{-\frac53-\eps}$, is $O(n^{-1-\eps})$; summing over $n$ and applying the Borel-Cantelli lemma shows that these events occur at must a finite number of times as well.
\end{proof}

\section{Proof of Proposition \ref{Laplacefunctionals}}
In the course of the proof we shall use the standard facts that for $1 \leq r \leq n$, and for all $y_1, \dots, y_r \in (-n/2, n/2]$, we have
$$0 \leq \rho_r^{(n)}(y_1, \dots, y_r) \leq 1,$$ and that the correlation functions of a  Poisson point process with intensity $1$ are all equal to $1$ (and hence are the correlation functions $\rho_r^{(n)}$ are dominated by those of a Poisson point process with intensity $1$).

We first note the following identity: for any integer $p \geq 0$,
$$ \left(\sum_{x \in E_n} f(x)  \right)^p 
= \sum_{m=1}^{u_p}
\sum_{x_1 \neq x_2 \neq \dots \neq x_{r_{p,m}} \in 
E_n}
G_{f,p,m} (x_1, \dots, x_{r_{p,m}} ),$$
where $u_p$ depends only on $p$, 
$r_{p,m}$ on $p$ and  $m \leq u_p$, 
and $G_{f,p,m}$ being a measurable, bounded  function with compact support
from $\mathbb{R}^{r_{p,m}}$ to $\mathbb{R}$, and depending only on $f, p$ and $m$. 
For instance
$$ \left(\sum_{x \in E_n} f(x)  \right)^3
= \sum_{x_1 \in E_n} (f(x_1))^3 
+ 3 \sum_{x_1 \neq x_2 \in E_n} (f(x_1))^2 f(x_2) 
+ \sum_{x_1 \neq x_2 \neq x_3 \in E_n}
f(x_1) f(x_2) f(x_3),$$
with
$$u_3 = 3, r_{3,1} = 1, r_{3,2} = 2,
r_{3,3} = 3,$$
$$G_{f,3,1}(x_1) = (f(x_1))^3,$$
$$G_{f,3,2}(x_1, x_2) = 3(f(x_1))^2 f(x_2),$$
$$G_{f,3,3}(x_1, x_2,x_3) =f(x_1) f(x_2)
f(x_3).$$
We can hence write
$$\mathbb{E} \left[\left(\sum_{x \in E_n} f(x)  \right)^p \right]
= \sum_{m=1}^{u_p}
\int_{(-n/2,n/2]^{r_{p,m}}}
G_{f,p,m} (y_1, \dots, y_{r_{p,m}} )
\rho_{r_{p,m}}^{(n)} (y_1, \dots, y_{r_{p,m}} ) \, dy_1 \dots dy_{r_{p,m}},$$
provided the above expression converges absolutely, which we now check.
Since $G_{f,p,m}$ is measurable, bounded with compact support, we can find
 $A_{f,p,m}> 0$ such that
$$|G_{f,p,m}(y_1, \dots, y_{r_{p,m}})| \leq 
A_{f,p,m} \mathds{1}_{|y_1|, \dots, 
|y_{r_{p,m}}| \leq A_{f,p,m}}$$
for $y_1, \dots, y_{r_{p,m}} \in 
\mathbb{R}$. 
Moreover from the remark above on the correlation functions, we have 
$$|\rho_{r_{p,m}}^{(n)} (y_1, \dots, y_{r_{p,m}} ) | \mathds{1}_{y_1, \dots, 
y_{r_{p,m}} \in (-n/2, n/2]}
 \leq  1.$$
Consequently the expression we are dealing with can be  bounded from above by
$$\sum_{m=1}^{u_p} 
\int_{[-A_{f,p,m}, A_{f,p,m}]^{r_{p,m}}}
A_{f,p,m} 
\leq  \sum_{m=1}^{u_p} 
(2A_{f,p,m})^{r_{p,m} + 1},$$
which is finite. 
Moreover our upper bound is independent 
of $n$. Now since the kernel 
$K^{(n)}$ converges pointwise to
$K^{(\infty)}$, we also have $$\rho_{r_{p,m}}^{(n)} (y_1, \dots, y_{r_{p,m}} ) \mathds{1}_{y_1, \dots, 
y_{r_{p,m}\in (-n/2, n/2]}}
\underset{n \rightarrow \infty}{\longrightarrow} \rho_{r_{p,m}}^{(\infty)} (y_1, \dots, y_{r_{p,m}} ),$$
and we can apply the dominated convergence theorem to obtain
$$\mathbb{E}[(X_f^{(n)})^p] 
\underset{n \rightarrow \infty}{\longrightarrow} M_{f,p}^{(\infty)}$$
where
$$X_f^{(n)} = \sum_{x \in E_n} f(x) $$ 
and 
$$ M_{f,p}^{(\infty)} = \sum_{m=1}^{u_p}
\int_{\mathbb{R}^{r_{p,m}}}
G_{f,p,m} (y_1, \dots, y_{r_{p,m}} )
\rho_{r_{p,m}}^{(\infty)} (y_1, \dots, y_{r_{p,m}} ) \, dy_1 \dots dy_{r_{p,m}}.$$
We also note that
$$\mathbb{E}[|X_f^{(n)}|^p]  \leq 
\mathbb{E}[(X_{|f|}^{(n)})^p] \leq 
\mathbb{E} \left[ \left(\sum_{x  \in N} 
|f(x)| \right)^p \right],$$
where $N$ is a Poisson point process defined on $\mathbb{R}$ with intensity $1$ (the last inequality follows from the fact that the correlation functions of  $E_n$ are smaller or equal than 1, and hence smaller or equal than the correlation functions of $N$).  

Now for every $\lambda \in \mathbb{R}$, each term of the series
$$\sum_{p \geq 0} \frac{(i\lambda)^p}{p!} \mathbb{E}[(X_f^{(n)})^p]$$
is uniformly dominated in absolute value and independently of $n$, by the corresponding term in the series
$$\sum_{p \geq 0} \frac{|\lambda|^p}{p!} \mathbb{E} \left[ \left(\sum_{x  \in N} 
|f(x)| \right)^p \right]
=\mathbb{E} \left[
 \exp \left( |\lambda| \sum_{x  \in N} |f(x)| \right) \right].$$
If we choose $A_f >  0$ in such a way that $|f|\leq A_f$ and such that the support of $f$ is contained in
$[-A_f, A_f]$, we have
$$\mathbb{E} \left[
 \exp \left( |\lambda| \sum_{x  \in N} |f(x)| \right) \right]
 \leq \mathbb{E} \left[
 \exp \left( |\lambda| A_f \operatorname{Card} (N \cap [-A_f, A_f]) \right) \right] = \mathbb{E} [e^{|\lambda|
 A_f Y_{2 A_f}}], $$
 $Y_{2 A_f}$ standing for a Poisson random variable with parameter  $2 A_f$. 
The latter is finite and we can thus apply the dominated convergence theorem to obtain $$\mathbb{E} \left[e^{i \lambda X_f^{(n)}} \right] \underset{n \rightarrow \infty}{\longrightarrow}
 \sum_{p \geq 0} \frac{(i\lambda)^p}{p!} M_{f,p}^{(\infty)},$$
 the last series in display being absolutely convergent and  bounded from above by
 \begin{align*}\left|1 - \sum_{p \geq 0} \frac{(i\lambda)^p}{p!} M_{f,p}^{(\infty)} \right|
  & \leq  \sum_{p \geq 1} \frac{|\lambda|^p}{p!} M_{|f|,p}^{(\infty)}
   \leq 
   \sum_{p \geq 1} \frac{|\lambda|^p}{p!} \sup_{n \geq 1} \mathbb{E} [|X_f^{(n)}|^p]
   \\ & \leq 
   \sum_{p \geq 1} \frac{|\lambda|^p}{p!}
   \mathbb{E} \left[
   \left( \sum_{x \in N} |f(x)| \right)^p
   \right]
    = \mathbb{E} \left[
 \exp \left( |\lambda| \sum_{x  \in N} |f(x)| \right) \right] - 1
 \\ & \leq \mathbb{E} [e^{|\lambda|
 A_f Y_{2 A_f}}] - 1 = e^{2 A_f( e^{|\lambda| A_f} - 1)} - 1.
 \end{align*}
Consider now a finite number $f_1, f_2, \dots, f_q$ of measurable and bounded functions with compact support, and let $A > 0$
be such that $|f_j| \leq A \mathds{1}_{[-A,A]}$ 
for  $j \in \{1, \dots, q\}$, and take
$\lambda, \lambda_1, \dots, \lambda_q \in \mathbb{R}$.
It follows from the definition of  $X_f^{(n)}$ that
$$\sum_{j=1}^q \lambda_j X_{f_j}^{(n)}
 = X_g^{(n)}$$
 where $$g := \sum_{j=1}^q \lambda_j f_j,$$
which implies that
 $$\mathbb{E} \left[e^{i \lambda \sum_{j=1}^q 
 \lambda_j X_{f_j}^{(n)}} \right]
  \underset{n \rightarrow \infty}{\longrightarrow}
 \sum_{p \geq 0} \frac{(i\lambda)^p}{p!} M_{g,p}^{(\infty)}.$$
 Now since $g$ is bounded by $ A \sum_{j=1}^q |\lambda_j| $ and since the support of $g$ is included in $[-A,A]$, we have
 $$\left| 1 -  \sum_{p \geq 0} \frac{(i\lambda)^p}{p!} M_{g,p}^{(\infty)} \right|
 \leq e^{2 A (1 +  \sum_{j=1}^q |\lambda_j|) \left( e^{|\lambda| A 
 (1 +  \sum_{j=1}^q |\lambda_j|) } - 1 
 \right)} - 1. $$
 If $\nu_1, \dots, \nu_q$ are real numbers not all equal to zero, we set
 $$\lambda = \sum_{j=1}^q |\nu_j|, 
\lambda_j = \nu_j/\lambda,$$
which implies $\sum_{j=1}^q |\lambda_j|
= 1$, and
$$\mathbb{E} \left[e^{i\sum_{j=1}^q 
 \nu_j X_{f_j}^{(n)}} \right]
  \underset{n \rightarrow \infty}{\longrightarrow}
  Q(f_1, \dots, f_q, \nu_1, \dots, \nu_q)$$
where
  $$ |Q(f_1, \dots, f_q, \nu_1, \dots, \nu_q) - 1| \leq 
  e^{4 A  ( e^{2|\lambda| A 
  } - 1 )} - 1.$$
 For fixed $f_1, \dots, f_q$, 
the quantity  
$ Q(f_1, \dots, f_q, \nu_1, \dots, \nu_q)$ hence tends to $1$ when 
$ (\nu_1, \dots, \nu_q)$ tends to zero. 
 It follows from L\'evy's convergence theorem that the vector
 $(X_{f_1}^{(n)}, \dots, X_{f_q}^{(n)})$
 converges in law, when $n$ goes to infinity, to a random variable 
 with values in $\mathbb{R}^q$ with characteristic function given by
 $$(\nu_1, \dots, \nu_q) 
 \mapsto Q(f_1, \dots, f_q, \nu_1, \dots, 
 \nu_q).$$
 Consequently we have shown that there exists a family of random variables  $(X_{f}^{(\infty)})_{f \in \mathcal{A}}$ indexed by the set  $\mathcal{A}$ of bounded and measurable functions with compact support 
 from $\mathbb{R}$ to $\mathbb{R}$, satisfying 
 $$(X_{f}^{(n)})_{f \in \mathcal{A}}
\underset{n \rightarrow \infty}{\longrightarrow} (X_{f}^{(\infty)})_{f \in \mathcal{A}},$$
in the sense of finite dimensional distributions. 
 Now for  $x \geq 0$, define
 $$X^{(n)} (x) = X_{\mathds{1}_{[0,x]}}^{(n)}, \; X^{(\infty)} (x) = X_{\mathds{1}_{[0,x]}}^{(\infty)},$$
and for  $x <0$, 
 $$X^{(n)} (x) = - X_{\mathds{1}_{(x,0)}}^{(n)}, \; X^{(\infty)} (x) = - X_{\mathds{1}_{(x,0)}}^{(\infty)}.$$
Note that for $y \geq x$, 
 $X^{(n)}(y) - X^{(n)}(x) $ represents the number of points of $E_n$ in the interval $(x,y]$. 
 Moreover we saw that $(X^{(n)}(x))_{x \in \mathbb{Q}}$
 converges in law (in the sense of finite dimensional distributions) 
 to $(X^{(\infty)}(x))_{x \in \mathbb{Q}}$. It follows from Skorokhod's representation theorem that there exist random variables  $(Y^{(n)}(x))_{x \in \mathbb{Q}}$
  and $(Y^{(\infty)}(x))_{x \in \mathbb{Q}}$, with respectively the same distributions as
  $(X^{(n)}(x))_{x \in \mathbb{Q}}$ and
 $(X^{(\infty)}(x))_{x \in \mathbb{Q}}$, 
 such that almost surely, 
 $Y^{(n)}(x)$ converges to $Y^{(\infty)}(x)$ for all  $x \in \mathbb{Q}$. 
By construction  $(Y^{(n)}(x))_{x \in \mathbb{Q}}$ is almost surely integer valued and increasing as a function of $x$: the same thing holds for $(Y^{(\infty)}(x))_{x \in \mathbb{Q}}$. Moreover by taking the limits from the right, we can extend $(Y^{(n)}(x))_{x \in \mathbb{Q}}$ et $(Y^{(\infty)}(x))_{x \in \mathbb{Q}}$ to c\`adl\`ag functions defined on $\mathbb{R}$. 
It is clear that  $(Y^{(n)}(x))_{x \in \mathbb{R}}$ has then the same law as 
 $(X^{(n)}(x))_{x \in \mathbb{R}}$, because
 $(X^{(n)}(x))_{x \in \mathbb{R}}$ is also c\`adl\`ag, with the same law when restricted to   $\mathbb{Q}$. We can thus conclude that like for $(X^{(n)}(x))_{x \in \mathbb{R}}$,  $(Y^{(n)}(x))_{x \in \mathbb{R}}$ is also the distribution function of some $\sigma$-finite measure
 $\mathcal{M}_n$, with the same law as the sum of the Dirac measures taken at the points of 
  $E_n$.
Almost surely, for $x \in \mathbb{Q}$, $(Y^{(n)}(x))$ converges to
 $(Y^{(\infty)}(x))$: hence this convergences also holds at all continuity points of $(Y^{(\infty)}(x))$. Consequently $\mathcal{M}_n$ converges weakly, in the sense of convergence in law on compact subsets, to a limiting random measure $\mathcal{M}_{\infty}$, with distribution function $Y^{(\infty)}$.  
 On can thus write
 $$\mathcal{M}_n = \sum_{k \in \mathbb{Z}}
 \delta_{t^{(n)}_k}, \; \mathcal{M}_{\infty} = \sum_{k \in \mathbb{Z}}
 \delta_{t^{(\infty)}_k},$$
 where $\{t^{(n)}_k, k \in \mathbb{Z}\}$ 
 is a set of points with the same distribution as  $E_n$. 
 The weak convergence of $\mathcal{M}_n$ 
 to $\mathcal{M}_{\infty}$ implies that 
 for $r \geq 0$, $F$ continuous with 
 compact support from $\mathbb{R}^r$ to 
 $\mathbb{R}$, 
 $$\sum_{k_1 \neq k_2 \neq \dots
 \neq k_r} F(t^{(n)}_{k_1}, \dots, t^{(n)}_{k_r})
 \underset{n \rightarrow \infty}{\longrightarrow} \sum_{k_1 \neq k_2 \neq \dots 
 \neq k_r} F(t^{(\infty)}_{k_1}, \dots 
 t^{(\infty)}_{k_r}).$$
 Indeed the left hand side can be written as: 
 $$\sum_{m = 1}^{u_r} 
 \int_{\mathbb{R}^{s_{r,m}}}
 H_{F,r,m}(y_1, \dots, y_{s_{r,m}})
 d\mathcal{M}_n(y_1) \dots 
 d\mathcal{M}_n(y_{s_{r,m}}),$$
 where $u_r$ depends only on $r$, $s_{r,m}$ on $r$ and on $m$ and $H_{F,r,m}$ 
 depends on $F, r,m$,  and the right hand side can be written in similar way 
 with $\mathcal{M}_n$ replaced with
 $\mathcal{M}_{\infty}$. 
For instance 
 \begin{align*}\sum_{k_1 \neq k_2 \neq k_3}
  F(t^{(n)}_{k_1}, \dots, t^{(n)}_{k_3})
&  = \int_{\mathbb{R}^3} F(y_1, y_2,y_3) 
 d\mathcal{M}_n(y_1) d\mathcal{M}_n(y_2)
 d\mathcal{M}_n(y_3)
\\ &  - \int_{\mathbb{R}^2} [F(y_1, y_2,y_2) 
 + F(y_2, y_1,y_2) + F(y_1, y_2,y_2)]
 d\mathcal{M}_n(y_1) d\mathcal{M}_n(y_2)
\\ &  + 2 \int_{\mathbb{R}} F(y_1, y_1,y_1) 
 d\mathcal{M}_n(y_1).
 \end{align*}
If we assume that  $F$ is positive, 
 it follows from Fatou's lemma that
 \begin{align*}
 \mathbb{E} \left[\sum_{k_1 \neq k_2 \neq \dots 
 \neq k_r} F(t^{(\infty)}_{k_1}, \dots 
 t^{(\infty)}_{k_r}) \right]
 & \leq \underset{n \rightarrow \infty}{\lim \inf} \, \mathbb{E} \, \left[\sum_{k_1 \neq k_2 \neq \dots 
 \neq k_r} F(t^{(n)}_{k_1}, \dots 
 t^{(n)}_{k_r}) \right]
 \\ & = \underset{n \rightarrow \infty}{\lim \inf} \int_{\mathbb{R}^r} 
 F(y_1, \dots, y_r) \rho_r^{(n)}
 (y_1, \dots, y_r)dy_1 \dots dy_r
\\ &  = \int_{\mathbb{R}^r} 
 F(y_1, \dots, y_r) \rho_r^{(\infty)}
 (y_1, \dots, y_r) dy_1 \dots dy_r.
 \end{align*}
Reproducing the same computations as in the beginning of our proof yields, 
for $f$ continuous and 
positive with compact support, and $p$ a positive integer:
$$N^{(\infty)}_{f,p} := \mathbb{E} \left[ \left(\sum_{k \in \mathbb{Z}}
f(t^{(\infty)}_k) \right)^p \right]
\leq M^{(\infty)}_{f,p}.$$
The bounds that we previously obtained for
$M^{(\infty)}_{f,p}$, and which obviously apply to
$N^{(\infty)}_{f,p}$ as well, allow us to deduce that for all $\lambda \in \mathbb{R}$,
$$\mathbb{E} \left[ \exp \left( i \lambda 
 \sum_{k \in \mathbb{Z}}
f(t^{(\infty)}_k) \right) \right]
= \sum_{p \geq 0} \frac{(i \lambda)^p}{p!}
N_{f,p}^{(\infty)}.$$
Moreover an application of the dominated convergence theorem yields
\begin{align*}\mathbb{E} \left[ \exp \left( i \lambda 
 \sum_{k \in \mathbb{Z}}
f(t^{(\infty)}_k) \right) \right]
& = \underset{n \rightarrow \infty}{\lim}
\mathbb{E} \left[ \exp \left( i \lambda 
 \sum_{k \in \mathbb{Z}}
f(t^{(n)}_k) \right) \right]
\\ & =\underset{n \rightarrow \infty}{\lim}
\mathbb{E} \left[e^{i \lambda X_f^{(n)}} 
\right] = \sum_{p \geq 0} \frac{(i \lambda)^p}{p!}
M_{f,p}^{(\infty)}. 
\end{align*}
We can hence conclude that the coefficients of both series in  $\lambda$ are equal, i.e. 
$M_{f,p}^{(\infty)} = N_{f,p}^{(\infty)}$. 
Going back to the expression of the expansion of the moment of order $p$ that we gave earlier in the proof, we see that the equality can hold only if 
$$\mathbb{E} \left[\sum_{k_1 \neq k_2 \neq \dots 
 \neq k_r} F(t^{(\infty)}_{k_1}, \dots 
 t^{(\infty)}_{k_r}) \right] 
 = \int_{\mathbb{R}^r} 
 F(y_1, \dots, y_r) \rho_r^{(\infty)}
 (y_1, \dots, y_r) dy_1 \dots dy_r,$$
for all $F, r$ such that 
 $r = r_{p,m}, F = G_{f,p,m}$, with 
 $1 \leq m \leq u_p$. Indeed the left hand side is always smaller or equal than the right hand side, 
 and if one of the inequalities were a strict inequality, we would obtain by summing up all terms that $ N_{f,p}^{(\infty)} < M_{f,p}^{(\infty)}$. The only term for which $r_{p,m} = p$
 gives
 $$\mathbb{E} \left[\sum_{k_1 \neq k_2 \neq \dots 
 \neq k_p} F(t^{(\infty)}_{k_1}, \dots 
 t^{(\infty)}_{k_p}) \right] 
 = \int_{\mathbb{R}^p} 
 F(y_1, \dots, y_p) \rho_p^{(\infty)}
 (y_1, \dots, y_p) dy_1 \dots dy_p,$$
 where 
 $$F(y_1, \dots, y_p) = f(y_1)\dots f(y_p).$$
 This then extends to all functions $F$ which are measurable, positive and continuous with compact support: 
 indeed there always exists an  $f$ which is continuous
 with compact support on $\mathbb{R}$ such that  $ F \leq G$, with
 $$G(y_1, \dots, y_p) = f(y_1)\dots f(y_p).$$
 Since we have inequalities for both functions
 $F$ et $G-F$ and an equality for their sum $G$, we in fact have an equality everywhere. 
 
We see that with a monotone class argument the previous equality then extends to functions $F$ which are measurable, bounded and with compact support. This shows the existence of a point process $E_{\infty}$ with the same correlation functions as those given in the statement of the Proposition, provided we do not exclude a priori point processes with multiple points. More precisely we take for $E_{\infty}$ the set of points $t_k^{(\infty)}$ of the support of the measure $\mathcal{M}_{\infty}$, taken with their multiplicities. 
Going back to our earlier computations, we see that for functions $f$ which are measurable, bounded with compact support from $\mathbb{R}$ to $\mathbb R$, and taking into account multiplicities, we have:
$$\mathbb{E} \left[ \left(\sum_{x \in E_{\infty}}
f(x) \right)^p \right] = M_{f,p}^{(\infty)}$$
and 
$$\mathbb{E} \left[ \exp \left( i \lambda \sum_{x \in E_{\infty}}
f(x) \right) \right] = \sum_{p \geq 0} 
\frac{(i \lambda)^p}{p!} M_{f,p}^{(\infty)}.$$
Consequently 
$$\mathbb{E} \left[ e^{i \lambda X_f^{(n)}}
\right] \underset{n \rightarrow \infty}{\longrightarrow} \mathbb{E} \left[ \exp \left( i \lambda \sum_{x \in E_{\infty}}
f(x) \right) \right],$$
which corresponds to the convergence in law stated in the Proposition. 

It only remains to show that $E_{\infty}$ does not have multiple points.  Indeed, if $E_{\infty}$ is the set of points
$(t_k^{(\infty)})_{k \in \mathbb{Z}}$, taken with multiplicities, then  
for any  mesurable bounded function $F$ with compact support from 
$\mathbb{R}^2$ in $\mathbb{R}$, 
$$\mathbb{E} \left( \sum_{k_1 \neq k_2} 
F(t_{k_1}^{(\infty)}, t_{k_2}^{(\infty)})
\right) = \int_{\mathbb{R}^2}
F(y_1, y_2) \rho_2^{(\infty)} (y_1,y_2) 
dy_1 dy_2.$$
Taking $F(y_1,y_2) = \mathds{1}_{y_1 = y_2}$ above yields 
$$\mathbb{E}\left[\operatorname{Card} \left\{( k_1,k_2) \in 
\mathbb{Z}^2, k_1 \neq k_2, 
t_{k_1}^{(\infty)} = t_{k_2}^{(\infty)}
\right\}\right] = 0,$$
which shows that $E_{\infty}$ does almost surely not have multiple points.

\end{document}